\documentclass[11pt]{amsart}
\usepackage{mathtools}
\usepackage{amsmath}
\usepackage{amssymb}
\usepackage{amsthm}
\usepackage{yhmath}
\usepackage{graphicx}
\usepackage{mathrsfs}
\usepackage{bbm}
\usepackage{xcolor}
\usepackage{tikz-cd}
\usepackage[colorlinks,citecolor=cyan,linkcolor=cyan]{hyperref}
\usepackage{cleveref}
\usepackage{verbatim}
\usepackage{enumitem}
\usepackage{scalerel}
\usepackage{wrapfig}
\usepackage{float}

\newtheorem{theorem}{Theorem}[section]

\newtheorem*{theorem*}{Theorem}
\newtheorem*{definition*}{Definition}
\newtheorem*{prop*}{Proposition}
\newtheorem*{cor*}{Corollary}
\newtheorem*{lemma*}{Lemma}

\newtheorem*{claim*}{Claim}

\newtheorem{lemma}[theorem]{Lemma}

\newtheorem{corollary}[theorem]{Corollary}

\newtheorem{cor}[theorem]{Corollary}

\newtheorem{proposition}[theorem]{Proposition}

\newtheorem{prop}[theorem]{Proposition}

\theoremstyle{definition}
\newtheorem{definition}[theorem]{Definition}

\theoremstyle{remark}
\newtheorem{remark}[theorem]{Remark}\newtheorem{rmk}[theorem]{Remark}

\newcommand{\para}[1]{\medskip\noindent\textbf{#1.}}

\numberwithin{equation}{section}

\DeclareMathOperator{\CR}{cr}

\newcommand{\cO}{\mathcal O}
\newcommand{\inverse}{^{-1}}

\newcommand{\Z}{\mathbb{Z}}
\newcommand{\Hd}{\mathbb{H}^d}
\newcommand{\Hplane}{\mathbb{H}^2}

\newcommand{\op}{\operatorname}

\newcommand{\cal}{\mathcal}

\newcommand{\SO}{\op{SO}}

\newcommand{\bH}{\mathbb H}

\renewcommand{\frak}{\mathfrak}

\newcommand{\be}{\begin{equation}}
	\newcommand{\ee}{\end{equation}}

\renewcommand{\L}{\mathcal L}

\renewcommand{\epsilon}{\varepsilon}
\newcommand{\ep}{\epsilon}

\newcommand{\R}{\mathbb{R}}

\newcommand{\PSL}{\op{PSL}}

\newcommand{\cT}{\mathcal{T}}

\newcommand{\olim}{\omega\text{-}\!\mathrm{lim}_\Z}

\newcommand{\Qm}{\cal{Q}}

\newcommand{\MF}{\cal{MF}}
\newcommand{\ML}{\cal{ML}}
\newcommand{\Sp}{\mathsf{Sp}}

\newcommand{\Horobetaplus}[1]{\mathcal{H}_{+}\left(#1\right)}
\newcommand{\Horobetaminus}[1]{\mathcal{H}_{-}\left(#1\right)}

\newcommand{\Gammazero}{\Gamma_0}
\newcommand{\Sigmazero}{\Sigma_0}
\newcommand{\pZ}[1]{p_{\Z}\left(#1\right)}

\newcommand{\GdmodGamma}{G_d / \Gamma}
\newcommand{\GdmodGammazero}{G_d / \Gammazero}
\newcommand{\GmodGamma}{G / \Gamma}
\newcommand{\GmodGammazero}{G / \Gammazero}

\begin{document}
	
	\graphicspath{ {} }
	
	\title[Minimizing laminations in regular covers]{Minimizing laminations in regular covers, horospherical orbit closures, and circle-valued Lipschitz maps}
	\author{James Farre, Or Landesberg and Yair Minsky}
	
	\begin{abstract}
		We expose a connection between distance minimizing laminations and horospherical orbit closures in $ \Z $-covers of compact hyperbolic manifolds. For surfaces, we provide novel constructions of $ \Z $-covers with prescribed geometric and dynamical properties, in which an explicit description of all horocycle orbit closures is given. We further show that even the slightest of perturbations to the hyperbolic metric on a $ \Z $-cover can lead to drastic topological changes to horocycle orbit closures.
	\end{abstract}
	\maketitle
	
	\section{Introduction}
\label{sec:intro}

In this paper we will study the dynamical behavior of horocyclic and geodesic flows in
geometrically infinite hyperbolic manifolds (mostly in dimension 2).  These two flows,
while geometrically related, exhibit dramatically different dynamical behaviors. Indeed,
over a finite area surface $ \Sigma $, the geodesic flow is ``chaotic'' and supports a
plethora of invariant measures and orbit closures while the horocycle flow is extremely
rigid, with all non-periodic orbits  dense \cite{Hedlund} and equidistributed in $
T^1\Sigma $ \cite{Furstenberg,Dani-Smillie}.  Something similar is true for the
geometrically finite case (which in dimension 2 just means finitely-generated fundamental
group), see e.g.~\cite{Eberlein,Dal'bo,Burger,Roblin,Schapira}.

We consider arguably the simplest and most symmetric geometrically infinite setting, that
of $ \Z $-covers of compact surfaces. Let
$ G = \PSL_2(\R) $ be the group of orientation preserving isometries of real hyperbolic $
2 $-space $ \mathbb{H}^2 $, let
$ \Gamma_0 < G $ be a torsion-free uniform
lattice and let $ \Gamma \lhd \Gamma_0 $ with $ \Gamma_0/\Gamma \cong \Z $, that is, $
\mathbb{H}^2/\Gamma $ is a $ \Z $-cover of the compact surface $ \mathbb{H}^2/\Gamma_0 $.

Let $A = \{a_t:t\in\R\}$ denote the diagonal subgroup generating the geodesic flow on
$\GmodGamma$, and
let $N$ denote the lower unipotent subgroup generating the (stable with respect to $A$) horocycle flow. 
We call a horocycle orbit closure $\overline{Nx}$  \emph{non-maximal} if it is not all of $ \GmodGamma $.
	\medskip
	
	\noindent In this setting we: 
	\begin{enumerate}[leftmargin=*]
		\item Study the structure of non-maximal horocycle orbit closures in $ \Z $-covers and expose their delicate dependence on a geometric optimization problem for associated circle-valued maps (\Cref{Thm intro: Q_omega is L^CR,Thm intro: structure of Nx bar,Thm intro:Nx bar in horoball}). 
		\item Describe novel constructions of $ \Z $-covers with prescribed geometric and dynamical properties (\Cref{Thm intro: tighten-lamination}); and, in doing so,
		\item Provide the first examples of $\Z$-covers with a full horocycle
                  orbit closure classification, including a description of orbit closures
                  that are neither minimal nor maximal (\Cref{weak-mixing-case}).  
            \item Demonstrate that the topological type of horocycle orbit closures varies discontinuously with the hyperbolic metric of a $\Z$-cover (\Cref{Theorem:non-rigidity}). 
	\end{enumerate}
	
	While the strongest results in this paper hold solely for $ \Z $-covers of compact hyperbolic surfaces, many of the techniques we develop are applicable in greater generality, both to higher-dimensional hyperbolic manifolds as well as maximal horospherical group actions on higher-rank homogeneous spaces.

\begin{remark}\label{rem:N-inv measures do not help}
	In contrast to the finite area setting, measure rigidity and equidistribution results for the horocycle flow over $ \Z $-covers have limited utility. Indeed, non-maximal horocycle orbit closures do not support any locally finite $ N $-invariant measures, since all such measures are $ AN $-quasi-invariant and hence have full support, see \cite{Sarig} as well as \cite{LL}.
\end{remark}
	
	At the heart of our analysis lies a connection to a seemingly unrelated geometric optimization problem of independent interest --- tight Lipschitz maps to the circle $ \R/\Z $.
	
	\subsection*{Tight circle-valued maps} 

 Given a compact hyperbolic $d$-manifold $\Sigmazero$ and a homotopically nontrivial  map $ f: \Sigma_0 \to \R/\Z $, we are interested in geometric properties of  maps realizing the minimum Lipschitz constant in the homotopy class of $f$.

    The homotopy class of $f$ is the same data as a cohomology class $\varphi \in H^1(\Sigma_0, \Z)$ recording the degree of the restriction to loops in $\Sigma_0$.
    Evidently, the minimum Lipschitz constant for circle valued maps representing $\varphi$ is bounded below by
	\[ \kappa = \sup_{\gamma} \frac{|\varphi(\gamma)|}{\ell(\gamma)}, \]
	where the $ \sup $ runs over all geodesic loops $\gamma$ and $ \ell(\gamma) $ is the length in $\Sigma_0$. 
    We call a map in $ [f] $  \emph{tight} if its Lipschitz constant is equal to $ \kappa $. 

    Given a Lipschitz map $f: \Sigmazero\to \R/\Z$, there is an upper semi-continuous function $ \ell_f: \Sigma_0\to \R_{\ge 0}$ measuring the local Lipschitz constant. The set $\ell_f^{-1}(\max \ell_f)$ is called the \emph{maximal stretch locus} of $f$. 

Daskalopoulos-Uhlenbeck \cite{DU:circle} and Gu\`eritaud-Kassel \cite{GK:stretch} studied tight maps, proving existence in each nontrivial homotopy class. They extended, in different ways, some of the ideas in Thurston's work on stretch maps between hyperbolic surfaces \cite{Thurston:stretch}: 
	
	\begin{theorem*}[\cite{DU:circle,GK:stretch}]
    
		There exists a tight map in every non-trivial homotopy class $ [f:\Sigma_0 \to \R/\Z] $, whose maximal stretch locus is a geodesic lamination.

        Moreover,\footnote{This was proved in \cite{GK:stretch}} 
        the intersection of the maximal stretch loci over all tight maps in the same homotopy class is a non-empty geodesic lamination $\mu_0\subset \Sigmazero$.
	\end{theorem*}

Our approach leads to an explicit description of the chain recurrent part of $\mu_0$ in dimension $d=2$, which is given in \Cref{Thm intro: Q_omega is L^CR} below.

	\subsection*{Quasi-minimizing points} The bridge between horocycle orbit closures and tight maps is given by the notion of a quasi-minimizing ray and a theorem of Eberlein, Dal'bo and Maucourant-Schapira. 

    Let $G_d = \SO^+(d,1)$ be the group of orientation preserving isometries of real hyperbolic $d$-space, $\Hd$. We may identify $G_d$ with $F\Hd$, the frame bundle of $\Hd$. As before, we denote by $A=\{a_t: t \in \R\}$ the one-parameter diagonal subgroup corresponding to the geodesic frame flow on $F\Hd$, and let $N$ be the stable horospherical subgroup (see \Cref{Subsection:Preliminaries - notations} for more details).
	
    Let $ \Sigma=\Hd/\Gamma $ be a $ \Z $-cover of a compact hyperbolic $d$-manifold $ \Sigma_0 $:
	
	\begin{definition}\label{def:quasi-minimizing}
		A point $ x \in \GdmodGamma \cong F\Sigma $ is called \emph{quasi-minimizing} if there exists a constant $ c \geq 0 $ for which
	\[ d_{\GdmodGamma}(a_t x,x) \geq t-c \quad \text{for all }t \geq 0, \]
	where $ a_tx $ is at distance $ t $ along the geodesic emanating from the frame $ x $.
	\end{definition}

	The above condition implies that the geodesic ray $ (a_tx)_{t\geq 0} $ escapes to infinity in $ \GdmodGamma $ at the fastest rate possible, up to an additive constant. Denote by $ \mathcal{Q} \subset \GdmodGamma $ the set of all quasi-minimizing points.
	
	The following theorem is of fundamental importance:
	\begin{theorem*}[\cite{Eberlein,Dal'bo,MS}]
		$ \overline{Nx}$ is non-maximal if and only if $ x $ is quasi-minimizing.
	\end{theorem*}

	\begin{remark}   
	This result in fact holds for any Zariski-dense discrete subgroup $ \Gamma < G_d $,
        with a suitable interpretation of non-maximal, see \Cref{subsec:horospherical orbit closures}. See also \cite{LO} for a generalization to higher-rank homogeneous spaces.
	\end{remark}
 
	An immediate corollary is that all non-maximal horospherical orbit closures are contained in $ \mathcal{Q} $. Hence analyzing the set $ \mathcal{Q} $ is a good first step to understanding non-maximal orbit closures. 
	
    The quotient map $\Gamma_0 \to \Gamma_0/\Gamma \cong \Z$ determines a homotopy class of circle maps, and we let $ \mu_0 $ be the canonical maximal stretch lamination defined above for tight maps in this homotopy class. We show the following:	
	\begin{theorem}\label{Thm intro: Q_omega is L^CR}
		In dimension $d=2$, the geodesic $ \omega $-limit set of $ \mathcal{Q} $ as projected onto $ \Sigma_0 $ is the chain recurrent\footnote{See \Cref{def:chain_recurrent}} part of  $ \mu_0 $.
	\end{theorem}
	
	In other words, for surfaces we have
	\[ \mu_0 = \{ q \in \Sigma_0 : \exists x \in \mathcal{Q} \text{ and }t_j \to \infty \text{ with } P(a_{t_j}x) \to q \}, \]
	where $ P: \GmodGamma \to \Sigma_0 $ is the natural projection from $ T^1\Sigma $. More generally, in higher dimensions we show that the geodesic $ \omega $-limit set of $ \mathcal{Q} $ is contained in the canonical maximal stretch lamination, see \Cref{Theorem:omega limit of Qm is a lamination in max stretch locus of u}.
	
	As a corollary we show: 
	\begin{cor}\label{Thm intro: Hausdorff dim=0}
        In dimension $d=2$, 
		the set of endpoints in $ S^1 $ of quasi-minimizing rays has Hausdorff dimension zero. Consequently, $\dim_{\mathrm{H}} \Qm = 2$.
	\end{cor}
	
	\begin{remark}\;
		\begin{enumerate}[leftmargin=*]
		    \item As follows from \Cref{rem:N-inv measures do not help}, it was known that this set of endpoints is null with respect to any $\Gamma$-conformal measure on the boundary of $\Hplane$ (including Lebesgue), since otherwise $\Qm$ would support a locally finite $N$-invariant measure.
            \item The cardinality of this set of endpoints can be either countable or uncountable, as follows from the examples provided in \Cref{Sec_construction of minimizing laminations}.
		\end{enumerate}
	\end{remark}
	\medskip
	
	Much of the work in this paper is to investigate the subtle relationship between the structure and dynamics of the geodesic
        lamination $ \mu_0 $ and the structure and topology of non-maximal horocycle
        orbit closures. We uncover a number of aspects of this; for example in certain cases it's
        possible to give a complete orbit closure classification, see Theorem
        \ref{weak-mixing-case} below.

    In dimension $d=2$, we have many powerful tools that allow us to refine our understanding of tight maps and their stretch laminations.
    For the rest of the introduction, we will focus our attention on the case of surfaces ($G=G_2$).

	\subsection*{Constructing tight circle maps with prescribed laminations}

Let $ S_0 $ be an orientable closed topological surface of genus $ \geq 2 $. Using either \cite{DU:circle} or \cite{GK:stretch} provides a tight circle map $S_0\to \R/\Z$ for any
choice of nonzero $\varphi\in H^1(S_0,\Z)$ and hyperbolic metric on $S_0$. We establish the
following complementary result, which allows us to specify the maximal-stretch lamination
instead of the hyperbolic metric:

\begin{theorem}\label{Thm intro: tighten-lamination}
  Let $\varphi\in H^1(S_0,\Z)$ and let $\lambda$ be the support of an oriented measured
  lamination in $S_0$. Suppose that $\varphi$ is Poincar\'e-dual to an oriented
  multicurve $\alpha$, such that $\alpha$ intersects $\lambda$ transversely with positive
  orientation and $\alpha\cup\lambda$ binds the surface (all complementary components
  are compact disks). Then there exists a hyperbolic metric $g$  on $S_0$ and a tight circle map
  $f:(S_0,g) \to \R/\Z$ whose homotopy class is $\varphi$ and whose maximal stretch
  lamination is equal to $\lambda$. 
\end{theorem}
In particular, any orientable measured lamination can occur as the stretch lamination of a suitable
tight map, and this is used in Theorem \ref{weak-mixing-case} below to describe a class
of $\Z$-covers for which the horocycle orbit closures can be classified. 

An interval exchange transformation can provide the data $(\varphi,\alpha,\lambda)$ for
use in Theorem \ref{Thm intro: tighten-lamination}. Starting with an interval exchange map $T:S^1\to S^1$, we suspend
it to obtain an annulus $[0,1]\times S^1$ with its two boundary circles identified by
$T$. This gives a surface $S$ with a foliation $F$ inherited from the foliation by
horizontal lines in the annulus. The lamination $\lambda$ is then the ``straightening'' of
$F$ with respect to any hyperbolic structure on $S$, and a circle $\alpha = \{t\}\times S^1$
represents the cohomology class $\varphi$ associated to the map $S\to S^1$ obtained from
projection to the $[0,1]$ factor in the annulus (see Figure \ref{IET-construction}).

        \begin{figure}[ht]
		\includegraphics[scale=0.42]{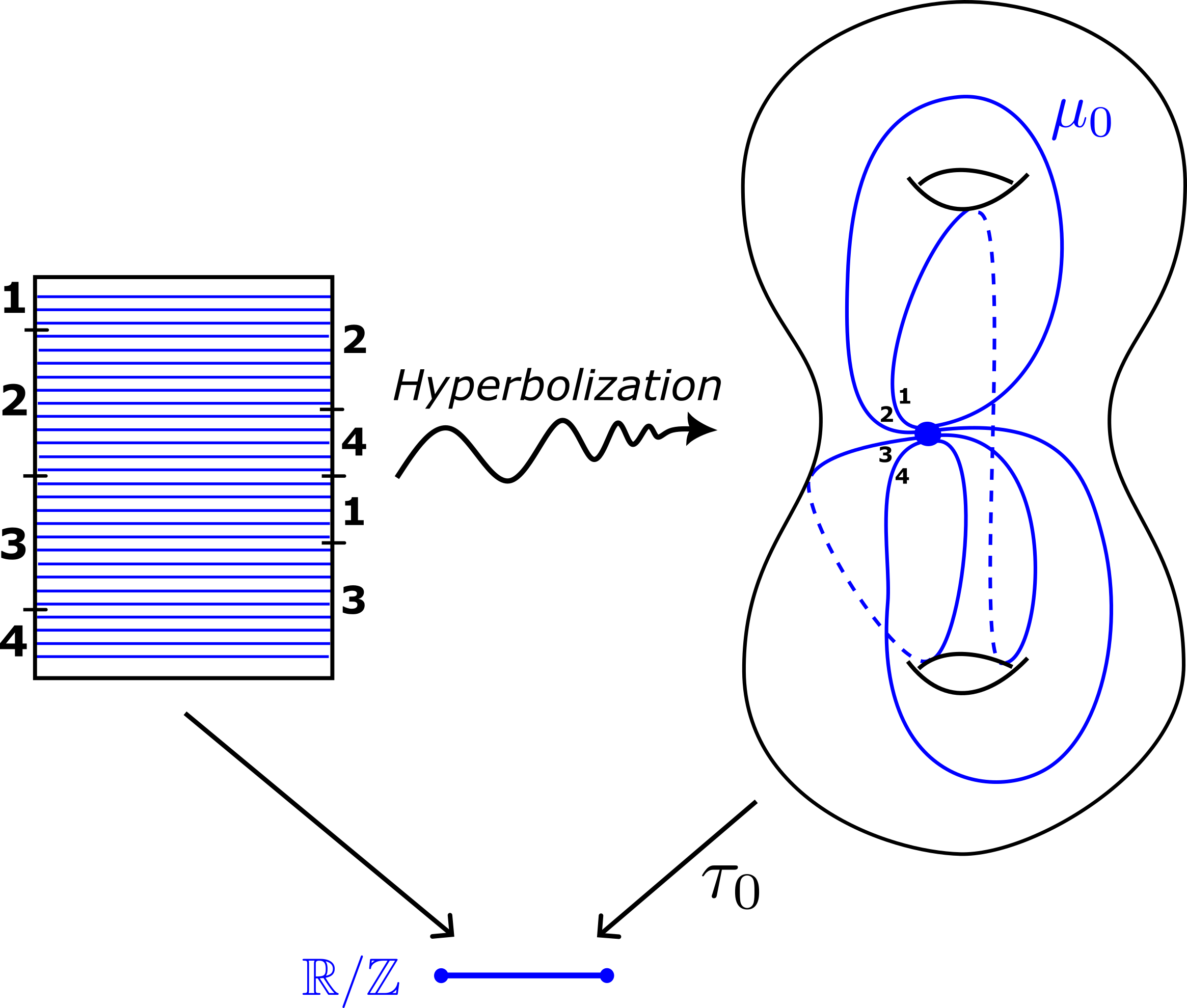}
                \vspace{0.2cm}
                \caption{Surface constructed from an IET with permutation $ \sigma=(1342) $. The corresponding distance minimizing lamination $ \mu_0 $ lies
                  within the depicted train track.}
                \label{IET-construction}
        \end{figure}

Borrowing from ideas of Mirzakhani \cite{Mirzakhani:EQ}, we use recent work of
Calderon-Farre \cite{CF:SHSH} to ``hyperbolize'' this construction.  Namely, we obtain a
\emph{different} hyperbolic structure $\Sigmazero$,  a measured geodesic lamination
$\mu_0$ which is measure-equivalent to $\lambda$, 
and a tight map $\tau_0: \Sigmazero\to \R/\Z$ taking the leaves of $\mu_0$ locally
isometrically to the circle. Moreover, the vertical foliation of the annulus by
$\alpha$-parallel curves is converted in $\Sigmazero$ to the \emph{orthogeodesic
foliation} of $\mu_0$, whose leaves are collapsed by $\tau_0$. See Figure
\ref{hyperbolization-foliation} for a local picture of this. 

\begin{figure}[ht]
  \includegraphics[width=4in]{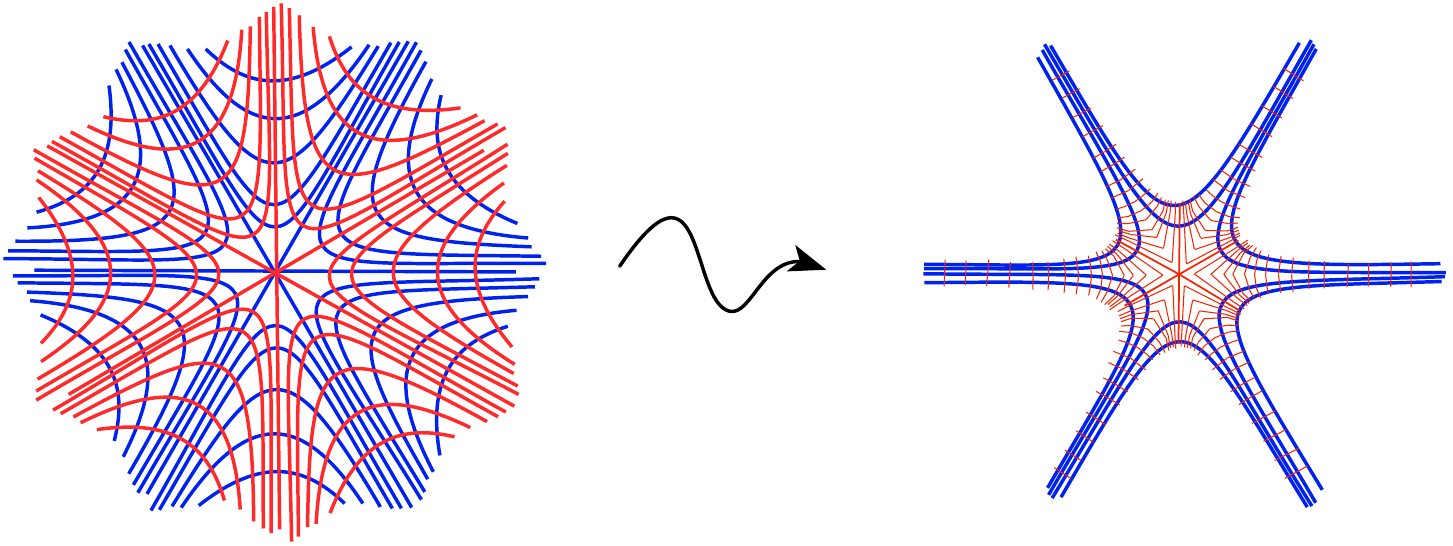}
  \caption{Horizontal and vertical singular foliations correspond to
                  geodesic lamination and orthogeodesic singular foliation, respectively.}
                 \label{hyperbolization-foliation}
\end{figure}

    \subsection*{Uniform Busemann-type functions}	
    Fix  $ \tau_0: \Sigma_0 \to \R/\Z $ a tight Lipschitz map. 
    Lifting $ \tau_0 $ to $ \Sigma $ and rescaling by $ \kappa^{-1} $ one obtains a $ 1 $-Lipschitz $ \kappa^{-1}\Z $-equivariant map $ \tau: \Sigma \to \R $.
	
	Denoting $ \mu $ the lift of $ \mu_0 $ to $ \Sigma $, we see that $ \mu $ is contained in the $ 1 $-Lipschitz locus of the map $ \tau $. In particular, all geodesic lines in $ \mu $ are isometrically embedded copies of $ \R $ in $ \Sigma $. The geodesic lamination $ \mu \subset \Sigma $ contains the ``fastest routes'' traversing along the $ \Z $-cover and the map $ \tau $ indicates how to ``collapse'' $ \Sigma $ onto these routes.

 The surface $ \Sigma $ has two infinite ends, one corresponding to unbounded positive values of $ \tau $ and the other to the negative. Abusing notation, we write $\tau(x)$ to mean $\tau(p)$, where $p\in \Sigma$ is the basepoint of $x\in \GmodGamma$.  We define a sort of ``uniform Busemann function" $ \beta_+: \GmodGamma \to [-\infty,\infty) $ with respect to the positive end, by
	\[ \beta_+(x) = \lim_{t \to\infty} \tau(a_tx)-t. \]
	This function is $ N $-invariant and upper semicontinuous. Furthermore, a point $ x \in \GmodGamma $ is quasi-minimizing and facing the positive end if and only if $ \beta_+(x) > -\infty $. The closed $ N $-invariant set
	\[ \mathcal{H}_+(x) :=\beta_+^{-1}\left([\beta_+(x),\infty)\right) \]
	can be thought of as a uniform horoball based at the positive end and passing through $ x $. We thus have the following:

    \begin{theorem}\label{Thm intro:Nx bar in horoball}
		Let $ \Sigma $ be any $ \Z $-cover of a compact hyperbolic surface together with a tight Lipschitz map $ \tau:\Sigma \to \R $. All quasi-minimizing points $ x \in T^1\Sigma $ facing the positive end of $ \Sigma $ satisfy
		\[ \overline{Nx} \subseteq \mathcal{H}_+(x). \]
	\end{theorem}
    \noindent An analogous statement holds for the negative end of $ \Sigma $. In fact, this theorem holds in arbitrary dimension, see \Cref{cor:Nx contained in beta+(x) to infinity}.

    \subsection*{Structure of horocycle orbit closures} All horocycle orbit closures satisfy the following two structural properties:
	\begin{theorem}\label{Thm intro: structure of Nx bar}
		Let $ \Sigma $ be any $ \Z $-cover of a compact hyperbolic surface and let $x\in T^1\Sigma$ be any quasi-minimizing point. Then
		\begin{enumerate}
            \item There exists a non-trivial, non-discrete closed subsemigroup $ \Delta_x $ of $ A_{\geq 0} = \{a_t : t \geq 0\} $ under which $ \overline{Nx} $ is strictly sub-invariant, that is, 
            \[ a \overline{Nx}\subsetneq \overline{Nx} \quad \text{for all }  a \in \Delta_x \smallsetminus \{e\}.\]
            \item $\overline{Nx}$ intersects all quasi-minimizing rays escaping through the same end as $x$. That is, if $y \in T^1\Sigma$ is quasi-minimizing and facing the same end as $x$ then $a_t y \in \overline{Nx}$ for some $t\ge0$.
        \end{enumerate}
	\end{theorem}

    \begin{remark}
        Theorem \ref{Thm intro: structure of Nx bar} holds for $ \Z $-covers of higher dimensional compact hyperbolic manifolds as well, with suitable adjustments addressing the group of frame rotations $M$ commuting with $A$ in $\SO^+(d,1)$, see \Cref{Thm:Recurrence semigroup in arbitrary d,Cor:Horospherical orbit closure intersect all Q of the same end}. In the case of $ d=2 $, partial results, applicable also to this setting, were obtained in \cite{Bellis}.
    \end{remark}

  While the full nature of the semigroup $ \Delta_x $ is a bit mysterious, we develop both geometric and algebraic tools to study it, most of which generalize to other homogeneous spaces; see \Cref{sec:semigroup}.   We show there are examples where $e \in \Delta_x$ is an isolated point and examples where $\Delta_x=A_{\geq 0}$ (\Cref{Theorem:semigroup is everything} and \Cref{cor:isolated leaf implies isolated e in Delta}).

     An immediate corollary of Theorems \ref{Thm intro:Nx bar in horoball} and \ref{Thm intro: structure of Nx bar}(1) is that: 
	\begin{cor}
		Every $ \Z $-cover of a compact surface contains uncountably many distinct non-maximal horocycle orbit closures, all of which are not closed $ N $-orbits. Moreover, on such surfaces no horocycle orbit closure is minimal.
	\end{cor}

    We construct a particular class of surfaces having favorable dynamical properties under which a full orbit-closure classification is given:
	\begin{theorem}\label{weak-mixing-case}
		If $ \Sigma $ is a $ \Z $-cover surface constructed by Theorem \ref{Thm intro: tighten-lamination} from a weakly-mixing and minimal IET then all non-maximal horocycle orbit closures in $ \Sigma $ are uniform horoballs. That is, for all $ x \in T^1\Sigma $, either $ Nx $ is dense in $T^1\Sigma$ or 
		\[ \overline{Nx}=\mathcal{H}_\pm(x). \]
	\end{theorem}

    It is worth remarking that in light of Avila-Forni \cite{Avila-Forni}, our construction in Theorem \ref{Thm intro: tighten-lamination} ensures an abundance of such examples.

    \subsection*{Non-rigidity of orbit closures}    
	It is intuitively clear that changing the geometry of $ \Sigma_0 $ could change the maximal stretch locus of $ \tau_0 $ and hence $ \mu_0 $. In light of the orbit closure rigidity in the finite volume and geometrically finite settings, and in light of the measure rigidity in the $ \Z $-cover setting, it was quite surprising to discover that slight changes to the geometry could dramatically change the topology of non-maximal horocycle orbit closures. To that effect we show:
	\begin{theorem}\label{Theorem:non-rigidity}
		Let $ S $ be any $ \Z $-cover of an orientable closed surface $ S_0 $ of genus $ \geq 2 $. There exist two $\Z$-invariant hyperbolic metrics on $ S $ corresponding to discrete groups $ \Gamma_1 $ and $ \Gamma_2 $ for which any two non-maximal orbit closures
		\[ \overline{N x_1} \subsetneq \GmodGamma_1 \quad \text{and}\quad \overline{N x_2} \subsetneq \GmodGamma_2 \]
		are non-homeomorphic. Moreover, these two metrics may be taken to be arbitrarily small deformations of one another.
	\end{theorem}
	
	We remark that the topological obstruction described in this theorem does not arise from the fiber, that is, the orbit closures $ \overline{N x_i} $ are non-homeomorphic even after projecting onto the respective surfaces $ \mathbb{H}^2/\Gamma_i $.

    \subsection{Questions and context}
    Many questions remain open.  The techniques we develop are applicable to studying non-maximal horospherical orbit closures in other geometrically infinite hyperbolic manifolds, but the input to our analysis requires some dynamical, geometrical, and topological information about the quasi-minimizing locus.

    There is a natural class of examples in dimension $3$ consisting of $\Z$-covers of closed manifolds that fiber over the circle.  These examples are geometrically infinite but have finitely generated fundamental group.
    Although they have been well studied, e.g., \cite{Thurston:bulletin,Thurston:fibered,CT:peano,ELTI,ELTII}, it seems quite difficult to identify $\mathcal Q_\omega$  in them.

A natural lamination to consider in such a manifold is the geodesic tightening of the orbits of its pseudo-Anosov suspension flow. One of the main results claimed in \cite{LM} seems to imply that this lamination is in fact $\mathcal Q_\omega$; unfortunately, there are errors in the proof that are not easily fixable.  Moreover, Cameron Rudd has shown us an interesting construction that indicates there are examples where $\mathcal Q_\omega$ is a finite union of closed geodesics. 

Even for $\Z$-covers in dimension 2 we do not yet have a complete  classification of horocycle orbit closures, and our results so far suggest that the possibilities are rich. 

It would also be interesting to understand quasi-periodic surfaces that are not covers of compact surfaces, as well as regular covers with other deck groups.
	
	\subsection{Organization of paper} In section 2 we present our notation and a short background on horospherical orbit closures and geodesic laminations. Section 3 addresses tight Lipschitz maps and their maximal stretch loci, containing a proof of Theorem \ref{Thm intro: Q_omega is L^CR}. A proof of \Cref{Thm intro: Hausdorff dim=0} is given in section 4. In section 5 we give our construction of $ \Z $-covers from IETs and prove \Cref{Thm intro: tighten-lamination}. Section 6 discusses ``uniform Busemann functions'' and their connection to horospherical orbit closures. In sections 7 and 8 we study the structure of horospherical orbit closures, concluding in particular the content of \Cref{Thm intro: structure of Nx bar}. In the final section, 9, we fully describe the horocycle orbit closures in $ \Z $-covers constructed by weakly-mixing and minimal IETs, showing \Cref{Thm intro:Nx bar in horoball} and concluding \Cref{Theorem:non-rigidity}.

    \setcounter{tocdepth}{1}
	\tableofcontents 
 
	\subsection{Acknowledgments} The first named author would like to thank Gerhard Knieper for bringing \cite{DU:circle} to his attention. The second named author would like to thank Subhadip Dey, Giuseppe Martone, Ilya Khayutin and Elon Lindenstrauss for helpful conversations. The third named author would like to thank the first two authors for a very enjoyable collaboration. Special thanks to Ido Grayevsky and Hee Oh for useful comments on an earlier version of this manuscript and to Cameron Rudd for  an interesting construction in dimension $3$.

 	\section{Preliminaries}
 	
	\subsection{Setting and Notation}\label{Subsection:Preliminaries - notations} We fix the following notations throughout this paper. Set $ G_d = \SO^+(d,1)$ the group of orientation preserving isometries of real hyperbolic $ d $-space, $ \Hd $. In the case of $ d=2 $ we will omit the subscript and denote $ G=G_2 $. Equip $ G_d $ with a right $ G_d $-invariant Riemannian metric.
	
	The group $ G_d $ acts freely and transitively on $ F\Hd $, the frame bundle of $ \Hd $, and hence we can identify $ G_d \cong F\Hd $. Consider the following subgroups with their corresponding left-multiplication action on $ G_d $:
	\begin{itemize}[leftmargin=*]
		\item  $ K \cong \SO(d) $ --- stabilizer subgroup of a point $ o \in \Hd $, inducing $ K \backslash G_d \cong \Hd $.
		\item $ A=\{a_t : t \in \R \} $ --- one-parameter subgroup of diagonal elements corresponding to the geodesic frame flow on $ F\Hd $.
		\item $ M=Z_K(A) \cong \SO(d-1) $ --- the centralizer of $ A $ in $ K $, corresponding to rotations of frames in $ F\Hd $ about the geodesic flow direction. We may hence identify $ M \backslash G_d $ with $ T^1\Hd $, the unit tangent bundle of $ \Hd $.
		\item $ N=\{ n \in G_d : a_t n a_{-t} \to e \text{ as } t \to +\infty \} \cong \R^{d-1} $ --- the contracted horospherical subgroup corresponding to a flow along the stable foliation for the geodesic flow.
		\item $ U=\{ u \in G_d : a_{-t} n a_{t} \to e \text{ as } t \to +\infty \} \cong \R^{d-1} $ --- the (opposite) expanded horospherical subgroup.
	\end{itemize}
	We may assume the right-$ G_d $-invariant metric is also left-$ K $-invariant. Given a discrete subgroup $ \Gamma' < G_d $, we denote by $ d_{\GdmodGamma'} $ the induced metric on $ \GdmodGamma' $ (we will typically abbreviate the metric as just $d$).
 
	\smallskip
	
	Let $ \Gammazero < G_d $ be a uniform torsion-free lattice and let $ \Gamma \lhd \Gammazero $ be a normal subgroup with $ \Gammazero / \Gamma \cong \Z $. We denote $ \Sigma_0 = K \backslash \GdmodGamma_0 $, a compact hyperbolic $ d $-manifold and $ \Sigma = K \backslash G_d / \Gamma $ its respective $ \Z $-cover. Their unit tangent bundles are
	\[ T^1\Sigma = M \backslash \GdmodGamma \quad \text{and}\quad T^1\Sigma_0 = M \backslash \GdmodGamma_0 \]

	We have the following commuting diagram of projections:
	
	\vspace{0.1cm}
	\noindent \begin{minipage}{.58\textwidth}
	\begin{equation}
		\begin{tikzcd}
			\GdmodGamma \arrow{r}{p_M} \arrow[swap]{d}{p_\Z} \arrow[bend left]{rr}{p_K} & T^1 \Sigma \arrow{r}{p_K} \arrow[swap]{d}{p_\Z} & \Sigma \arrow{d}{p_\Z} \\%
			\GdmodGammazero \arrow{r}{p_M} \arrow[bend right]{rr}{p_K}& T^1 \Sigma_0 \arrow{r}{p_K} & \Sigmazero
		\end{tikzcd}
	\end{equation}
\end{minipage}
\hfill
\begin{minipage}{.4\textwidth}
	\centering
	\includegraphics[width=0.86\textwidth]{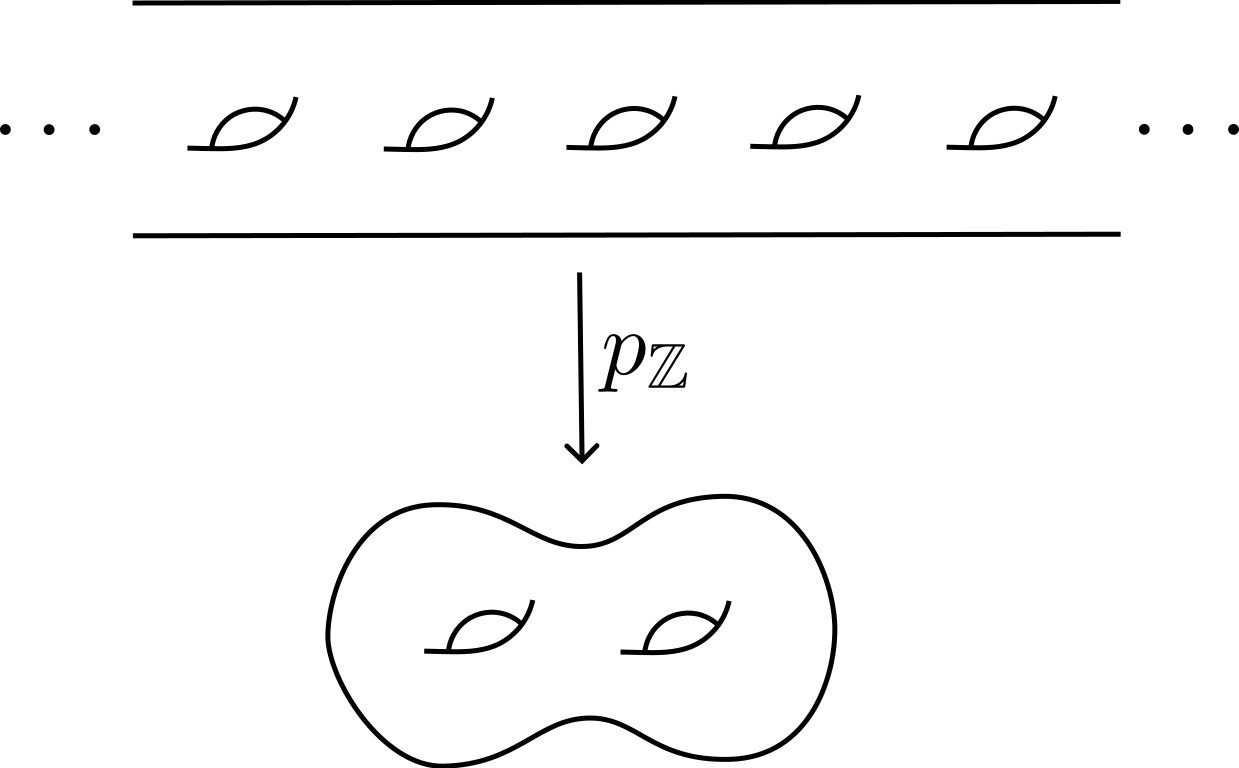}
\end{minipage}
\vspace{0.26cm}

	\noindent where $ \pZ{\cdot\;\Gamma}=\cdot\;\Gamma_0 $ are the cover maps, and where $ p_M(x)=Mx $ and $ p_K(x)=Kx $ are the bundle projections.

	\subsection{Horospherical orbit closures}\label{subsec:horospherical orbit closures}
	As described in the introduction, whenever $ \Gamma' < G_d $ is a lattice the horospherical flow on $ \GdmodGamma' $ is extremely rigid. In this setting all horospheres are either dense in $ \GdmodGamma' $ or periodic (bounding a cusp). This extreme rigidity is understood today as part of a much broader phenomenon of unipotent group action rigidity on finite volume homogeneous spaces, as follows from the works of Hedlund, Furstenberg, Veech, Dani, Margulis, and ultimately Ratner \cite{Ratner}.  
	
	Consider $ \Gamma' < G_d $ a general non-elementary discrete subgroup, not necessarily a lattice. Let $ \Lambda_{\Gamma'} \subseteq S^{d-1}=\partial\Hd $ be the limit set of $ \Gamma' $, that is, the unique closed $ \Gamma' $-minimal subset of $ S^{d-1} $. Consider the subspace
	\[ \mathcal{E}_{\Gamma'} = \{ g\Gamma' : g^+ \in \Lambda_{\Gamma'} \} \subset \GdmodGamma', \]
	where $ g^+ \in S^{d-1} $ is the forward endpoint of the geodesic in $ F\mathbb{H}^d $ emanating from the frame $ g \in G_d $. Algebraically, $ S^{d-1} $ is identified with the Furstenberg boundary of $ G_d $, i.e.~$ S^{d-1}\cong P \backslash G_d $ where $ P=MAN $, in which case $ g^+=Pg $. The set $ \mathcal{E}_{\Gamma'} $ is the non-wandering set for the $ N $-action, see e.g.~\cite{DalBo_Book}.
	
	\begin{definition}
		A limit point $ \xi \in S^{d-1} $ is called horospherical w.r.t.~$ \Gamma' $ if $ o\Gamma' $  intersects all horoballs based at $ \xi $ for some (and any) point $ o \in \Hd $.
	\end{definition}
	
	Recall the notion of a quasi-minimizing point from \cref{def:quasi-minimizing}. We have the following:
	\begin{lemma}\label{lemma:nonhoro-quasimini}
		The point $ g\Gamma' $ is quasi-minimizing if and only if $ g^+ $ is non horospherical.
	\end{lemma}
	
	\begin{proof}
		For all $ C\geq 0 $ the family of balls $ B(p_K(a_tg),t-C) $ in $ \Hd $ of radius $ t-C $ around $ p_K(a_tg) $ converges to a horoball tangent to $ g^+ $ as $ t \to \infty $. This horoball is disjoint from $ \Gamma'.p_K(g) $ if and only if $ g $ is quasi-minimizing with corresponding constant $ C $.
	\end{proof}
	
	A limit point $ g^+ $ is called conical if the geodesic ray $ (a_t g \Gamma')_{t \geq 0} $ emanating from $ g \Gamma' $ recurs infinitely often to a compact set in $ \GdmodGamma' $. The lemma above immediately implies that all conical limit points are horospherical (with the converse clearly false in general). In dimensions $ d=2,3 $ all parabolic fixed points are non-horospherical. Surprisingly, this is not true for $ d \geq 4 $ as demonstrated in \cite{Apanasov,Starkov}.
	\medskip
	
	The following theorem gives a necessary and sufficient condition for denseness of a horosphere in $ \mathcal{E}_{\Gamma'} $:
	\begin{theorem*}[\cite{Eberlein},\cite{Dal'bo},{\cite[Corollary 3.2]{MS}}]
		If $ \Gamma' < G_d $ is Zariski-dense then $ \overline{Ng\Gamma'}=\mathcal{E}_{\Gamma'} $ if and only if $ g^+ $ is horospherical.
	\end{theorem*}
	
	\noindent As this result is a fundamental ingredient in our analysis, we provide a short sketch of its proof. For technical simplicity we will restrict our discussion to the unit tangent bundle $ M \backslash \GdmodGamma' $. We also refer the reader to \cite[Theorem 3.1]{DalBo_Book} for a different flavored elementary proof.
	\medskip
	
	\noindent\emph{Sketch of proof.}\;
		Assume $ g^+ $ is horospherical, therefore $ p_K(g)\Gamma' $ enters arbitrarily deep inside the horoball in $ \Hd $ based at $ g^+ $ and passing through $ o=p_K(g) $. Equivalently, there is a sequence of $ \Gamma' $-copies $ H_j=p_K(Ng)\gamma_j $ of the horosphere emanating from $ g $ whose distance from $ p_K(g) $ tends to infinity. This implies that $ H_j $ converges uniformly to $ \partial \Hd $, e.g.~in the ball model of hyperbolic $ d $-space.
		
		\begin{wrapfigure}{r}{0.3\textwidth}
			\vspace{-0.7cm}
			\begin{center}
				\includegraphics[width=0.26\textwidth]{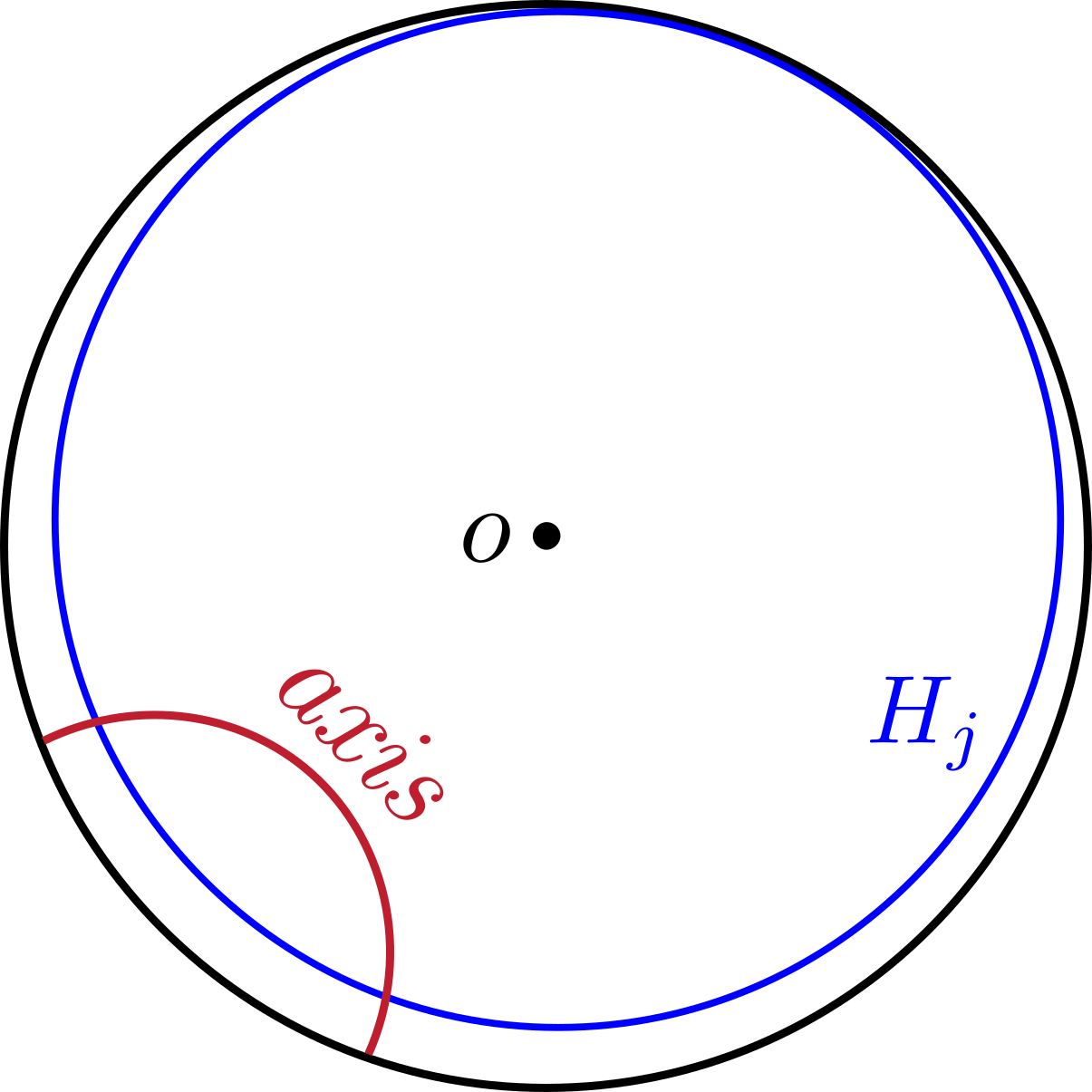}
			\end{center}
		\end{wrapfigure}
		
		The axis of any loxodromic element in $ \Gamma' $ intersects $ \partial \Hd $ at a right angle, implying the horospheres $ H_j $ tend to do the same. All such axes are compact in $ \Hd / \Gamma' $ hence the sequence of intersection points of any such axis with the horosphere $ p_K(Ng\Gamma') $ has an accumulation point with angle tending to $ \pi/2 $. In other words, $ \overline{Ng\Gamma'} $ contains frames on any closed geodesic in $ \GdmodGamma' $\footnote{This argument goes back to Hedlund {\cite[Theorem 2.3]{Hedlund}}.}.
		
		The length spectrum of any Zariski-dense discrete group is ``non-arithmetic'' \cite{Dal'bo_Spectrum,Kim}, that is, the group generated by the set of lengths of all closed geodesics in $ \GdmodGamma' $ is dense in $ \R $. Consider some closed geodesic $ L $ of length $ l_1 $ in $ \GdmodGamma' $, and let $ \varepsilon > 0 $. By non-arithmeticity, there exists a closed geodesic in $ \GdmodGamma' $ of length $ l_2 $ satisfying that $ \langle l_1,l_2 \rangle $ has covolume$ \leq \varepsilon $ in $ \R $. Let $ y \in \overline{Ng\Gamma'} $ be some point on this second closed geodesic of length $ \tau_2 $.
		
		\begin{wrapfigure}{r}{0.4\textwidth}
			\begin{center}
				\includegraphics[width=0.4\textwidth]{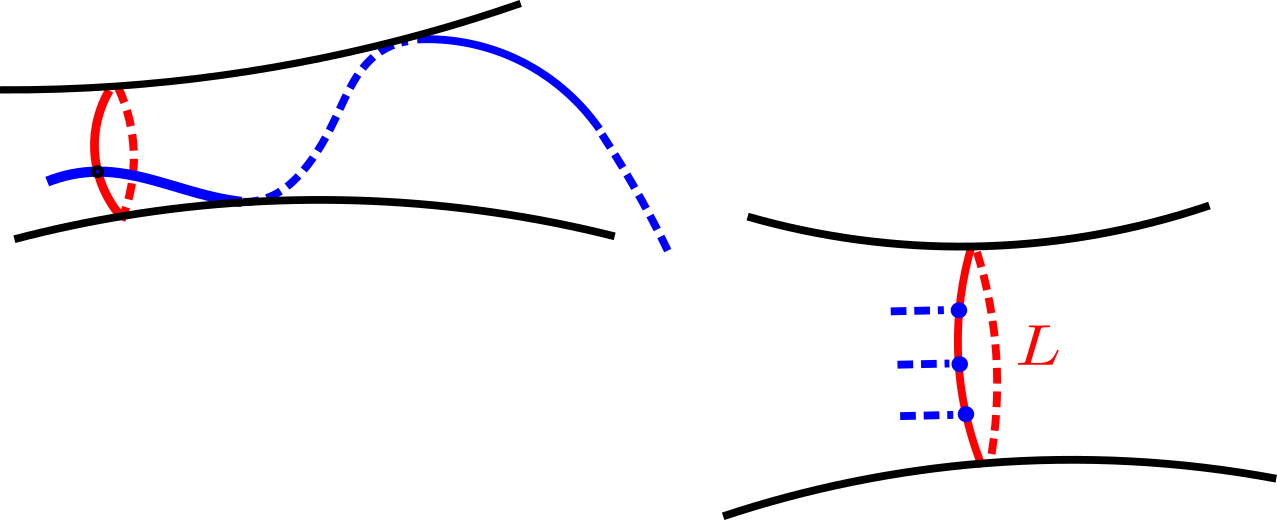}
			\end{center}
		\end{wrapfigure}
		
		The limit point $ y^+ $ is clearly conical and hence horospherical, implying that there exists a point $ z \in \overline{Ny} \subseteq \overline{Ng\Gamma'} $ contained in $ L $. The $ N $-orbit closure of $ y $ is $ a_{\tau_2} $-invariant up to a rotation in $ M $, therefore the set of intersection points of $ \overline{Ny} $ with $ L $ is $ \varepsilon $-dense in $ L $ (again up to rotation). As $ L $ and $ \varepsilon $ were arbitrary, we conclude that $ \overline{MNg\Gamma'} $ contains all closed geodesics and hence $ \mathcal{E}_{\Gamma'} \subset \overline{MNg\Gamma'} $.
		We leave the remaining details to the reader.\hfill $\square$

	\begin{remark}\;
		\begin{enumerate}
			\item To show denseness in the frame bundle one uses a stronger form of non-arithmeticity, taking into account the $ M $ components of the loxodromic elements in $ \Gamma' $, see \cite{Guivarch-Raugi} and \cite{MS}.
			\item For a higher rank generalization of the notion of a horospherical limit point and the above theorem regarding orbit closures, see \cite{LO}.
		\end{enumerate}
	\end{remark}
	
	A group $ \Gamma' $ is called geometrically finite if all of its limit points in $ \Lambda_{\Gamma'} $ are either conical (and hence horospherical) or bounded parabolic, see e.g.~\cite[Theorem 12.4.5]{Ratcliffe} or \cite{Bowditch} for other quivalent definitions and a more complete discussion. In such a case all horospheres are either dense in $ \mathcal{E}_{\Gamma'} $ or closed.
	
	It is worth noting that a discrete group $ \Gamma' < G_d $ always has non-horospherical limit points unless $ \Gamma' $ is convex cocompact (or cocompact). Indeed, if $ \mathcal{D}_z(\Gamma') $ denotes the Dirichlet domain of $ \Gamma' $ in $ \Hd $ based at $ z $, then the set $ \overline{\mathcal{D}_z(\Gamma')} \cap \Lambda_{\Gamma'} $ consists of non-horospherical limit points, whenever non-empty (this set is empty only when $ \Gamma' $ is convex cocompact), see e.g.~\cite[Corollary 4.10]{DalBo_Book}.
	\medskip
	
	In this paper we consider normal subgroups of lattices $ \Gamma \lhd \Gamma_0 $, which satisfy $ \Lambda_\Gamma=\Lambda_{\Gamma_0}=S^{d-1} $ (see e.g.~\cite[Theorem 12.2.14]{Ratcliffe}) and consequently $ \mathcal{E}_\Gamma = \GdmodGamma $. Therefore the theorem above may be applied, as done in the introduction, to show that a point has a non-dense horospherical orbit in $ \GdmodGamma $ if and only if it is quasi-minimizing.
	
	We note that thus far, apart from $ \mathcal{E}_\Gamma $ and closed $ MN $-orbits, no other horospherical orbit closures have been fully described in the literature, see e.g.~\cite{Bellis,Matsumoto,Kulikov,Dalbo-Starkov,Apanasov,Starkov,CM,Ledrappier}\footnote{Note that Prop.~3 in \cite{Ledrappier,Ledrappier_Err} is false, see \cite{Bellis} and \S\ref{Sec: Weakly mixing} for counterexamples.} for relevant constructions and partial results in this direction.
	
    \subsection{Bruhat decomposition for geometers}\label{subsec:bruhat}

     Let $g, h\in G_d$ be two points which we think of as frames in $\bH^d$. 
     \Cref{figure_bruhat} indicates a way of moving from $g$ to $h$ using horospherical and diagonal
     flows:

     \begin{figure}[ht]
       \includegraphics{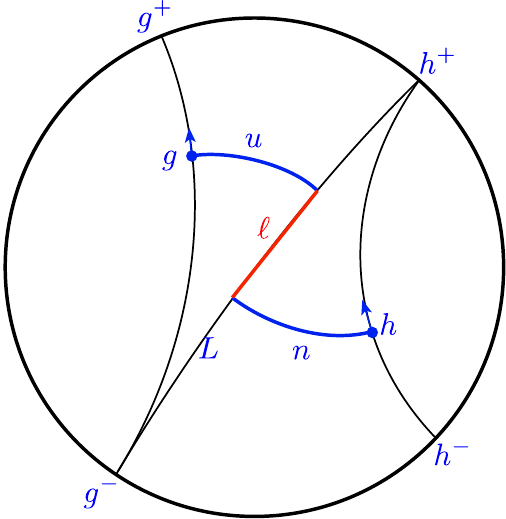}
       \caption{We slide $g$ to $h$ by applying an unstable horospherical element $u\in U$, a
         slide $\ell\in MA$ along $L$, and a stable horospherical element $n\in N$.}
       \label{figure_bruhat}
     \end{figure}

     Recall that $N$ preserves the forward-endpoint $g^+$ of a frame $g$, and $U$
     preserves the backward endpoint $g^-$.
     
     Supposing that $ g^-\ne h^+$,  let $L$ denote the directed geodesic
     joining $g^-$ to $h^+$. There is a 
     unique element $u\in U$ that takes the $A$-orbit of $g$
     to $L$, and then a unique element $n$ taking $L$ to the $A$-orbit of $h$. The
     discrepancy between $u g$ and $n^{-1} h$ is given by an element $\ell$ of $MA$, describing
     translation along and rotation around $L$. This gives us a decomposition
     $$ h = n \ell u g$$
     or in other words we express $hg^{-1}$ in the product $N M A\,U$.

     To deal with the case $g^- = h^+$, we fix  $\omega$
     an involution that reverses the first vector of the baseframe, or equivalently
     satisfies $a_{-t} = \omega a_t \omega$ (a representative of the generator of the Weyl group). Thus $(\omega g)^+ = g^-=h^+$, which means
     $h$ can be written as $n\ell \omega g$, where $n\ell\in NMA$.  What we have seen is
     how to write $hg^{-1}$, or any element of $G_d$, in one of two forms. This gives the
     {\em Bruhat decomposition}\footnote{The classical form of the Bruhat decomposition is $ G_d=P\omega P \cup P $ with $ P=MAN $. Equation \eqref{eq:Bruhat decomposition of G_d} is given by multiplication by $ \omega $.},
     	\begin{equation}\label{eq:Bruhat decomposition of G_d}
	G_d = NMA\,U \cup NMA\omega.
	\end{equation}
	We further record the fact that the set $ NMA\,U \subset G_d $ is open, and Zariski-open. The multiplication map $ N \times MA \times U \to NMA\,U $ is a diffeomorphism (see e.g.~\cite[Lemma 6.44]{Knapp}) and an isomorphism of varieties (see e.g.~\cite[Lemma 8.3.6(ii) and its proof]{Springer}). 
      
        \medskip

        We will use this decomposition to study $N$ orbit closures using the following
        observations: Suppose we are in a situation where $hg^{-1}$ lies in some compact set
        in $NMA\,U$, and hence $n,\ell$ and $u$ are controlled as well. If we push back
        via geodesic flow to $a_{-T}g$ and $a_{-T}h$ we can write
        $$
        a_{-T} h =  (a_{-T}na_T)\ \ell \ (a_{-T} u a_T ) \ (a_{-T} g).
        $$
        Now since $u$ is in the unstable horospherical group, $a_{-T} u a_T \approx e$
        for large $T$, while $\hat  n  = a_{-T} n a_T$ is quite large. This is illustrated
        in \Cref{fig:bruhat-trick}, where the stable horocycle is expanded by $a_{-T}$ and the unstable
        horocycle is contracted. Rewriting slightly,
        $$
        \hat n^{-1} (a_{-T}h) =  \ell u_T (a_{-T} g)
        $$
        where $u_T \approx e$. In other words, the horospherical orbit of $a_{-T} h$ comes
        very close to $\ell a_{-T}g$.

     \begin{figure}[ht]
       \includegraphics{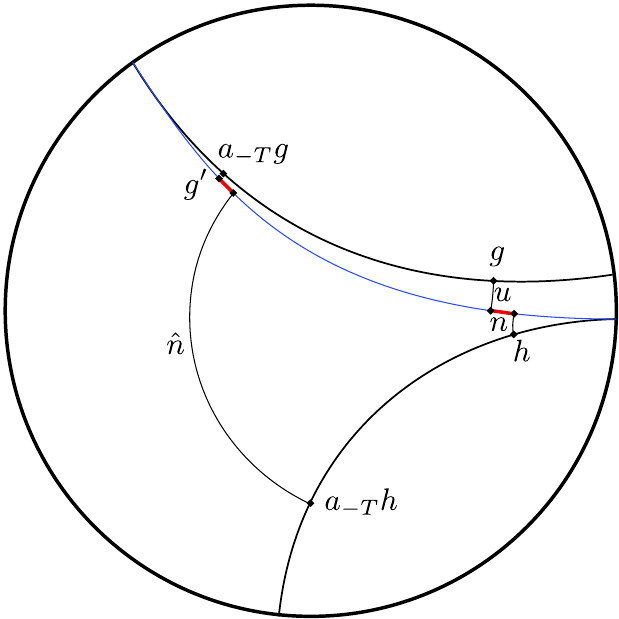}
       \caption{Sliding back the Bruhat decomposition by $a_{-T}$. The red segments
         designate $\ell\in MA$. The stable horocycle ($n$) is expanded and the unstable
         contracted. $g'$ is $\ell u_T a_{-T}g$, where $u_T$ is the contracted unstable
         horocycle.}
       \label{fig:bruhat-trick}
     \end{figure}

        We will use this phenomenon several times in this paper, usually in a setting
        where, fixing $x,y\in \GdmodGamma$, we consider large times $T$ where $a_T x$ and
        $a_T y$ are close to each other. We lift $a_T x$ and $a_T y$ to $h$ and $g$,
        respectively, and use this argument to control recurrences of $Nx$ to $MAy$.
        See  \Cref{lemma_delta of geometric limit in Delta}, \Cref{prop:fellow travelers
          in horo closure}, and \Cref{Theorem:semigroup is everything}.

	\subsection{Geodesic laminations}
	Let $\Sigma$ be a hyperbolic $d$-manifold.
    
    \begin{definition}\label{def_geodesic lamination}
    A \emph{geodesic lamination} $\lambda \subset \Sigma$ is a non-empty closed set that is a disjoint union of complete simple geodesics called its \emph{leaves}.
    Additionally, we require that $\lambda$ is equipped with a continuous local product structure, i.e., we have the data of a covering $\{U_i\}$ of a neighborhood of $\lambda$ in $\Sigma$ and charts $\phi_i: U_i \to \R\times \R^{d-1}$ where $\phi_i(\lambda\cap U_i) = \R \times B$ and $B\subset \R^{d-1}$.  
    Moreover, the transitions are required to be of the form $\phi_{ij}(x,y) = (f_{ij}(x,y), g_{ij}(y))$, for $y\in B$.
    \end{definition}

    In the language of homogeneous spaces $ \lambda \subset K \backslash \GdmodGamma $ is a geodesic lamination if it is a projection of a closed $ A $-invariant set in $ M \backslash \GdmodGamma \cong T^1\Sigma $ where each fiber contains exactly two points.

    Although we will be working with geometrically infinite manifolds, in this paper, they typically cover a closed (or finite volume) manifold $\Sigmazero$, and our geodesic laminations are invariant under the deck group of the cover, hence define a geodesic lamination $\lambda_0 \subset \Sigmazero$.
    As such, we now specialize to the case that $\Sigmazero$ is closed.

    A lamination is \emph{connected} if it is connected as a topological space and \emph{minimal} if it has no non-trivial sublaminations, i.e. proper, non-empty subsets which are themselves laminations. 
    Alternatively, a lamination is minimal if each of its leaves is dense.
    A lamination is \emph{orientable} if it admits a continuous orientation. 
    \medskip

    When $d=2$, there is a good deal of additional structure for geodesic laminations.
    Let $S_0$ be a closed oriented surface of genus $g\ge 2$ and $\Sigmazero$ be a hyperbolic structure on $S_0$.
    The structure theory for geodesic laminations tells us that the local product structure is unique, and any geodesic lamination $\mu\subset \Sigmazero$ can be decomposed uniquely as a union of minimal sublaminations $\mu_1, ..., \mu_m$ and isolated leaves $\ell_1, ..., \ell_k$ that accumulate (or spiral) onto the minimal components with bounds $0\le k\le 6g-6$ and $1\le m\le 3g-3$.

    An argument using the Poincaré-Hopf index formula for a line field on $S_0$ constructed from a geodesic lamination proves that the area of $\mu$ is always $0$, hence the geodesic completion of the complement of $\mu$ is a (possibly non-compact, disconnected) negatively curved surface with totally geodesic boundary and area equal to the area of $\Sigmazero$.
    We sometimes refer to the surface obtained by geodesic completion of the complement of $\mu$ with respect to some negatively curved metric as being obtained by \emph{cutting open along $\mu$}.
    A lamination is \emph{filling} if the surface obtained by cutting open along $\mu$ is a union of ideal polygons.
    A lamination is \emph{maximal} if it is not a proper sublamination of any other geodesic lamination; equivalently, the cut surface is a union of $4g-4$ ideal triangles.

    Geodesic laminations were introduced by Bill Thurston in \cite{Thurston:notes, Thurston:bulletin} and have become an important tool in various problems in Teichm\"uller theory, low dimensional geometry topology and dynamics. A comprehensive introduction to the structure theory for geodesic laminations in dimension $2$ can be found in \cite{CB}; see also \cite[Chapter 4]{CEG}.
\medskip

We will later make use of the notion of chain recurrence in a geodesic lamination which goes back to Conley for general topological dynamical systems \cite{Conley:chain_recurrent}:
		\begin{definition}[see \cite{Thurston:stretch}]\label{def:chain_recurrent}
			A point $ p $ in a geodesic lamination $ \lambda $ is called \emph{chain recurrent} if for any $ \varepsilon $ there exists an $ \varepsilon $-trajectory of $ \lambda $ through $ p $. That is, there exists a closed unit speed path through $ p $ such that for any interval of length 1 on the path there is an interval of length 1 on some leaf of $ \lambda $ such that the two paths remain within $ \varepsilon $ distance from one another in the $ C^1 $ sense.
		\end{definition}

		If one point on a leaf $ \ell \subset \lambda$ is chain recurrent, then every point of $\ell$ is chain recurrent, and chain recurrence is a clearly a closed condition.
		Thus the subset $\lambda^{\CR} \subset \lambda$ of chain recurrent points is a sublamination, called the \emph{chain recurrent part} of $\lambda$.  
  
		\section{Tight Lipschitz functions and maximal stretch laminations}\label{Section_Tight functions}

    In this section, $\Sigmazero$ is a closed hyperbolic $d$-manifold, with $d\ge 2$.
    Recall that a Lipschitz map $f:\Sigmazero\to \R/c'\Z$ representing a cohomology class $\varphi \in H^1(\Sigmazero,c'\Z)$ is called \emph{tight} if the Lipschitz constant of $f$ is equal to 
    \[\kappa = \sup_{\gamma} \frac{|\varphi(\gamma)|}{\ell(\gamma)},\]
    where the supremum is taken over geodesic curves $\gamma$ in $\Sigmazero$ and we think of $|\varphi(\gamma)|$ as the geodesic length of $f(\gamma)$ in the circle $\R/c'\Z$.
    Using either the results of Daskalopoulos-Uhlenbeck \cite{DU:circle} or Gu\'eritaud-Kassel \cite{GK:stretch}, we know that given a cohomology class $\varphi \in H^1(\Sigmazero,c'\Z)$, there is a tight Lipschitz map $f:\Sigmazero\to \R/c'\Z$ inducing $\varphi$ on $\pi_1$.
    
    We will be especially interested in $1$-Lipschitz tight maps.
    Let $c = \kappa\inverse c'$ so that post-composition of $f$ with an affine map of the circle $\R/c'\Z \to \R/ c\Z$ yields a new  map \[\tau_0:\Sigmazero\to \R / c \Z\]
    representing the cohomology class $\kappa\inverse \varphi$ that is $1$-Lipschitz and tight.
    In this section, we establish some properties for arbitrary tight $1$-Lipschitz maps (not just those coming from any particular construction).

    The projection $p_K : \GmodGammazero\to \Sigmazero$ pulls $\tau_0$ back to a function that is constant on frames in the same fiber.  Abusing notation, we write $\tau_0: \GmodGammazero\to \R / c\Z$ for $\tau_0\circ p_K$.
	
    Fix some positive $\delta < c/2$ and define a function 
        \begin{align*}
            L: \GdmodGammazero &\to \R \\ 
            x & \mapsto \frac{1}{\delta}|\tau_0(a_{\delta/2}x) - \tau_0(a_{-\delta/2}x)|
        \end{align*}
        that records the efficiency of  $\tau_0$ along geodesic trajectories of length $\delta$.\footnote{Note that the absolute value in the definition of $L$ is unambiguous, because $\delta <c/2$ and $\tau_0$ is $1$-Lipschitz.}
        Clearly, $L$ is continuous,  $M$-invariant, and takes the same value on antipodal frames, i.e., if $x$ and $y$ are related by $p_M(x) = \pm p_M(y)\in M \backslash \GdmodGammazero\cong T^1\Sigmazero$, then $L(x) = L(y)$.
        Moreover, the value of $L$ does not exceed the Lipschitz constant of $\tau_0$.

	\begin{proposition}\label{proposition:definition of lamination L}
	    Let $\tau_0: \Sigmazero \to \R/c\Z$ be $1$-Lipschitz and tight, and let $\L_0\subset \GdmodGammazero$ be the $A$-invariant part of $L\inverse (1)$.  Then $p_K(\L_0)$ is a non-empty geodesic lamination on $\Sigmazero$.
	\end{proposition}

	\begin{proof}
        First we show that the $A$-invariant part of $L\inverse (1)$ is non-empty, and then we show that its projection to $\Sigmazero$ is a geodesic lamination.  Note that the $A$-invariant part of $L\inverse(1)$ is automatically $MA$-invariant (by $M$-invariance of $L$).
       
        We consider a sequence of closed geodesic curves $\gamma_n\subset \Sigmazero$ such that 
        \[\lim_{n\to \infty}\frac{|\varphi(\gamma_n)|}{\ell(\gamma_n)} =1. \]
		Now take the $MA$-invariant probability measure $\mu_n$ on $\GdmodGammazero$ supported on all frames $x$ for which $p_M(x)$ is tangent to $\gamma_n$ and symmetric under taking antipodes.
        
        Since $\GdmodGammazero$ is compact, we can extract an $MA$-invariant probability measure $\mu_0$ that is a weak-$*$ limit point of $\{\mu_n\}$.
        Up to passing to a subsequence with the same name, we can assume that $\mu_n \to \mu_0$.
        Our goal is to show that $L(x)=1$ for every $x$ in the support of $\mu_0$.  In particular, this implies that $\tau_0|_{Ax}$ is locally isometric.

        Observe that for any fixed $x\in \GdmodGammazero$, the function $t\mapsto \tau_0(a_tx)$ is locally $1$-Lipschitz, and is therefore differentiable Lebesgue almost everywhere.  
        For any $x_n$ in the support of $\mu_n$, using Fubini we have
        \begin{align*}
            \int L~ d\mu_n &=  \frac{1}{\ell(\gamma_n) } \int_{0}^{\ell(\gamma_n)}\frac{1}{\delta}|\tau_0(a_{\delta/2}a_tx_n) - \tau_0(a_{-\delta/2}a_tx_n)| ~dt  \\
            & \ge \frac{1}{\ell(\gamma_n) }\left| \frac{1}{\delta}\int_{-\delta/2}^{\delta/2}\int_0^{\ell(\gamma_n)}  \left.\frac{d}{du}\right|_{u = s}\tau_0(a_{u+t}x_n)~dt ~ds\right| \\
            & = \frac{1}{\ell(\gamma_n) }\left| \frac{1}{\delta}\int_{-\delta/2}^{\delta/2}\varphi(\gamma_n) ~ds\right| 
             = \frac{|\varphi(\gamma_n)|}{\ell(\gamma_n)},
        \end{align*} 
      
       and this ratio tends to $1$ as $n\to \infty$.
       Since $\mu_n \to \mu_0$ weak-* and $L$ is continuous, we obtain
       \[\int L ~d\mu_0 = 1.\]
       On the other hand, $L\le 1$, implying $L(x)=1$ for every $x$ in the topological support of $\mu_0$, which itself is an $MA$-invariant set.
       This proves that $\mathcal L_0 \subset L\inverse (1)$ is non-empty. 
       
        \medskip

        Our next goal is to show that $\L_0$'s projection to $\Sigmazero$ is in fact a geodesic lamination. 
        By continuity,  $L\inverse(1)$ is closed, and hence also $\L_0$.  
        It therefore suffices to prove that if $x, y\in \L_0$ satisfy $p_K(x) = p_K(y)$, then $p_K(Ax) = p_K(Ay)$. 
        
        It will be convenient to work with a lift $\tau: \GdmodGamma\to \R$ of $\tau_0$.  Let $\L = p_\Z\inverse (\L_0)$.
        For sake of contradiction, suppose there are $x, y\in \L$ where $p_K(x)=p_K(y)$ but $p_K(Ax)$ meets $p_K(Ay)$ at a definite angle (different from $0$ or $\pi$).

        Without loss of generality, we assume that for $s>0$, $\tau$ is increasing along both $p_K(a_sx)$ and $p_K(a_sy)$, so  $\tau(p_K(a_{s}y))-\tau(p_K(a_{-s}x))=2s$.
        Since $p_K(Ax)$ and $p_K(Ay)$ meet at a definite angle, the length of the shortest geodesic segment connecting $p_K(a_{-s}y)$ to $p_K(a_sx)$ is strictly smaller than $2s$.  
        This contradicts the fact that $\tau$ is $1$-Lipschitz on $\Sigma$ and completes the proof that $p_K(\L_0)$ is a non-empty geodesic lamination.
	\end{proof}

    	We draw the following useful corollary:
    	\begin{lemma}\label{lemma:u_is_a_1C-qi}
        Let $\tau_0: \Sigmazero\to \R/c\Z$ be $1$-Lipschitz and tight, and let $\tau: \Sigma \to \R$ be any lift.
		There exists a constant $ C_\tau>0 $ satisfying for all $ y,z \in \GdmodGamma $
		\begin{equation}\label{eq:u_is_a_1C-qi}
			d(y,z)-C_\tau \leq |\tau(y)-\tau(z)| \leq d(y,z).
		\end{equation}
	\end{lemma}
	
	\begin{proof}
		The upper bound in \eqref{eq:u_is_a_1C-qi} follows from $ \tau $ being 1-Lipschitz on $\Sigma$. Let us prove the lower bound. 		
		Given $ k \in \Z \cong \Gammazero/\Gamma $ denote by $ k.\square $ the isometric action of the deck transformation on $ \Sigma $. For any $ k $ and $ x \in \GdmodGamma $ we have $ \tau(k.x)=\tau(x) + ck $, and $ F=\tau^{-1}([0,c)) $ is a precompact fundamental domain for the $ \Z $-action on $ \GdmodGamma $. By Proposition \ref{proposition:definition of lamination L}, there exists $ x \in \L \subset \GdmodGamma $ satisfying $ \tau(a_t x) = \tau(x) + t $ for all $ t $. We will use the values of $ \tau $ along $ Ax $ as a kind of ruler for coarsely measuring distances in $ \Sigma $.
		
		Now, given $ y,z \in \GdmodGamma $ there are $ t_y$ and $t_z $ satisfying $ \tau(a_{t_y}x)=\tau(y) $ and $ \tau(a_{t_z}x)=\tau(z) $. Since $a_{t_y}x$ and $y$ lie in the same $\tau$ level set, they lie in the same $ \Z $-translate of $F$, and so their distance in $\GdmodGamma$ is bounded by $\mathrm{diam}(F)$, and similarly for $a_{t_z}x$ and $z$.
        Thus
		\begin{align*}
			d(y,z) &\leq d(y,a_{t_y}x)+d(a_{t_y}x,a_{t_z}x)+d(a_{t_z}x,z) \leq \\
			&\leq |t_y-t_z|+2\mathrm{diam}(F)  \\
			&= |\tau(y)-\tau(z)| + 2\mathrm{diam}(F).
		\end{align*}
		Set $ C_\tau=2\mathrm{diam}(F) $ and the claim follows.
	\end{proof}

    Recall our notation of
	\[ \Qm = \left\{ x \in \GdmodGamma : (a_t x)_{t \geq 0} \text{ is quasi-minimizing} \right\} \]
	and
	\[ \Qm_\pm = \{ x\in \Qm : \lim_{t \to \infty} \tau(a_t x)=\pm\infty \}. \]
	That is, $ \Qm_+ $ is the set of all points in $ \GdmodGamma $ whose geodesic trajectory as $ t \to +\infty $ is quasi-minimizing and escaping through the positive end of $ \Sigma $.
	\begin{definition}
		We define the \emph{$ \omega $-limit set mod $ \Z $} of an $ A $-invariant set $ F \subset \GdmodGamma $ by
		\[ \olim(F) = p_\Z^{-1}\left( \left\{ y \in \GdmodGammazero : \exists x \in p_\Z(F) \text{ and } t_j \to \infty \text{ s.t. } a_{t_j}x \to y \right\} \right). \]
		That is, $ \olim(F) $ contains the lifts of accumulation points in $ \GdmodGammazero $ of geodesic trajectories in $ p_\Z(F) $.
	\end{definition}

	Denote $ \Qm_\omega = \olim(\Qm) $, the $\omega$-limit set of quasi-minimizing rays in $ \GdmodGamma $.
    Let $\L_0^{\CR}$ denote the chain recurrent part of $\L_0$.
    Then $\Qm_\omega$ is related to $\L_0^{\CR}$ as follows.

    \begin{theorem}[Definition of $\lambda_0$]\label{Theorem:omega limit of Qm is a lamination in max stretch locus of u}
    The projection $p_K(p_\Z(\Qm_\omega))$, denoted by $\lambda_0$, is a chain recurrent geodesic lamination on $\Sigmazero$ contained in $p_K(\L_0^{\CR})$.
    
    Furthermore, every $ A $-orbit in $ \Qm_\omega $ is mapped isometrically to $\R$ by $\tau$ and is bi-minimizing, i.e., isometrically embedded in $\GdmodGamma$.
    \end{theorem}

	\begin{proof}
    	First we show that $p_\Z(\Qm_\omega)\subset L\inverse(1)$. The basic idea is that if a point in the the $ \omega $-limit of a quasi-minimizing ray is not itself bi-minimizing, then the quasi-minimizing ray would be ``losing'' a definite amount of $ \tau $-value with each close encounter to that limit point, which would contradict the qausi-minimizing property, see \cref{figure_omegalimits}.
    
		Fix $ y \in \Qm_\omega $ and let $ x $ be a point in $ \Qm $ with $ a_{t_j}\pZ{x} \to \pZ{y} $ for some $ t_j \to \infty $. Assume for sake of contradiction that $ p_\Z (y) \notin L\inverse(1) $.  In other words, we have
		\[ L(p_\Z(y))<1-\varepsilon \]
		for some $ \varepsilon > 0 $. The function $ L $ is continuous and therefore $ L^{-1}\left([0,1-\varepsilon)\right) $ is an open set in $ \GdmodGammazero $ containing $ p_\Z(y) $. In particular, any point $z\in\GdmodGamma$ close enough to any $\Z$ translate of $y$ has $L(z)<1-\epsilon$.
 
		Since $ a_{t_j}\pZ{x} \to \pZ{y} $, for all large $ j $ we have
		$ L(p_\Z(a_{t_j}x)) < 1-\epsilon $, or in other words
		\begin{equation}\label{eq:delta epsilon}
|\tau(a_{t_j+\delta/2}x)- \tau(a_{t_j-\delta/2}x)| < \delta(1-\varepsilon). 
\end{equation}
		We may assume the above holds for all $ j\geq 1 $ and that $ a_{t_{j+1}} > a_{t_j}+\delta $. 
  Rename the $t_j\pm \delta/2$ as an increasing sequence $0=s_0<s_1<s_2<\cdots$. Then we have for any $k>1$ the telescoping sum
  \[
      \tau(a_{s_{2k}}x) - \tau(x)  = \sum_{i=1}^{2k} \tau(a_{s_i}x)-\tau(a_{s_{i-1}}x).
  \]
  Each term is bounded above by $s_{i}-s_{i-1}$ because $\tau$ is 1-Lipschitz.  Moreover, for each even $i$ we have $s_{i}-s_{i-1} = \delta$, which together with  \Cref{eq:delta epsilon} gives
  \[\tau(a_{s_i}x)-\tau(a_{s_{i-1}}x) < s_i-s_{i-1} - \delta\epsilon.\]
  We conclude that 
  \[\tau(a_{s_{2k}}x) -\tau(x) < s_{2k} - k\delta\epsilon.\]

		On the other hand, \Cref{lemma:u_is_a_1C-qi} implies
		\[ d(a_{s_{2k}}x,x)-C_\tau \leq  \tau(a_{s_{2k}}x)- \tau(x)\leq s_{2k} - k \delta \varepsilon \]
		for some $ C_\tau > 0 $ and all $ k \geq 1 $, which contradicts our assumption that $ x $ is quasi-minimizing.
		Hence we conclude that $ y \in L\inverse(1)$. Since $ \Qm_\omega $ is $ A $-invariant, we deduce that $p_\Z(\Qm_\omega)\subset \L_0$.

		\begin{figure}[h]
			\includegraphics[scale=0.7]{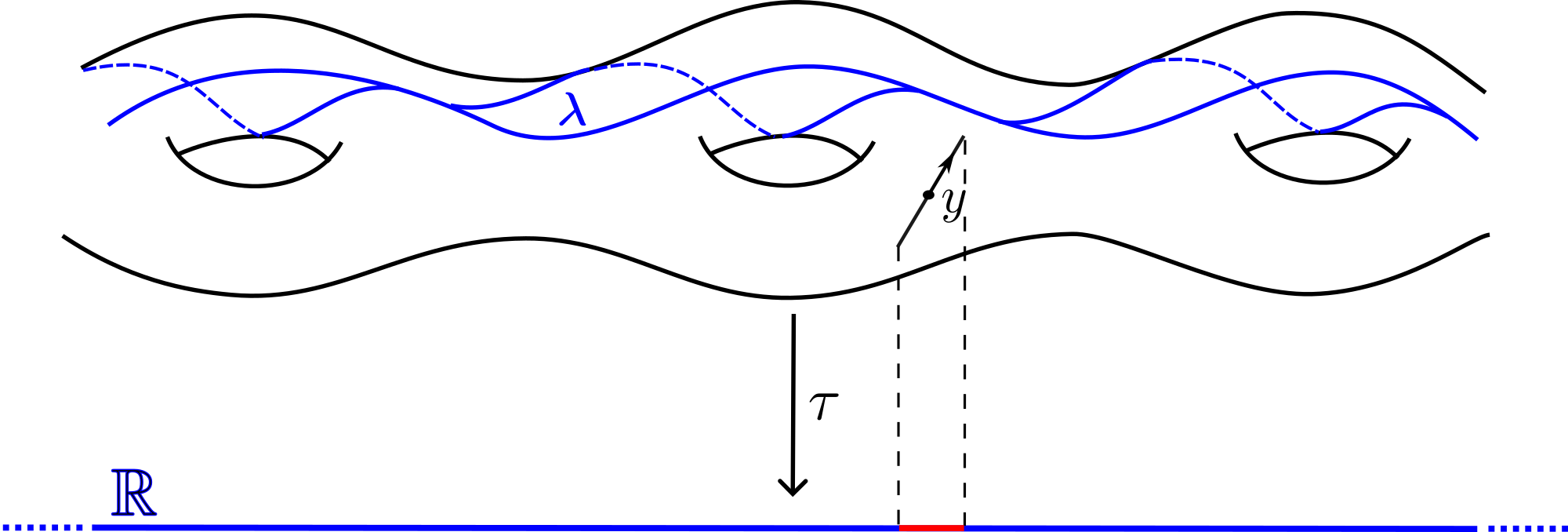}
			\caption{The $\omega$-limit set of $ \Qm $ is Lipschitz maximizing.}
			\label{figure_omegalimits}
		\end{figure}	

		Chain recurrence follows from the fact that for any $ y \in \Qm_{\omega} $ and any $\epsilon > 0$, we can take a segment of a quasi-minimizing ray whose end points are $\epsilon$-close to $ y $ and which stays $\epsilon$-close to $ \Qm_{\omega} $. Projecting this segment to $ \Sigma_{0} $ and adjusting in a suitable neighborhood of the endpoints gives a closed loop which is an $ \varepsilon $-trajectory of the lamination $p_K(\Qm)$ passing through $ p_K(p_\Z(y)) $.
	\end{proof}

	In dimension $ d=2 $ we can say more, completing the proof of \Cref{Thm intro: Q_omega is L^CR}:

\begin{prop}\label{prop:equality of omega limits}
	If $d=2$, then $p_\Z(\cal Q_\omega) = \L_0^{\CR}$.
\end{prop}

\begin{proof}
    When $d=2$, we can use the structure theory for geodesic laminations to construct a quasi-minimizing ray that accumulates onto any point of $\L_0^{\CR}$.  From this we will conclude that $\L_0^{\CR}\subset p_\Z(\Qm_{\omega})$.

    Let $\mu = p_K(\L_0^{\CR})$ and $\mu_\omega$ be the $\omega$-limit set (of $\mu$).   The classification theory tells us that $\mu_\omega$ consists of finitely many minimal components, and $\mu\setminus \mu_\omega$ consists of finitely many isolated leaves; each spirals onto a minimal component in its future and onto another in its past.

    Clearly, for any $x\in \GdmodGamma$ with $p_\Z(x)\in \L_0^{\CR}$ and $p_K(p_\Z(x)) \in \mu_\omega$, then $Ax$ is bi-minimizing and $p_\Z(Ax)$ accumulates onto $p_\Z(x)$.  This shows that the $\omega$-limit set of $\L_0$ is contained in $p_\Z(\Qm_\omega)$.
    
    If $p \in \mu\setminus\mu_\omega$ lies on an isolated leaf, then chain recurrence of $\mu$ tells us we can find a closed loop $\gamma_0$ through $p$ of the form \[\gamma_0 = g_1\cdot j_1 \cdot g_2 \cdot j_2\cdot ... \cdot j_{k-1} \cdot g_k,\] where $g_i$ are segments of isolated leaves $\ell_i$ of $\mu$ that spiral onto minimal components $\mu_i$; the $j_i$ are small segments that that join $g_{i}$ to $g_{i+1}$ that are $C^1$-close to a leaf of $\mu_i$. 
    
    Let $\epsilon_n$ be a summable sequence of small positive numbers.
    We can create an $\epsilon_n$-trajectory $\gamma_n$ of $\lambda_0$ through $p$ by modifying $\gamma_0$ as follows.
    Since $g_i$ accumulates onto $\mu_i$ in its forward direction and $g_{i+1}$ accumulates onto $\mu_i$ in its backward direction, we can extend the future of the former and the past of the later so that their endpoints are joined by segment of size at most $\epsilon_n$ that is $\epsilon_n$ $C^1$-close to a leaf of $\mu_i$. 
    We can concatenate the $\gamma_n$'s to form an infinite path $r=\gamma_1 \cdot \gamma_2 \cdot ...$.  
    Let $\tilde r$ be a lift of $r$ to $\bH^2$ and let $\tilde r^*$ be the geodesic ray with the same initial point that is asymptotic to $\tilde r$.  
    By stability of geodesics in hyperbolic space, there is a constant $c$ such that  \[d_{\bH^2}(\tilde r(t), \tilde r^*(t))\le c \text{ for all $t$.}\]
    Moreover, if $r(t_n)$ belongs to $\gamma_n$, then there is a $t_n^* \in [t_n - c, t_n +c]$ such that 
    \[d_{\bH^2}(\tilde r^*(t_n^*),\tilde r(t_n)) \le c \epsilon_n.\]
    In particular, the projection $r^*$ of $\tilde r^*$ accumulates onto $p$ in $\Sigmazero$.

    Abusing notation, we let $r$ and $r^*$ also denote lifts to $\Sigma$. Then since $\tau$ is isometric along the leaves of $\mu$  and $1$-Lipschitz, we have
    \[d_\Sigma(r(t), r(0))\ge \tau(r(t)) - \tau (r(0)) \ge t - \sum _{j\le n} \epsilon_j, ~\text{when $r(t)$ belongs to $\gamma_n \subset r$}.\]
    The triangle inequality then provides 
    \[d_\Sigma(r^*(0),r^*(t))\ge \tau(r^*(t)) - \tau(r^*(0)) \ge t - \sum_{j\le n} \epsilon_j - c.\]
    Since $\sum_{j = 1}^\infty \epsilon_j$ is finite, we conclude that $r^*$ is quasi-minimizing.
    Together with \cref{Theorem:omega limit of Qm is a lamination in max stretch locus of u} the proposition follows.
\end{proof}

\begin{remark}
    Although we do not have a structure theorem for geodesic laminations in dimensions $d\ge 3$, it is not difficult to see that our argument  can be adapted to prove $\L_0^{\CR} = p_\Z(\Qm_\omega)$ whenever the complement of the set of minimal components of $p_K(\L_0^{\CR})$  consists of at most countably many geodesics.
\end{remark}

 Thurston studied the set of points realizing the maximum local Lipschitz constant for tight maps between finite area complete hyperbolic surfaces.
 He proved that the \emph{maximal stretch locus} comprises a geodesic lamination and that the chain recurrent part is contained in the maximum stretch locus of any tight map between the same surfaces \cite[Theorem 8.2]{Thurston:stretch}.  
 As a corollary of the Proposition \ref{prop:equality of omega limits}, we deduce the following analogue of Thurston's result. 

\begin{cor}\label{cor:chain_recurrent_canonical}
    To a given cohomology class $\varphi\in H^1(S;c\Z)$ and hyperbolic structure $\Sigmazero$, the geodesic lamination $p_K(\L_0^{\CR})$ constructed from any tight-Lipschitz map $f: \Sigmazero\to \R/c\Z$ inducing $\varphi$ on $\pi_1$, depends only on $\varphi$.
\end{cor}

\begin{proof}
By Proposition \ref{prop:equality of omega limits}, $\mathcal L_0^{\CR} = p_\Z(\mathcal Q_\omega)$, which only depends on the $\Z$-cover $\Sigma$ of $\Sigmazero$ corresponding to $\varphi$ and not the tight Lipschitz function $f: \Sigmazero \to \R/c\Z$.
\end{proof}

\subsection{Overview}\label{Subsec:Tight map summary of notations} We end this section with a brief summary of notations and inclusions for future reference. Let $ \Sigma \to \Sigma_0 $ be a $ \Z $-cover of a compact hyperbolic manifold to which we associated the following:
\begin{center}
	\begin{tabular}{|c|p{9.2cm}|}
		\hline
		Notation & Definition \\ [0.5ex] 
		\hline\hline
		$ \tau_0 : \Sigma_0 \to \R/c\Z $ & 1-Lipschitz tight map corresponding to $\Sigma\to \Sigmazero$ 
		\\ \hline
		$ \tau : \Sigma \to \R $ & 1-Lipschitz $ \Z $-equivariant lift of $ \tau_0 $ \\ \hline
		$ \cal{L}_0 $ & maximal $ A $-invariant subset of $ L^{-1}(1) \subset \GdmodGammazero$
		\\ \hline
		$ \cal{L} $ & lift of $ \mathcal{L}_0 $ to $ \GdmodGamma $, i.e.~$ p_\Z\inverse({\mathcal{L}_0}) $ \\ \hline
         $ \cal L_0^{\CR} $, $\cal L^{\CR}$ & chain recurrent part of $ \L_0 $ and its lift to $\GdmodGamma$ \\ \hline
	\end{tabular}
\end{center}

All the sets above implicitly depend on the choice of a tight map $ u_0 $. We further consider the following canonical loci, which are independent of the tight map:

\begin{center}
	\begin{tabular}{|c|p{9.2cm}|}
		\hline
		Notation & Definition \\ [0.5ex] 
		\hline\hline
		$ \Qm $ & set of quasi-minimizing points in $ \GdmodGamma $ \\ \hline
		
		$ \Qm_{\omega} $ & $ \omega $-limit set mod $ \Z $ of $ \Qm $ in $ \GdmodGamma $ \\
		\hline
		$ \lambda_0 $ & geodesic lamination in $ \Sigma_0 $ equal to $ p_K (p_\Z(\Qm_\omega)) $\\ \hline
		$ \lambda $ & lift of $ \lambda_0 $ to $ \Sigma $, i.e.~$ p_\Z\inverse({\lambda_0}) $ (also $p_K(\Qm_\omega)$) \\ \hline
	\end{tabular}
\end{center}
In the sequel, the canonical laminations $\lambda_0$ and $\lambda$ will be referred to as the \emph{minimizing} laminations.
Note the following relationship between the canonical objects. 
\begin{equation}
	\begin{tikzcd}
		\Qm_\omega \arrow{r}{p_K} & \lambda \arrow{r}{p_\Z} & \lambda_0	 
	\end{tikzcd}.
\end{equation}

Recall that $x\in \GdmodGamma$ is \emph{bi-minimizing} if $t\mapsto a_tx$ is an isometric embedding. 
There is a sequence of inclusions:
\begin{equation}\label{eq:Q_L_inclusions}
	\Qm_{\omega} \subset \L^{\CR} \subset \mathcal{L} \subset \{\text{bi-minimizing points}\} \subset \Qm \subset \GdmodGamma.
\end{equation}
In the case of $ d=2 $, the first inclusion is an equality, that is $ \Qm_{\omega}=\mathcal{L}^{\CR} $ which is independent of the choice of tight function. The authors suspect that the remaining inclusions might otherwise be strict.

\section{The non-horospherical limit set in \textit{d} = 2}

In this section we show that the subset $ \Lambda_{\mathrm{nh}} \subset \partial \bH^2 $ of non-horospherical limit points is of Hausdorff dimension zero. Since the $ \omega $-limit set of $ \Qm $ is the lift of a geodesic lamination in $ \Sigma_0 $, the statement is a consequence of the following theorem:
\begin{theorem}\label{Theorem:Hausdorff 0}
Let $S=\bH^2/\Gamma$ be a finite area hyperbolic surface and $\lambda$ a geodesic lamination.  Let
$\omega(\lambda)$ be the set of endpoints in the circle at infinity of rays whose
$\omega$-limit sets are in $\lambda$. Then the Hausdorff dimension of $\omega(\lambda)$ is
0.
\end{theorem}

Note that the dimension of the set of endpoints of $\lambda$ itself is 0 by Birman-Series \cite{BS} (and
similarly Bonahon-Zhu \cite{BonZhu:HD}). This is an
extension of those ideas, though we note that $\omega(\lambda)$ is a larger set.

Before we start, recall some structural facts about laminations and train tracks:

For $\ep$ smaller than the injectivity radius of $S$, an {\em $\ep$-track} for a lamination
$\lambda$ is an $\ep$-neighborhood of $\lambda$, divided into a finite number of
rectangles or ``branches'',  whose ``long'' edges run along $\lambda$ and whose short edges have length
$O(\ep)$. Branches are foliated by arcs (``ties'') parallel to the short edges and transverse to $\lambda$, and are attached to
each other along the short edges. Collapsing the ties we get a 1-complex whose vertices are images of connected unions of short edges. A train route is a path running through the track
transverse to the  ties. Each train route is $O(\ep)$-close to a
geodesic which is also transverse to the ties. We think of a train route as combinatorial -- two routes are the same if they traverse the same branches in the same order. In particular we say a train route is {\em embedded} if it passes through each vertex at most once.  We say a train route is {\em almost embedded} if it passes through each branch at most twice, with opposite orientations.

A train route is a cycle if it begins and ends in the same branch. When $\ep$ is sufficiently small, no two cycles in an $\ep$-track can be homotopic. 

A geodesic lamination $\lambda$ in a finite area hyperbolic surface can have a finite number of
closed leaves (possibly none). Let $\lambda_c$ denote  the union of these leaves, 
and let $r$ denote the number of components of $\lambda_c$. 

\begin{lemma}\label{embedded routes}
There is a bound $m_0$ such that, for any train track in $S$, there are at most $m_0$
almost embedded train routes. 
\end{lemma}
\begin{proof}
There is a finite number of homeomorphism types
of train tracks on a given surface. 
\end{proof}

\begin{lemma}\label{track cycles}
Given $\lambda$ and  $K>0$ there exists $\ep$ so that for any $\ep$-track for $\lambda$, any cyclic train route of
length bounded by $K$ is homotopic to a component of $\lambda_c$. 
\end{lemma}

\begin{proof}
Fix $K$ and consider a sequence $\ep_i\to 0$, and $\ep_i$-tracks $\sigma_i$. Suppose for each $i$ we have a train route
$\beta_i$ in $\sigma_i$ of length at most $K$. After taking a subsequence, $\beta_i$ must converge to
a closed geodesic, and on the other hand the $\ep_i$ neighborhoods converge to
$\lambda$. Hence eventually $\beta_i$ is homotopic to one of the components of
$\lambda_c$. 
\end{proof}

\begin{lemma}\label{track routes}
There exist $c_0, d$ such that, for any $K>0$, if $\ep$ is small enough then any $\ep$-track for
$\lambda$  has at most $c_0 K^d$ train routes of length bounded by $K$.
\end{lemma}

\begin{proof}
Recall that $r$ is the number of components of $\lambda_c$.
As a warm-up consider the case that $\lambda_c$ is empty ($r=0$). Hence fixing $K$,
 for sufficiently small $\ep$ every train route of length $K$ is
almost embedded -- if it were not then a sub-route would give a cycle, contradicting Lemma \ref{track cycles}. Thus, there are at most $m_0$ of them by Lemma \ref{embedded routes} and we are
done with $d=0$.

For $r>0$, let $\gamma_1,\ldots,\gamma_r$ be the components of $\lambda_c$.
There is exactly one closed train route $\hat\gamma_i$ homotopic to each $\gamma_i$, since no
two closed train routes are homotopic.

Fixing $K$, choose $\ep$ as given by Lemma \ref{track cycles}. Possibly choosing $\ep$ even smaller we can arrange that the $\hat\gamma_i$ are embedded and disjoint. If a train route $\alpha$ of length at most $K$ traverses an edge of $\hat\gamma_i$ twice with the same orientation, then between those two traversals it must remain in $\hat\gamma_i$, for otherwise there would be a cycle of length at most $K$ which is not homotopic to any of the
$\gamma_j$.

It follows that we can remove an integer number of traversals of each $\hat\gamma_i$ from
$\alpha$, leaving an almost embedded train route $\alpha_0$. The number of such $\alpha_0$, then,
is bounded by $m_0$. For each one of them we can ``splice'' back in a power
of each $\hat\gamma_i$ at a suitable place, but their lengths can add to at most $K$. 
The number $d$ of places to splice is bounded by twice the maximal number of edges in any track on $S$ (because each edge can appear at most once with each orientation) and a power $p$ of $\hat\gamma_i$ spliced in contributes at least $p$ times the length of $\gamma_i$. Hence
there is a bound of the form $c(K/b)^{d}$ on the number of ways to do this, where $b$ is the
minimal length of a component of $\lambda_c$.

Adding over all $m_0$ possible $\alpha_0$'s we obtain a bound of the desired form on the
number of all train routes of length bounded by $K$. 
\end{proof}

\begin{lemma}\label{route growth}
Given $s>0$ there exist $\ep>0$ and $c_1$ such that, for any $\ep$-track of $\lambda$,
the number $N(L)$ of train routes of length at most $L$ is bounded by
$$N(L) \le c_1 e^{sL}.$$
\end{lemma}

\begin{proof}
  Fix $K>0$ and let $\ep=\ep(K,\lambda)$ and $ d \geq 0 $ be given by Lemma \ref{track routes}. For $L>>K$, any
  train route can be divided into $\lceil L/K \rceil$ segments of length at most
  $K$. Since there are at most $c_0 K^d$ possibilities for each of these, the total number
  of routes of length at most $L$ is at most
  $$
  (c_0 K^d)^{\lceil L/K \rceil}.
  $$
Since $K^{1/K}\to 1$ as $K\to\infty$, for sufficiently large $K$ this is bounded by an
expression of the form $c_1 b^L$, for $b$ as close as we like to $1$. This gives the
desired bound.   
\end{proof}
  
\begin{proof}[Proof of Theorem \ref{Theorem:Hausdorff 0}]
  Fixing a basepoint $x_0$ of $\bH^2$, the set $\omega=\omega(\lambda)$ determines a set
  $R\omega$ of 
rays emanating from $x_0$ and landing at the points of $\omega$. 

Fix $\delta>0$ and let $\cal H_\delta$ denote $\delta$-dimensional Hausdorff
content.  We must prove $\cal H_\delta(\omega) = 0$. 

Let $\ep$ be the value given by  Lemma \ref{route growth} for $s = \delta/2$, and let
$\tau$ be an $\ep$-track for $\lambda$. Every ray in $R\omega$ is eventually carried in the
preimage track $\tilde\tau$. Let $\omega_n\subset\omega$ denote the subset whose rays are carried in
$\tilde \tau$ after time $n$. Thus $\omega = \cup_n \omega_n$ and it suffices to prove
$\cal H_\delta(\omega_n) = 0$ for each $n$.

Any ray in $R\omega_n$ passes, at time $n$, through some branch in
$\tilde\tau$. There are finitely many such branches meeting the circle of radius $n$; let
$m_1=m_1(n)$ denote their number. Fixing such a branch $e$, all the rays of $R\omega_n$
passing through $e$ continue further along $\tilde\tau$ in an infinite train route. For
$L>0$, consider the initial train routes of length $L$ beginning in $e$.

Any such route projects to a train route of $\tau$ beginning at the image of $e$. Two
different train routes upstairs project to different routes downstairs (all the routes
begin together, and as soon as two diverge their images do as well). Thus by lemma
\ref{route growth} there are at most $c_1 e^{sL}$ such routes.

These routes provide a covering of $\omega_n(e)$, the subset corresponding to rays that
pass through $e$ at time $n$. Two rays with the same initial length-$L$ route stay at
distance $\ep$ apart for at least length $L$ (in fact $L+n$), which means their endpoints are at most
$O(e^{-L})$ apart on the circle. Thus, we have a covering of $\omega_n(e)$ by $c_1 e^{sL}$
intervals of length  $O(e^{-L})$. Summing over all the $m_1$ branches, we find
$$
\cal H_{\delta}(\omega_n) < m_1 c_1 e^{sL} e^{-\delta L}.
$$
Since $s<\delta$, this goes to $0$ as $L\to\infty$. This completes the proof that $\cal
H_\delta(\omega_n) = 0$, so that $\cal H_\delta(\omega)=0$ as well for all $\delta>0$  and so $\dim_{\mathrm{H}}(\omega) = 0$.
\end{proof}
  
\begin{remark}\label{Remark:packing dimension}
	In fact a stronger claim holds, that is, \emph{the upper packing dimension of $ \omega(\lambda) $ is zero}. Recall the definition of the upper packing dimension of a set $ E $ in a metric space is
	\[ \overline{\dim}_{\mathrm{p}} E = \inf \left\{ \sup_{j\geq 1} \overline{\dim}_{\mathrm{box}}(E_j) : E \subseteq \bigcup_{j=1}^\infty E_j  \right\}, \]
	where $ \overline{\dim}_{\mathrm{box}} $ denotes the upper box (or Minkowski) dimension and the infimum is defined over all countable covers of $ E $ by bounded sets, see e.g. \cite[\S 5.9]{Mattila_Book}.
	
	Indeed, using the notations in the proof of \cref{Theorem:Hausdorff 0}, we have for all $ \delta > 0 $ a decomposition of the set $ \omega(\lambda) $ into countably many subsets $ \omega_n $. For any $ L>0 $, let $ N(\omega_n,e^{-L}) $ denote the minimal number of sets of diameter at most $ e^{-L} $ needed to cover $ \omega_n $. Then we have shown
	\[ N(\omega_n,e^{-L}) \leq m_1 c_1 e^{\delta L},  \]
	and hence
	\[ \overline{\dim}_{\mathrm{box}} (\omega_n) = \limsup_{L \to \infty} \frac{\log N(\omega_n,e^{-L})}{-\log (e^{-L})} \leq \delta. \]
	Therefore, by definition
	\[ \overline{\dim}_{\mathrm{p}} \omega(\lambda) \leq \sup_{n\geq 1} \overline{\dim}_{\mathrm{box}}(\omega_n) \leq \delta \]
	for all $ \delta > 0 $, implying $ \overline{\dim}_{\mathrm{p}} \omega(\lambda) = 0 $.
\end{remark}

	\begin{cor}
		For $ d=2 $, the set of quasi-minimizing points $ \Qm $ in $ \GmodGamma $ has Hausdorff dimension 2. In particular, for any $ x \in \Qm $
		\[ 1 \leq \dim_{\mathrm{H}}\overline{Nx} \leq 2. \]
	\end{cor}
	
	\begin{proof}
		Let $ \tilde{\Qm} $ denote the lift of $ \Qm $ to $ G $. The smooth covering map ensures that $ \dim_{\mathrm{H}} \Qm = \dim_{\mathrm{H}} \tilde{\Qm} $ hence it suffices to consider the latter. Since $ \tilde{\Qm} $ is $ AN $-invariant we may present it under the Iwasawa decomposition as a product set $ \tilde{\Qm}=NAK_{\mathrm{nh}} $ where $ K_{\mathrm{nh}} \subset K $ (recall that the multiplication map $ N \times A \times K \to G $ is a diffeomorphism, see e.g.~\cite[Theorem 6.46]{Knapp}). Furthermore, $ K_{\mathrm{nh}} $ is bi-Lipschitz equivalent to the non-horospherical limit set $ \Lambda_{\mathrm{nh}} \subset \partial \Hplane $. Therefore
		\[ \dim_{\mathrm{H}} \tilde{\Qm} = \dim_{\mathrm{H}} NA + \dim_{\mathrm{H}} \Lambda_{\mathrm{nh}} = 2+0,\footnote{Note that both factors in the product above have their Hausdorff dimension equal to their packing dimension, in particular implying the product formula.} \]
		by \Cref{Theorem:Hausdorff 0}.
		
		The second claim follows directly from the inclusions
		\[ Nx \subset \overline{Nx} \subset \Qm \]
		for all $ x \in \Qm $. For the last inclusion see the remark after the Eberlein-Dal'bo theorem in \Cref{sec:intro}.
	\end{proof}
	
\section{Constructing minimizing laminations in \textit{d} = 2}\label{Sec_construction of minimizing laminations}
	
	Let $S_0$ be an oriented closed (topological)  surface of genus $g\ge 2$.
	In this section, we use methods specific to dimension $d=2$ to construct minimizing laminations for $\mathbb Z$-covers of closed hyperbolic surfaces with prescribed geometric and dynamical properties.
	
	\subsection{Background on surface theory}
	We rely on a dictionary between certain kinds of geometric, dynamical, and topological objects on surfaces;  we require some background.
	
	\para{Measured laminations}
    Endow $S_0$ with an auxiliary negatively curved metric, let $\lambda_0\subset S_0$ be a minimal geodesic lamination and let $k$ be a transversal, i.e., a $C^1$ arc meeting $\lambda_0$ transversely.  
    Giving $k$ an orientation gives a local orientation to the leaves of $\lambda_0$, and following the leaves of $\lambda_0$ induces a continuous Poincar\'e first return map \[P : k\cap \lambda_0 \to k\cap \lambda_0.\]
    If $k'$ is another transversal and $k$ is isotopic to $k'$ preserving the transverse intersections with $\lambda_0$, then the corresponding dynamical systems are conjugated by the isotopy, and the invariant ergodic measures are in bijective correspondence.

    A \emph{transverse measure} $\mu_0$ supported on $\lambda_0$ is, to each transversal $k$, the assignment of a finite positive Borel measure  $\mu_0(k)$ supported on $k\cap \lambda_0$ that is invariant holonomy and natural with respect to inclusion.   
    A geodesic lamination $\lambda_0$ equipped with a transverse measure $\mu_0$ is a \emph{measured lamination}.
    
    A measured lamination can also be thought of as a flip and $A$-invariant finite Borel measure $\mu_0$ on $T^1S_0$ whose support projects to a  geodesic lamination on $S_0$. 
    Since the geodesic foliation of $T^1S_0$ does not depend on a choice of negatively curved metric, the space of measured laminations with its weak-$*$ topology only depends on the topology of $S_0$.  Denote this space by $\ML(S_0)$.
       See \cite{Thurston:notes, Thurston:bulletin, PennerHarer, CB, FLP} for a development of the theory of measured laminations.

	\para{Measured foliations}
	A (singular) measured foliation on a surface $S_0$ is a $C^1$ foliation $\mathcal F_0$ of $S_0\setminus Z$, where $Z$ is a finite set (called singular points) equipped with a transverse measure $\nu_0$ on arcs transverse to $\mathcal F_0$.  The transverse measure is required to be invariant under holonomy and every singularity is modeled on a standard $k$-pronged singularity; see  \cite{Thurston:bulletin, FLP} for details and further development. 
	Isotopic measured foliations are viewed as identical.
        A measured foliation is orientable if there is a continuous orientation of the non-singular leaves.
	
	The space  $\MF(S_0)$ of \emph{Whitehead equivalence classes} of singular measured foliations (see Figure \ref{fig:whitehead}) is equipped with a topology coming from the geometric intersection number with homotopy classes of simple closed curves on $S_0$.
	A theorem of Thurston asserts that $\MF(S_0)$ is a $6g-g$ dimensional manifold. 
	We remark that Whitehead equivalence does not in general preserve orientability of a foliation.

 \begin{figure}
    \centering
    \includegraphics[scale = .3]{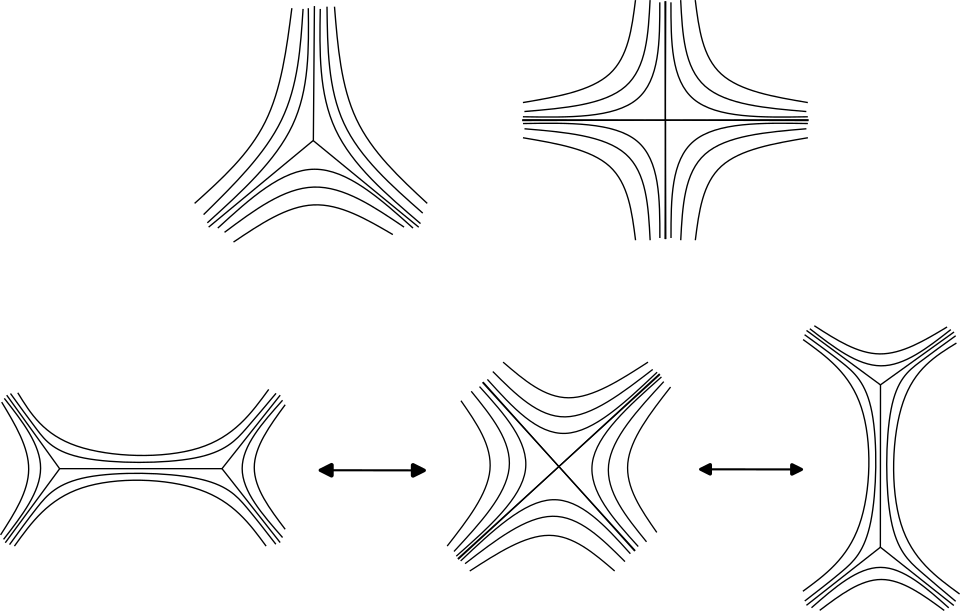}
    \caption{$3$- and $4$-pronged singularities and Whitehead moves. }
    \label{fig:whitehead}
 \end{figure}

 \para{Equivalence of measured foliations and laminations}
	With respect to our auxiliary negatively curved metric on $S_0$, we can pull each non-singular leaf of a measured foliation $(\mathcal F_0, \nu_0)$ tight in the universal cover to obtain a geodesic.  The closure  is a geodesic lamination invariant under the group of covering transformations that projects to a geodesic lamination $\lambda_0$ on $S_0$, and $\lambda_0$ carries a measure of full support $\mu_0$ obtained in a natural way from $\nu_0$.

This procedure defines a natural homeomorphism $\MF(S_0)\to \ML(S_0)$ \cite{Levitt:MFML}, so that we may pass between measure equivalence classes of measured foliations and the corresponding measured lamination at will.
	In the sequel, we will often abuse notation and write $\lambda_0\in \ML(S_0)$ or $\lambda_0 \in \MF(S_0)$ to refer to both the underlying geodesic lamination/equivalence class of foliation and to the transverse measure.

    \subsection{The orthogeodesic foliation}
    Let $\cT(S_0)$ denote the Teichm\"uller space of marked hyperbolic structures on $S_0$. Fix $\Sigmazero\in \cT(S_0)$ and let $\lambda_0$ be a
    geodesic lamination.

    The complement of $\lambda_0$ in $\Sigmazero$ is a (disconnected) hyperbolic surface whose metric completion has totally geodesic (non-compact) boundary.  For each such complementary component $Y$, away from a piecewise geodesic $1$-complex $\Sp$ called the \emph{spine}, there is a nearest point in $\partial Y$. 
    The fibers of the projection map $Y\setminus \Sp \to \partial Y$ form a foliation of $Y\setminus \Sp$ whose leaves extend continuously across $\Sp$ to a piecewise geodesic singular foliation $\cO(Y)$ called the \emph{orthogeodesic foliation} with $n$-prong singularities at the vertices of the spine of valence $n$.
    Every endpoint of every leaf of $\cO(Y)$ meets $\partial Y$ orthogonally; see Figure \ref{fig:orthofoliation}.
    \begin{figure}
        \centering
        \includegraphics[scale = .36]{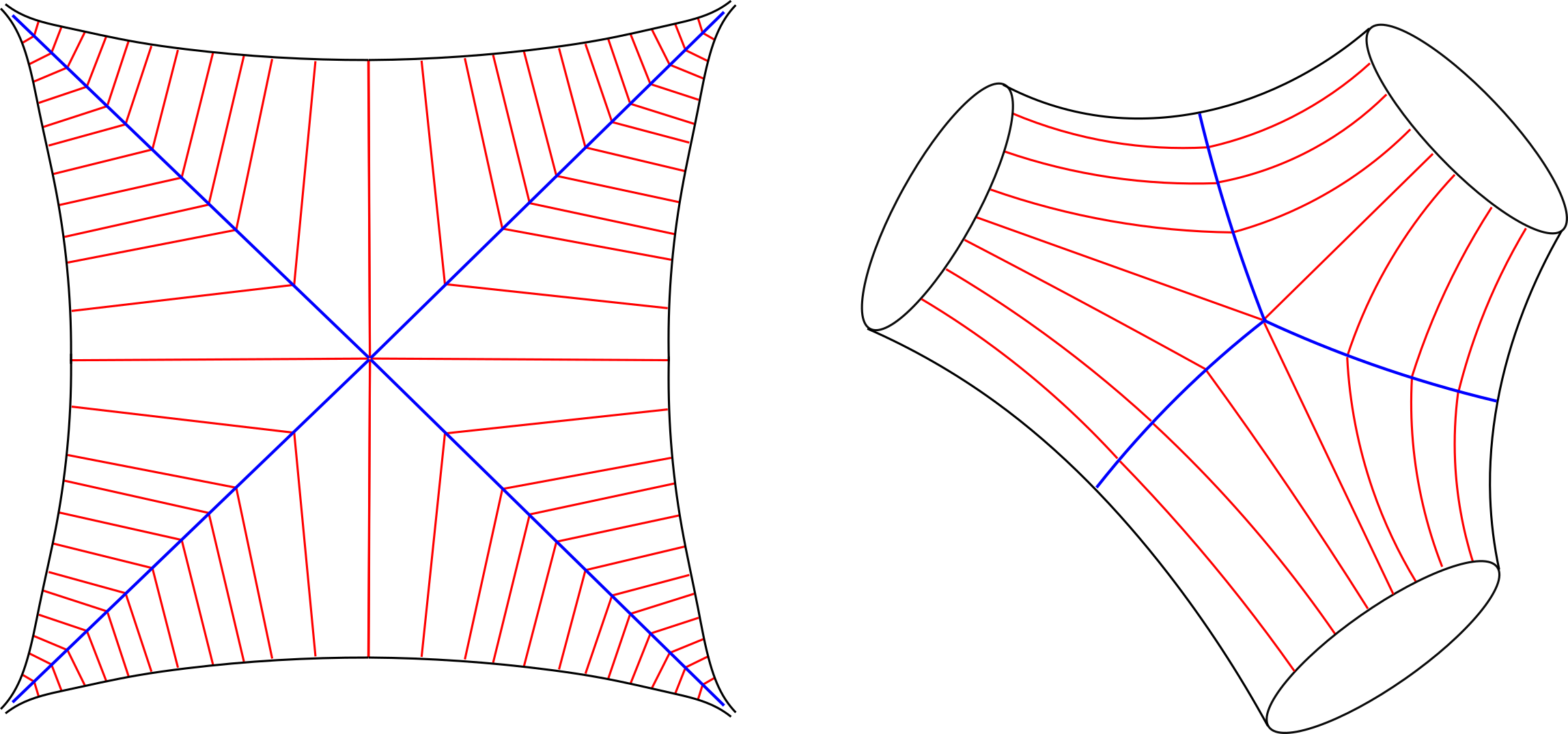}
        \caption{Leaves of the orthogeodesic foliation in red with critical graph in blue.  }
        \label{fig:orthofoliation}
    \end{figure}
    
    The orthogeodesic foliation $\cO(Y)$ is equipped with a transverse measure; the measure assigned to  a small enough transversal is Lebesgue after projection to $\partial Y$, and 
    this assignment is invariant by isotopy transverse to $\cO(Y)$. 
    This construction produces a foliation on  $\Sigmazero\setminus \lambda_0$, which extends continuously across the leaves of $\lambda_0$ and defines a   singular measured foliation $\cO_{\lambda_0}(\Sigma_0)$ on $S_0$; see \cite[Section 5]{CF:SHSH} for more details.
 
    To a geodesic lamination $\lambda_0$, we have produced a map \[\cO_{\lambda_0} : \cT(S_0) \to  \MF(S_0),\]
    corresponding to the  Whitehead equivalence class of $\cO_{\lambda_0}(\Sigmazero)$.
    The image of $\cO_{\lambda_0}$ lands in the set $\MF(\lambda_0)$ of foliations that \emph{bind} together with $\lambda_0$, i.e. $\MF(\lambda_0)$ consists of those $\eta \in \MF(S_0)$ for which there is an $\epsilon>0$ satisfying 
    \begin{equation}\label{eqn:binding}
        \int_\gamma ~d\lambda_0 + \int_\gamma ~d\eta >\epsilon,
    \end{equation}
    for all simple closed curves $\gamma\subset S_0$ meeting both $\lambda_0$ and $\eta$ transversely. 

	The following theorem was proved by Calderon-Farre \cite{CF:SHSH} building on work of Mirzakhani \cite{Mirzakhani:EQ}, Thurston \cite{Thurston:stretch}, and Bonahon \cite{Bon:SPB}. 
    \begin{theorem}\label{thm:hyperbolize}
        For each measured geodesic lamination $\lambda_0 \in \ML(S_0)$,
        \[\cO_{\lambda_0} : \cT(S_0)\to \MF(\lambda_0)\]
        is a homeomorphism.
    \end{theorem}    

    The following  is one of the main ingredients in the proof of Theorem \ref{Thm intro: tighten-lamination}.
    \begin{corollary}\label{cor:find a hyperbolic metric}
        Let $\lambda_0\in \ML(S_0)$ be oriented and suppose $\alpha$ is an oriented multi-curve with positive $c\Z$-weights such that every intersection of $\lambda_0$ with $\alpha$ is positive and such that $\lambda_0$ and $\alpha$ satisfies $S_0\setminus (\lambda_0\cup\alpha)$ consists of compact disks.
        Then there is a unique hyperbolic metric $\Sigmazero\in \cT(S_0)$ such that $\cO_{\lambda_0}(\Sigmazero)$ is equivalent to $\alpha$ in $\ML(S_0)$.  
        Moreover,  $\cO_{\lambda_0}(\Sigmazero)$ is oriented and every intersection with $\lambda_0$ is positive.
    \end{corollary}

    \begin{proof}
    We will prove that the measured foliation equivalent to $\alpha$ lies in $\MF(\lambda_0)$.
    Suppose not.  Then we can find a sequence of simple closed curves $\gamma_n$ such that 
    \[\int_{\gamma_n} |d\alpha| + \int_{\gamma_n}|d\lambda_0| \to 0.\]
    Since $|d\alpha|$ consists of atomic measures supported on its components, $\gamma_n$ is eventually disjoint from $\alpha$.
    Let $\mu_0$ denote any weak-* accumulation point of  $\gamma_n/\ell_{\Sigmazero}(\gamma_n)\in \ML(S_0)$.
    Continuity of the intersection number gives that the intersection of $\mu_0$ with $\lambda_0$ is zero.  We conclude that $\mu_0$ either has support contained in $\lambda_0$ or is disjoint from $\lambda_0$. Both possibilities are impossible, since then $\mu_0$ would cross $\alpha$ essentially.
    This proves that $\lambda_0$ and $\alpha$ bind in the sense of \cref{eqn:binding}.

    We can apply Theorem \ref{thm:hyperbolize} to deduce that there is a unique hyperbolic metric $\Sigmazero$ on $S_0$ so that $\cO_{\lambda_0}(\Sigmazero) = \alpha$.  
    It follows that every non-singular leaf of $\cO_{\lambda_0}(\Sigma_0)$ is closed and isotopic to a component of $\alpha$.  The orientation of $\alpha$ is compatible with one of the orientations of $\cO_{\lambda_0}(\Sigma_0)$, which then meet  $\lambda_0$ positively. 
    \end{proof}

     \subsection{Tight maps with prescribed stretch locus}\label{sec:lip_graph}
     In this section, we will prove Theorem \ref{Thm intro: tighten-lamination} from the introduction, which we restate here in a slightly more general form.
    \begin{theorem}\label{thm:tighten lamination}
        Let $\varphi \in H^1(S_0,c\Z)$ and let $\lambda_0$ be the support of an oriented measured geodesic lamination on $S_0$.  Suppose that $\varphi$ is Poincar\'e dual to a homology class represented by an oriented multi-curve $\alpha$ with positive $c\Z$ weights  that meets $\lambda_0$ transversely and positively and such that $\alpha\cup \lambda_0$ binds $S_0$.

        Then there is a hyperbolic metric $\Sigma_0 \in \cT(S_0)$ and a $1$-Lipschitz tight map $\tau_0: \Sigma_0\to \R/c\Z$ inducing $\varphi$ on homology such that the locus of points whose local Lipschitz constant is $1$ is equal to $\lambda_0$ .
    \end{theorem}

    \begin{proof}
        We apply \Cref{cor:find a hyperbolic metric} to find a unique hyperbolic structure $\Sigma_0\in \cT(S_0)$ such that $\cO_{\lambda_0}(\Sigma_0)= \alpha$ and such that $\cO_{\lambda_0}(\Sigma_0)$ and  $\lambda_0$ are positively oriented.
        The leaf space $\mathcal G$ of $\cO_{\lambda_0}(\Sigma_0)$ is a directed  graph with a metric $d$ induced by integrating the transverse measure; there is an oriented edge for every component of $\alpha$.

        The leaf space $\cal G$ is obtained by collapsing the leaves of $\cO_{\lambda_0}(\Sigmazero)$ in each complementary component $Y$ of $\Sigmazero\setminus {\lambda_0}$.
        In other words, a point $p\in Y$ is identified with its nearest point on ${\lambda_0}$.  Let $\pi: \Sigma_0\to \mathcal G$ denote the quotient projection. 
        Then $\pi$ maps the leaves of $\lambda_0$ locally isometrically preserving orientation to $\mathcal G$ by construction of the transverse measure on $\cO_{\lambda_0}(\Sigmazero)$.
        
        The nearest point projection map onto a geodesic in hyperbolic space is a strict contraction away from the geodesic, so the quotient map $\pi: \Sigma_0 \to \mathcal G$ is $1$-Lipschitz, and the $1$-Lipschitz locus is exactly $\lambda_0$.
        There is a canonical map $\mathcal G \to \R/c\Z$ that is orientating preserving, locally isometric along the edges of $\mathcal G$ and such that the composition with $\pi$ induces $\varphi$ on $\pi_1$.
        Let $\tau_0: \Sigma_0\to \cal G \to \R/c\Z$ denote this composition.

        Tightness of $\tau_0$ is an immediate consequence of the construction: any sequence of geodesic curves $\gamma_n\subset \Sigma_0$ that converge in the Hausdorff topology to $\lambda_0$ will satisfy \[\lim_{n\to \infty} \frac{|\varphi(\gamma_n)|}{\ell(\gamma)} =1,\]
        giving a lower bound for the supremum of this ratio over all curves $\gamma$, while the Lipschitz constant of $\tau_0$ gives an upper bound.  This completes the proof of the theorem.   
    \end{proof}

    \subsection{Minimizing laminations in $\mathbb Z$-covers with prescribed dynamics}
    Now we use Theorem \ref{thm:tighten lamination} to construction tight $1$-Lipschitz maps to the circle whose maximum stretch locus has prescribed dynamical properties.  More precisely, given a $1$-Lipschitz tight map $\tau_0: \Sigma_0\to \R/c\Z$ obtained by collapsing leaves of $\cO_{\lambda_0}(\Sigmazero)$ as in the previous section, let $\gamma$ be the preimage of a point, which is generically a simple multi-curve.  We consider the first return  map \[P : \gamma \cap \lambda_0 \to \gamma \cap \lambda_0 \]
    obtained via the flow tangent to $\lambda_0$ in the positive direction for time $c$.

    In this section, we prove \Cref{Thm:Construction of surface with weak mixing IET}, in which we assume that $\gamma$ is connected and show that the dynamics of $P$ can be prescribed by the data of an interval exchange transformation $T: I \to I$.

        \para{Interval Exchange transformations}
	 An \emph{interval exchange transformation} (or \emph{IET}) $T: I \to I$  is a peicewise isometric orientation preserving bijection of an interval $I$ to itself determined by a finite partition of $I$ into sub-intervals $I_1, ..., I_p$ and a permutation $\pi$ that reorders them. 
  Thus $T$ clearly preserves the Lebesgue measure.
    
    An IET is \emph{minimal} if every orbit is dense.
	A sufficient condition for minimality is that the lengths $\{|I_j|\}_{j = 1}^p$ and $|I|$ are rationally independent and $\pi$ is \emph{irreducible}, i.e., $\pi$ does not map $\{1, ..., j\}$ to itself for $j<p$.
	Thus, almost every IET with a given permutation is minimal with respect to the Lebesgue measure on the $p$-dimensional simplex parameterizing the lengths of the subintervals.
 
 	Recall the following definitions:
 	\begin{definition}\;\label{Def:weakly mixing}
 		\begin{enumerate}
 			\item A transformation $T: I\to I$ preserving a probability measure $ \nu $ is called \emph{weakly mixing} if $T\times T: I\times I\to I \times I$ is ergodic with respect to $ \nu \times \nu $. 
 			\item Given two probability measure preserving systems $ (Y,\nu_Y,T) $ and $ (Z,\nu_Z,S) $, we say that $ Z $ is a \emph{factor} of $ Y $ if there exist co-null subsets $ Y' \subseteq Y $ and $ Z' \subseteq Z $ with $ TY' \subseteq Y' $ and $ SZ' \subseteq Z' $, and there exists a measurable map $ \psi: Y' \to Z' $ satisfying
 			\[ \nu_Z=\psi_* \nu_Y \text{ and }\quad \psi \circ T = S \circ \psi \;\;\text{ on }Y'. \]
 			We say the systems are \emph{isomorphic} if $ \psi $ is additionally assumed to be invertible.
 			\item A continuous transformation $S: Z\to Z$ is called \emph{topologically weakly mixing} if $ S \times S $ is topologically transitive in $ Z \times Z $, or equivalently if all continuous and $S\times S$-invariant functions $f: Z\times Z \to \R$ are constant.
 	\end{enumerate} 	\end{definition}
 	\noindent It is  known that almost every IET with irreducible permutation $\pi$ is weakly mixing \cite{Avila-Forni}; see also  \cite{NR:top_weak_mixing}.
 	\medskip

  We show the following:
    \begin{theorem}\label{Thm:Construction of surface with weak mixing IET}
    Given an irreducible IET $T: I \to I $ not equivalent to a circle rotation, there is a closed oriented surface $S_0$ such that the following holds.
    For any primitive  cohomology class $\varphi \in H^1(S_0;\Z)$ and any positive $c>0$, there is a marked hyperbolic structure $\Sigmazero\in \cT(S_0)$ and a tight Lipschitz map \[\tau_0 : \Sigmazero\to \R/c\Z\] inducing $c\cdot\varphi$ on cohomology; the minimizing lamination $\lambda_0$ is equipped with a transverse measure $\mu_0$ and the first return system  $( \tau_0^{-1}(x)\cap \lambda_0, \mu_0, P )$ is a factor of $(I, Leb,T)$.  
    
    If $\mu_0$ has no atoms (equivalently, $T$ has no periodic orbits or $\lambda_0$ has no closed leaves), then the two systems are isomorphic.
    \end{theorem}

    The proof factors through the construction of a \emph{translation structure} encoding the dynamics of $T$.
    Briefly, a translation structure $\omega$ on a closed surface $S_0$ is the data of a finite set $Z\subset S_0$ and an atlas of charts to $\mathbb C$ on $S_0\setminus Z$, such that the transition maps are translations.  
    The transitions preserve the Euclidean metric as well as the horizontal and vertical directions, i.e., the $1$-forms defined locally in charts by $dy$ and $dx$.
    The horizontal and vertical (kernel) oriented measured foliations on $S_0\setminus Z$ extend to singular measured foliations on $S_0$, which are binding in the sense of \eqref{eqn:binding}. For details see \cite{GM}, which in particular states that the space of (half) translation structures is parameterized by pairs of binding measured foliations on $S_0$.
    
    \begin{proof}
    Our plan is to build a translation structure on a closed surface $S_0$ as a suitable suspension of $T: I \to I$ with constant roof function $c>0$.  The leaf space of the vertical measured foliation is the circle $\R/c\Z$ and the quotient map represents $\varphi$.
    This flat structure will provide us with a pair of binding measured foliations.  We then use Theorem \ref{thm:tighten lamination} to hyperbolize this example and conclude the theorem.

    To build the flat structure, endow the rectangle $[0,c] \times I\subset \mathbb C$ with its Euclidean structure,  identify its horizontal sides isometrically by translation and its vertical sides by $(c,p)\sim(0, T(p))$.
    See \Cref{fig:IET}.

    \begin{figure}[ht]
        \centering
        \includegraphics[scale = .44]{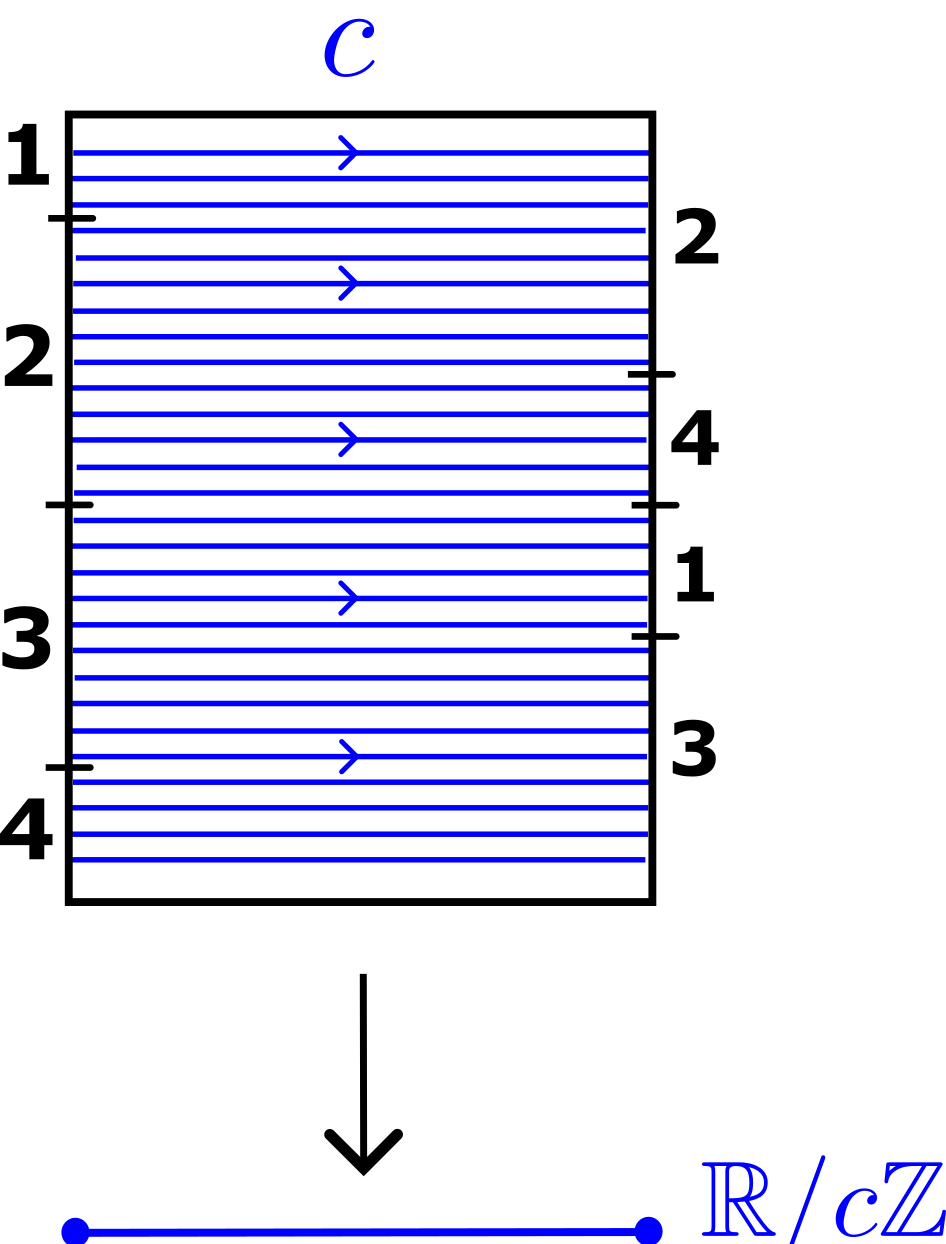}
        \caption{The suspension of an IET with constant roof function.  Horizontal blue lines each have length $c$, wrap once around the circle, and becoming the maximal stretch lamination. Vertical lines are pairwise isotopic simple closed curves, except for the critical graph of vertical saddle connections.}
        \label{fig:IET}
    \end{figure}
    
    The resulting object $\omega$ is homeomorphic to a closed orientable surface $S_0$.
    Since the action of the mapping class group of $S_0$ on $H^1(S_0,c\Z)$ is transitive on primitive $c\Z$-valued cohomology classes, we may assume that  the quotient map $q\to \R/c\Z$ collapsing leaves of the vertical foliation of $\omega$ represents the class of $c\cdot \varphi$.

    We observe that the non-singular leaves of the vertical foliation are isotopic simple closed curves.  Let $\alpha$ denote this isotopy class, oriented suitably. Then $c\cdot \alpha$ represents the Poincar\'e dual of $c\cdot \varphi$, and $\alpha$ meets positively the horizontal oriented measured foliation of $\omega$.  Call the corresponding lamination $\lambda_0$ and transverse measure $\mu_0$.
    
    We apply \Cref{thm:tighten lamination} which gives us a hyperbolic metric $\Sigma_0\in \cT(S_0)$ such that $\cO_{\lambda_0}(\Sigma_0)=\alpha$ and a $1$-Lipschitz tight map $\tau_0: \Sigma_0\to \R/c\Z$ inducing $\varphi$ on $\pi_1$, with maximal stretch locus $\lambda_0$.
    Moreover, $\cO_{\lambda_0}(\Sigma_0)$ is isotopic to the vertical foliation of $\omega$ (c.f. \cite[Proposition 5.10]{CF:SHSH}).
    
    After postcomposing $\tau_0$ with a rotation of the circle, we may assume that the critical graph of $\cO_{\lambda_0}(\Sigma_0)$ lies entirely in the fiber over $c\Z$.
    Then the following follow directly from the construction:
    \begin{itemize}
            \item For any $x,y \in \R/c\Z \setminus \{c\Z\}$, the pre-images $\tau_0^{-1}(x)$ and $\tau_0\inverse(y)$ are isotopic simple closed curves.
            \item The Poincar\'e first return map $P:\tau_0^{-1}(x)\cap {\lambda_0} \to \tau_0^{-1}(x)\cap {\lambda_0}$ is induced by the geodesic flow tangent to ${\lambda_0}$ at time $c$ in the forward direction.
    \end{itemize}
    
    We can define a map $\tau_0^{-1}(x)\cap {\lambda_0} \to I $ by choosing a point $p\in \tau_0^{-1}(x)$ and integrating 
    \[\pi: q\mapsto \int_p^q ~d\mu_0\in I,\]
    where we use the positively oriented arc from $p$ to $q$.
    For a suitable choice of $p$, this map is dynamics preserving, i.e., $P^n(q) = T^n(\pi(q))$ for all $q$ and $n$; see \cite[Proposition 5.10]{CF:SHSH}.
    If $\mu_0$ has no atoms, then $\pi$ is also continuous, measure preserving factor map and injective outside of a countable set, hence invertible Lebesgue almost everywhere.
    When $\mu_0$ has atoms, our factor map goes in the other direction, namely, we can collapse the holes made by the finitely many atoms in the image of $\pi$ to define a measure preserving map  $I\setminus \{\text{singular orbits}\} \to \tau_0^{-1}(x)\cap \lambda_0$ that inverts $\pi$.  
    There are only countably many singular orbits, so this mapping is defined almost everywhere and gives the desired factorization.
    \end{proof}

    \subsection{Geometric convergence}\label{subsec:geometric convergence}
    In this section we construct examples of $\Z$-covers that are geometrically close, but where the dynamics of the corresponding bi-minimizing loci are very different.
    We will use this example to prove \Cref{Theorem:non-rigidity} that the topology of $N$-orbit closures can change in the limit along a geometrically convergent sequence of regular $\mathbb Z$-covers.

    \begin{theorem}\label{Theorem:Epsilon deformation of weak mixing}
        Let $S_0$ be a closed orientable surface of genus $\geq 2$ and $S\to S_0$ be a $\Z$-cover.  Let $S \to\Sigma$ be a marked hyperbolic metric constructed from a weak-mixing IET as in \Cref{Thm:Construction of surface with weak mixing IET} with minimizing lamination $\lambda \subset \Sigma$.
        There exist marked  hyperbolic metrics $S\to \Sigma_n$ satisfying:
        \begin{enumerate}
            \item The minimizing lamination $\lambda_n\subset \Sigma_n$ consists of finitely many (uniformly) isolated geodesics and $\lambda_n\to \lambda$ in the Hausdorff topology.
            \item  There are $(1+\epsilon_n)$-bi-Lipschitz maps $f_n:\Sigma\to \Sigma_n$ in the homotopy classes of the markings, where $\epsilon_n\to 0$.
            \item Denote by $d$ the metric on $\Sigma$ and $d_n$ the  metric on $\Sigma_n$.  There is a constant $C$ such that \[|d(p,q) - d_n (f_n(p), f_n(q))|<C\]
            for all $p, q\in \Sigma$ and $n$.
        \end{enumerate}
    \end{theorem}

    The proof of this theorem relies on a continuity property of the maps $\{\cO_{\mu_0}\}_{\mu_0\in \ML(S_0)}$, which do not in general vary continuously.  However, with respect to a finer topology on measured lamination space, we have:

    \begin{theorem}[\cite{CF:SHSHII}]\label{thm:O_continuous}
        Suppose $\mu_{0,n}\in \ML(S_0)$ is a sequence of measured geodesic laminations converging in measure to $\mu_0$ such that the supports of $\mu_{0,n}$ converge in the Hausdorff topology to the support of $\mu_0$. Then for any $\eta \in \MF(\mu_0)$, the sequence $\Sigma_{0,n} = \cO_{\mu_{0,n}}\inverse (\eta) $ converges to $\Sigma_0 = \cO_{\mu_0}\inverse (\eta)$ as $n\to \infty$.
    \end{theorem}
    \begin{proof}[Proof of  \Cref{Theorem:Epsilon deformation of weak mixing}]
    Let $T: I\to I$ be a weak-mixing IET, and let $\omega$ be the translation structure on $S_0$ constructed in the proof of  \cref{Thm:Construction of surface with weak mixing IET}.
    Let $\lambda_0 \in \ML(S_0)$ be equivalent to the horizontal measured foliation of $\omega$, and let $\eta$ denote the vertical measured foliation.

    We can approximate $\lambda_{0}$ in measure by $\lambda_{0,n}\in \ML$ whose supports are multi-curves.  
    Indeed, there is a natural oriented train track $\sigma$ with one switch and $p$ branches constructed from the permutation $\pi$ defining $T$; $\lambda_0$ defines a positive weight system $w$ on $\sigma$ giving positive weight to each branch; see \cref{IET-construction}.  
    
    \begin{lemma}
        Any sequence $w_n$ of rational weight systems on $\sigma$ satisfying the switch conditions and converging to $w$ represent oriented weighted multi-curves $\lambda_{0,n}\in \ML(S)$ converging to $\lambda_0$ both in measure and in the Hausdorff topology.
    \end{lemma}

    \begin{proof} 
    There is a splitting sequence $\sigma = \sigma_1 \succ \sigma_2 \succ ... $ such that the intersection of the positive cones of measures carried by $\sigma_i$ is the set of measures with support equal to $\lambda_0$. 
    Since the $\lambda_{0,n}$ are converging to the horizontal foliation of $q$ in measure, for every $K$, there is an $N$ such that for $n\ge N$, $\sigma_K$ carries $\lambda_{0,n}$.
    It follows, using e.g. \cite{BonZhu:HD}, that $\lambda_{0,n} \to \lambda_0$ in the Hausdorff topology. 
    \end{proof}

    Since $\lambda_{0,n}$ converge to $\lambda_0$ in measure, we have $\eta\in \MF(\lambda_{0,n})$ as long as $n$ is large enough.  Using Theorem \ref{thm:hyperbolize}, define $\Sigma_{0,n}= \cO_{\lambda_{0,n}}\inverse (\eta)$;
     Theorem \ref{thm:O_continuous} gives $\Sigma_{0,n} \to \Sigma_0 \in \cT(S_0)$.
    This implies that there are $(1+\epsilon_n)$-bi-Lipschitz maps $f_{0,n}: \Sigma_{0,n} \to \Sigma_{0}$ in the homotopy class compatible with the markings, where $\epsilon_n\to 0$.
    
    Let $\tau_{0,n}: \Sigma_{0,n}\to \R/c\Z$ (also $\tau_0: \Sigma_0\to \R/c\Z$) be the $1$-Lipschitz tight maps with minimizing laminations $\lambda_{0,n}$ (or $\lambda_0$)  obtained from either Theorem \ref{thm:tighten lamination} or Theorem \ref{Thm:Construction of surface with weak mixing IET}.
    Lift $f_{0,n}$ to $\Z$-covers $f_n: \Sigma\to \Sigma_n$ and  let $\tau_n:\Sigma_n\to \R$ and $\tau: \Sigma\to \R$ be lifts. 
    Then $f_n: \Sigma\to \Sigma_n$ are $(1+\epsilon_n)$-bi-Lipschitz.
    We still need to show that the metrics $d$ and $f_n^*d_n$ only differ by an additive error, which follows from the following claim together with Lemma \ref{lemma:u_is_a_1C-qi}.

    \begin{claim*}
        For all $p,q\in \Sigma$, there is a $C'$ such that
        \[||\tau(p) - \tau(q)| -|\tau_n(f_n(p))-\tau_n(f_n(q))||\le C' \]
        for all $n$.
    \end{claim*}
    \begin{proof}[Proof of the Claim] 
        Let $k\in \Z$ be such that $|\tau(p) - \tau(k.q)|\le c$ (there are at most two possibilities).  
        Then $\Z$-equivariance of $\tau$ gives 
        \[c(k-2)\le |\tau(p) - \tau(q)|\le c(k+2).\]
        
        Let $D$ be the diameter of $\Sigma_0$. 
        Then \[d(p,k.q)\le (c+D).\]  Since $f_n$ is $(1+\epsilon)$-bilipschitz, we have \[d_n(f_n(p), f_n(k.q)) \le (1+\epsilon_n) (c+D).\] 
        Since $\tau_n$ is $1$-Lipschitz, we obtain \[|\tau_n(f_n(p)) - \tau_n(f_n(q))|\le (1+\epsilon_n) (c+D).\]  
        Now $\Z$-equivariance of both $f_n$ and $\tau_n$ give us 
        \[ck - (1+\epsilon_n) (c+D) \le |\tau_n(f_n(p)) - \tau_n(f_n(q))| \le ck+(1+\epsilon_n) (c+D). \]

        Taking $C' = \max_n\{2c + (1+\epsilon)(c+D)\}$ proves the claim.
    \end{proof}
    This concludes the proof of the theorem.
    \end{proof}

    \subsection{An example with a single bi-minimizing ray}\label{Subsection:single bi-minimizing ray}
    Here, we construct a $\Z$-cover of a closed surface (of any genus) with a single bi-minimizing ray.  This construction is elementary and ad hoc, as opposed to being an application of either Theorem \ref{thm:tighten lamination} or Theorem \ref{Thm:Construction of surface with weak mixing IET}.  

    \begin{lemma}\label{lem: single bi-minimizing}
        For any closed oriented surface $S_0$ of genus at least $2$ and primitive class $\varphi\in H^1(S_0,c\Z)$, there is a hyperbolic structure $\Sigma_0\in \cT(S_0)$  such that the $\Z$-cover $\Sigma$ corresponding to $\varphi$ has a single bi-minimizing geodesic so that $\Qm_\omega$ consists of a single geodesic line.
    \end{lemma}

    The construction occupies  the rest of the section.  The idea is to build a metric structure on a torus with one short boundary component, which we can then attach to any higher genus surface with one boundary component mapping trivially to the circle $\R/c\Z$.  The long skinny neck of the torus acts as a barrier: no bi-minimizing geodesic wants to take an excursion into this region.

    \begin{proof}
    Let $\Sigma_0(s,a)$ be a (finite area) hyperbolic metric on a one holed torus with totally geodesic boundary curve $\gamma$  of length $s$ such that there is a pair of simple closed geodesics $\alpha$ and $\beta$ intersecting once orthogonally where the length of $\alpha$ is $a$.  The length of $\beta$ is $c$, see \cref{fig:unique_bi_minimizer}. 
    These conditions determine $\Sigma_0(s,a)$ uniquely (up to marking).
    We let $a$  be any positive number (allowing $c$ to take on any positive value), and take $s$ such that the collar neighborhood about $\gamma$ has length larger than the  length of $\beta$ plus the length of $\alpha$.
    For example $s = e^{-10}$ and $a = 1$ would work.

        Let $\Sigmazero$ be a hyperbolic metric on a surface of genus at least $2$ where $\Sigma_0(s,a)$ embeds isometrically.
    We form a regular $\Z$-cover $\Sigma$ of $\Sigmazero$ by cutting open $\Sigmazero$ along $\alpha$ and placing $\Z$ many copies of this building block together, one after the next; see \Cref{fig:unique_bi_minimizer}.
    This is the cover corresponding to intersection with $\alpha$.
    Then the preimage $\tilde\beta$ of $\beta$ consists of a single bi-infinite geodesic exiting both ends of $\Sigma$.

    \begin{claim*}
        The canonical orientable chain recurrent geodesic lamination $\lambda_0 = p_K(p_\Z(\Qm_\omega))$ is contained in $\Sigmazero(s,a)$. 
    \end{claim*}
    \begin{proof}[Proof of the claim]
    Any bi-minimizing geodesic in $\Sigma$ must exit the ends of $\Sigma$ and pass through every component of the preimage of $\alpha$.
    The distance in $\Sigma$ between any pair of points in two neighboring  translates of $\alpha$ is at most the length of $\beta$ plus that of $\alpha$. 
    We chose $\gamma$ with a long collar neighborhood, so a bi-minimizing geodesic will never cross a lift of $\gamma$ and is therefore completely contained in the subsurface $\Sigma(s,a) \subset \Sigma$ covering $\Sigma_0(s,a)$.
    By \cref{Theorem:omega limit of Qm is a lamination in max stretch locus of u},  $\Qm_\omega$ consists only of bi-minimizing points, and so must be contained in the frame bundle over $\Sigma(s,a)$, proving the claim. 
    \end{proof}

    Now we prove that $\beta \subset \lambda_0$.
    There is a $1$-Lipschitz retraction  \[\pi_0:\Sigmazero(s,a) \to \beta\] defined as follows.  Cut $\Sigmazero(s,a)$ open along $\alpha$ to obtain a pair of pants, and then cut this pants open along the three orthogeodesics joining different boundary components.  
    We have a pair of isometric right angled hexagons.
    Now, the nearest point projection on each hexagon to the side corresponding to $\beta$ is $1$-Lipschitz and a strict contraction away from $\beta$.  Since $\Sigmazero(s,a)$ is obtained by gluing together these hexagons without any twist, i.e., so that the endpoints of $\beta$ match up and meet $\alpha$ orthogonally, the retractions on hexagons glue together to the desired retraction $\pi_0$.

    Then $\pi_0$
    lifts to a $1$-Lipschitz retraction $\pi:\Sigma(s,a)\to \tilde \beta$ that is strictly contracting away from $\tilde \beta$.
    This proves that $\tilde\beta$ is isometrically embedded in $\Sigma(s,a)$, hence also in $\Sigma$.
    
    For a geodesic $g$ in a hyperbolic surface, let $T^1g$ denote the set of unit vectors tangent to $g$.  Any $x\in T^1\tilde \beta$ is (quasi-)minimizing, and satisfies that $p_\Z (x)$ is an accumulation point of $p_\Z(A_{>0}x)$.  Thus $T^1\tilde \beta \subset \Qm_\omega$ and so $\beta \subset \lambda_0$.

    Now we prove that $\lambda_0\subset \beta$.   There are no non-peripheral closed curves in $\Sigma_0(s,a)\setminus\beta$ and any leaf of $\lambda_0$ must intersect $\alpha$.  This implies that any leaf of $\lambda_0$ must accumulate onto $\beta$.  Up to finite symmetry, either $\lambda_0 = \beta$ or $\lambda_0 = \beta \cup \delta$, where the geodesic $\delta$ is pictured in \Cref{fig:unique_bi_minimizer}.

    Consider any lift $\tilde\delta \subset \Sigma(s,a)$ of $\delta$. It spends some definite amount of time in the locus of points where the local Lipschitz constant of $\pi$ is strictly smaller than $1$.  Far  out in the ends of $\Sigma(s,a)$,  $\tilde \delta$ is arbitrarily close to $\tilde\beta$.  Thus the distance traveled along $\tilde \delta$ between a pair of far away points very close to $\tilde\beta$ is some definite amount larger than the distance traveled along $\tilde\beta$, which implies that $\tilde \delta$ cannot be bi-minimizing.
    We conclude that $\lambda_0 = \beta$.
    
    Finally, we would like to say that there is no other bi-minimizing geodesic in $\Sigma$ besides $\tilde\beta$ (even outside of $\lambda_0$).  Suppose $\delta'$ is bi-minimizing.  Then again any $x\in T^1\delta'$ is (quasi-)minimizing, so its projection to $T^1\Sigmazero$ accumulates onto $T^1\lambda_0$.  Thus $\delta'$ must be asymptotic to $\tilde\beta$ in both directions.  
    Applying an identical argument as in the previous paragraph shows that $\delta'$ cannot be different from $\tilde\beta$, which is what we wanted to show.

    \begin{figure}
        \centering
        \includegraphics[scale=.45]{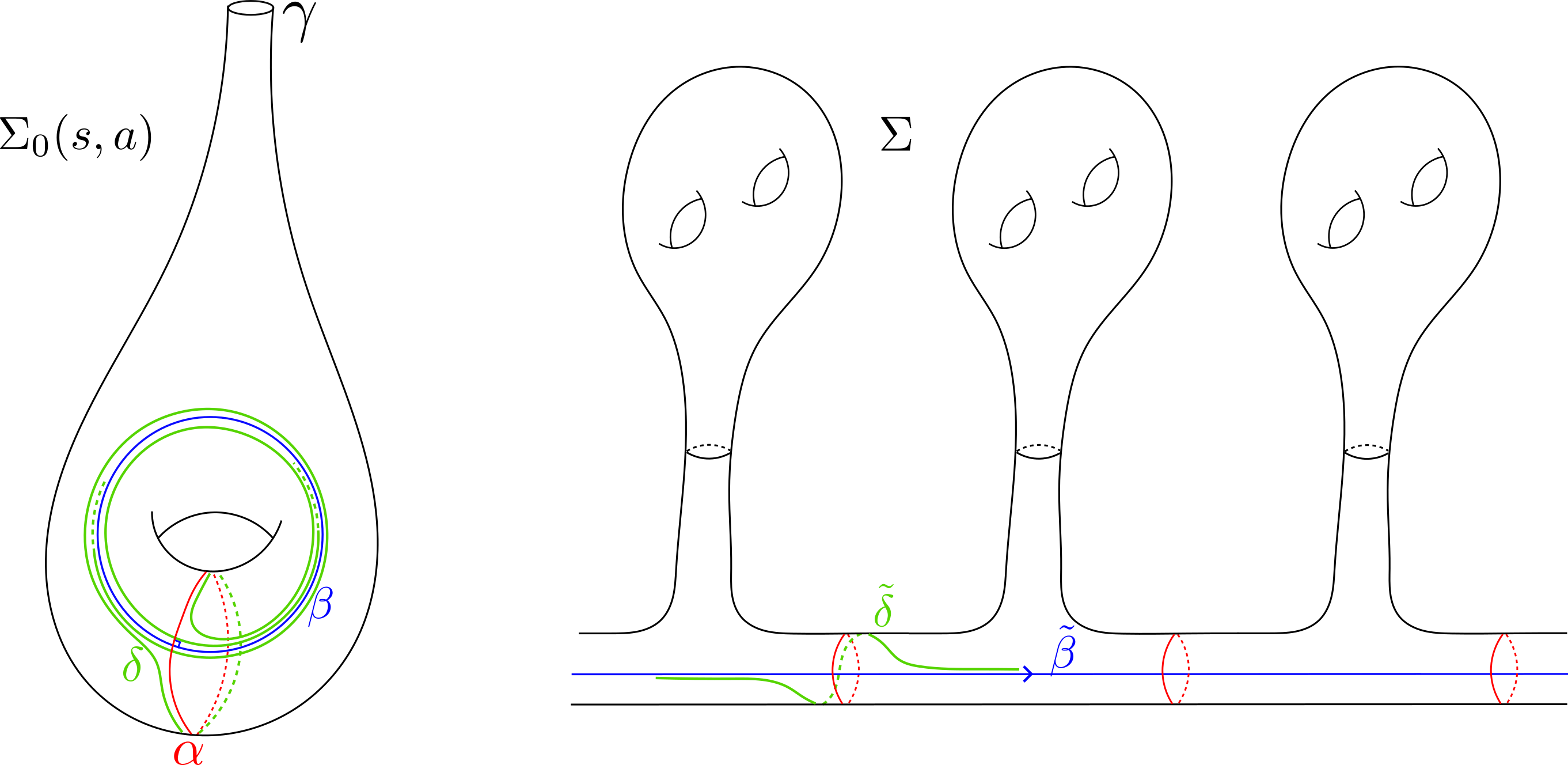}
        \caption{}
        \label{fig:unique_bi_minimizer}
    \end{figure}    

    Using again transitivity of the mapping class group of $S_0$ on primitive classes in $H^1(S_0,c\Z)$ allows us to embed this example in any cohomology class.
    \end{proof}

	\begin{remark}
		Note that whenever $ \Sigma $ contains a single bi-minimizing line, then the set $ \Lambda_{\mathrm{nh}} $ of all endpoints in $ S^1 $ of quasi-minimizing rays contains exactly two $ \Gamma $-orbits and is in particular countable.
	\end{remark}
	
	\section{A uniform Busemann function}\label{Sec:Uniform Busemann function}
	
	Fix a $ 1 $-Lipschitz tight function $ \tau: \Sigma \to \R $ as discussed in \cref{Section_Tight functions}. Let $ B \subset \Sigma $ be a sufficiently large ball satisfying $ \pZ{B}=\Sigma_0 $. Notice that the complement of $ B $ in $ \Sigma $ has two connected components, one corresponding to unbounded positive values of $ \tau $ and the other to unbounded negative values. We shall respectively refer to these two components as the positive and negative ends of $ \Sigma $.
	\medskip
	
	Abusing notation, we shall refer to $ \tau $ as a function defined over the frame bundle via $ \tau \circ p_K $. Consider the function $ \beta_+: \GdmodGamma \to [-\infty,\infty) $ defined by
	\[ \beta_+(x) = \lim_{t \to \infty} \tau(a_t x)-t. \]
	Notice that since $ \tau $ is 1-Lipschitz the function $ \tau(a_t x)-t $ is monotonically decreasing in $ t $. That is,
	\[ \beta_+(x)=\inf_{t \geq 0} (\tau(a_t x)-t) \]
	and in particular $ \beta_+ \leq \tau $.

	\begin{lemma}\label{lemma_properties_beta}
		The function $ \beta_+: \GdmodGamma \to [-\infty,\infty) $ is
		\begin{enumerate}
			\item upper semicontinuous;
			\item $ N $-invariant; and
			\item satisfies for any $ x \in \GdmodGamma $ and $ s \in \R $ 
			\[ \beta_+(a_s x) = \beta_+(x)+s. \]
		\end{enumerate}
	\end{lemma}
	
	\begin{proof}
		The function $ \beta_+ $ is defined as a limit of a decreasing sequence of continuous functions, thus implying (1).
		For (2), note that $ d(a_t nx,a_t x) \to 0 $ as $ t \to \infty $ for any $ x \in \GdmodGamma $ and $ n \in N $, and recall that $ \tau $ is 1-Lipschitz.
		
		Changing variables gives
		\begin{align*}
			\beta_+(a_s x)&= \lim_{t \to \infty} [\tau(a_{t+s} x)-t] = s+ \lim_{t \to \infty} [\tau(a_{t+s} x)-(t+s)] = \\
			&= s+ \beta_+(x),
		\end{align*}
		implying (3).
	\end{proof}
	
	Recall the notations from \cref{Subsec:Tight map summary of notations}. Define
	\[ \Qm_\pm =\{ x \in \Qm : \lim_{t \to \infty} \tau(a_t x) = \pm \infty \}, \]
	and respectively $ \L_\pm=\L \cap \Qm_\pm $.	
	Let $ x $ be in $ \L_+ $, then by definition the geodesic line $ Ax $ is in the maximal stretch locus for $ \tau $ (and is hence bi-minimizing) and facing the positive end. In particular, $ \tau(a_t x)=\tau(x)+t $ for all $ t \in \R $ and thus
	\[ \beta_+(x)= \lim_{t \to \infty} [\tau(x)+t-t] = \tau(x). \]
	More generally, we have the following characterization of $ \cal Q_+ $:
	\begin{lemma}\label{lemma:Qm=beta(R)}
		$ \Qm_+ = \beta_+^{-1}(\R) $.
	\end{lemma}
	
	\begin{proof}
		
		Recall that $ x \in \cal Q $ if and only if there exists $ C>0 $ for which
		\[ d(x,a_t x) \geq t-C \quad\text{for all } t. \]
		On the other hand, $ \beta_+(x) \in \R $ if and only if there exists $ c > 0 $ for which
		\[ \tau(a_t x) \geq t-c \quad\text{for all } t. \]
		By \cref{lemma:u_is_a_1C-qi} we have
		\[ d(x,a_t x)-C_\tau \leq \tau(a_t x)-\tau(x) \leq d(x,a_t x), \]
		and the claim follows.
	\end{proof}
	
	The function $ \beta_+ $ as described above serves as a sort of ``uniform Busemann function'' for all the non-horospherical limit points of $ \Gamma $ associated to the positive end of $ \Sigma $. 
 Denote by
	\begin{equation}\label{eq:definition of betaplus horoball}
	\Horobetaplus{x}:=\beta_+^{-1}([\beta_+(x),\infty))
	\end{equation}
	the \emph{$ \beta_+ $-horoball through $x$}. 
	An immediate consequence of the above lemmata is the following:
	\begin{cor}\label{cor:Nx contained in beta+(x) to infinity}
		For any $ x \in \mathcal{Q}_+ $
		\[ \overline{Nx} \subseteq \Horobetaplus{x}. \]
	\end{cor}
	
	\begin{proof}
		Since $ \beta_+ $ is $ N $-invariant we have $ Nx \subseteq \beta_+^{-1}(\beta_+(x)) $. Upper semicontinuity ensures the set $ \beta_+^{-1}([\beta_+(x),\infty)) $ is closed, implying the claim. 
	\end{proof}
	
	One can similarly define the function $ \beta_-: \GdmodGamma \to (-\infty,\infty] $ by
	\[ \beta_-(x) = \lim_{t \to \infty} \tau(a_t x)+t = \sup_{t \geq 0} (\tau(a_t x)+t) \geq \tau(x), \]
	which would serve as a uniform Busemann function for the negative end of $ \Sigma $. The function $ \beta_- $ is lower semicontinuous with $ \cal Q_-=\beta_-^{-1}(\R) $. Similarly
	\[ \overline{Nx} \subseteq \Horobetaminus{x} \]
	for all $ x \in \cal Q_- $, where $ \Horobetaminus{x}:=\beta_-^{-1}((-\infty,\beta_-(x)]) $.

	\section{The recurrence semigroup}\label{sec:semigroup}
	This section is devoted to studying the sub-invariance properties of non-maximal horospherical orbit closures, the main result of which is:
	\begin{theorem}\label{Thm:Recurrence semigroup in arbitrary d}
		Let $ \Hd / \Gamma $ be any $ \Z $-cover of a compact hyperbolic $ d $-manifold and let $x \in \GdmodGamma $ be any quasi-minimizing point. Then there exists a non-compact, non-discrete closed subsemigroup $ \Delta_x $ of $ MA_{\geq 0} $ under which $ \overline{Nx} $ is strictly sub-invariant, that is, 
		\[ \ell \overline{Nx}\subsetneq \overline{Nx} \quad \text{for all }  \ell \in \Delta_x \smallsetminus \{e\}.\]
	\end{theorem}
	\noindent A special case of this theorem, in dimension $ d=2 $, was presented in the introduction as \Cref{Thm intro: structure of Nx bar}(i).
	
	While our focus is on $ \Z $-covers, much of the tools developed in this section are applicable in greater generality and we encourage the reader to consider arbitrary discrete subgroups as they read through the text. 
	\medskip
	
	Given $ x \in \GdmodGamma $ consider the following subset of $ MA $:
	\[ \Delta_x = \{ \ell \in MA : \ell x \in \overline{Nx} \}. \]
	Note that if $ \ell_1,\ell_2 \in \Delta_x $ then
	\[ \ell_2\overline{Nx} = \overline{N \ell_2 x} \subseteq \overline{Nx}, \]
	where the equality follows from the fact that $ \ell_2 $ normalizes $ N $ and the inclusion follows from $ \overline{Nx} $ being closed and $ N $-invariant. Therefore
	\[ \ell_2\ell_1 x \in \ell_2\overline{Nx} \subseteq \overline{Nx}, \]
	implying:
	\begin{lemma}
		$ \Delta_x $ is a closed subsemigroup of $ MA $.
	\end{lemma} 
	We call $ \Delta_x $ the \emph{recurrence semigroup} of the $ N $-orbit to the ``line'' $ MAx $ (in dimension 2 it is the line $Ax$). Clearly when $ x $ is not quasi-minimizing then $ \Delta_x=MA $, as $ \overline{Nx} = \mathcal{E}_\Gamma $. 	
 	In contrast we have:
	
	\begin{lemma}\label{lemma_positive_recurrence_semigp}
		If $ x \in \GdmodGamma $ is quasi-minimizing then
		\begin{equation}\label{eq_Delta_MA+}
			\Delta_x \subseteq MA_{\geq 0}
		\end{equation}
		where $ A_{\geq 0} = \{a_t : t \geq 0 \} $.
	\end{lemma}
	
	Roughly speaking, the lemma above shows that horospherical orbits at quasi-minimizing points can only accumulate `forward', deeper into their associated end.

	\begin{proof}
		Let $ x=g\Gamma $ and assume in contradiction that $ ma_{-t}x \in \overline{Nx} $ for some $ t>0 $. Since $ \Delta_x $ is a semigroup we have $ m^ka_{-kt}x \in \overline{Nx} $ for all $ k\geq 1 $. This implies in particular that for any $ k $ and any $ \varepsilon > 0 $ there exists $ n \in N $ and $ \gamma \in \Gamma $ satisfying
		\[ h_\varepsilon m^k a_{-kt} g \gamma = ng \]
		for some $ h_\varepsilon \in G_d $ with $ \|h_\varepsilon\|_{G_d}<\varepsilon $. Having $ k $ fixed and choosing $ \varepsilon $ arbitrarily small we may ensure
		\[ h_\varepsilon' g \gamma \in NMa_{kt}g \]
		where $ \|h_\varepsilon'\|_{G_d}<1 $. That is, there exists an element of
                $ g\Gamma $ roughly $ kt>0 $ `deep' inside the horoball bounded by $ Ng $. As
                $ k\geq 1 $ was arbitrary this is in contradiction with the assumption
                that $ g\Gamma $ is quasi-minimizing and hence $ g^+ $ is a
                non-horospherical limit point, see
                \Cref{figure_positive_recurrence_semigp} and recall
                \Cref{lemma:nonhoro-quasimini}.
	\end{proof}

	\begin{figure}[ht]
		\includegraphics[scale=1.4]{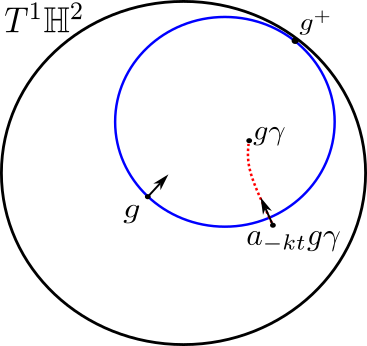}
		\caption{}
		\label{figure_positive_recurrence_semigp}
	\end{figure}

	\begin{remark}
		In the case of a $ \Z $-cover one can deduce \Cref{lemma_positive_recurrence_semigp} directly from \Cref{cor:Nx contained in beta+(x) to infinity} and property (3) of \Cref{lemma_properties_beta}.
	\end{remark}

	A natural question at this point is --- \textit{What is $ \Delta_x $ for $ x $ quasi-minimizing?}
	
	In what follows we provide tools to study this question and examples in which a complete answer can be given. See also Bellis \cite{Bellis} which has studied this question in dimension $ d=2 $.
	\medskip

	A closed subgroup $ \Theta \leq G_d $ is called a \emph{geometric limit} of a sequence of closed subgroups $ \Gamma_j \leq G_d $ if there exists a subsequence $ \Gamma_{j_k} $ which converges to $ \Theta $ in the Hasudorff metric when restricted to any compact subset of $ G_d $ (this is referred to as convergence in the Chabauty topology on the space of closed subgroups of $ G_d $).

        Recalling the Bruhat decomposition in \Cref{eq:Bruhat decomposition of G_d}, 
	we define the following projection $ \delta: G_d \to MA $ by
	\[ \delta(n \ell u)=\ell \]
	for all $ u \in U, \ell \in MA $ and $ n \in N $, and set $ \delta(g)=e $ whenever $ g \notin NMAU $. We have the following:
	
	\begin{lemma}\label{lemma_delta of geometric limit in Delta}
		Given $ x=g\Gamma $, let $ \Theta $ be any geometric limit of the family 
		\[ \{a_{t}g\Gamma g^{-1}a_{-t} \}_{t\geq 0}, \]
		and let $ e \neq \theta \in \Theta \cap NMAU $. Then $ \delta(\theta)x $ is an accumulation point of $ Nx $.
		In particular, $ \delta(\Theta) \subset \Delta_x $. 
	\end{lemma}

	Note that the conjugate $ a_{t}g\Gamma g^{-1}a_{-t} $ is the stabilizer of the point $ a_t x=a_tg\Gamma  $ in $ \GmodGamma $. Therefore, in a sense, geometric limits $ \Theta $ as above capture the asymptotic geometry ``as seen'' along the geodesic ray $ (a_t x)_{t \geq 0} $. 
	One can consider the following closed subset of $ G_d $ associated with the point $ x=g\Gamma $:
	\begin{equation}\label{eq_def_Xix}
	\Xi_x = \bigcap_{k\geq 1} \overline{\bigcup_{t \geq k} a_{t}g\Gamma g^{-1}a_{-t}}.
	\end{equation}
	This is the union of all geometric limits $ \Theta $ as in \cref{lemma_delta of geometric limit in Delta}, hence as a direct consequence we have:
	\begin{cor}\label{lemma_delta_Xi_in_Delta}
		For all $ x \in \GdmodGamma $
		\[ \delta\left( \Xi_x \right) \subseteq \Delta_x. \]
	\end{cor}
		
	\begin{remark}
		See also \cite[Theorem 1.3]{LL} which is somewhat similar in form to the lemma above and deals with horospherical measure rigidity.
	\end{remark}

	\begin{proof}[Proof of \cref{lemma_delta of geometric limit in Delta}]
		Let $ \theta_0=n_0 \ell_0 u_0 $ be an element in $ \Theta \cap NMAU $. Hence there exist sequences $ \gamma_j \in \Gamma, u_j \in U, \ell_j \in MA $ and $ n_j \in N $ satisfying
		\[  a_{t_j}g \gamma_j g^{-1} a_{-t_j} = n_j \ell_j u_j \]
		where $ u_j \to u_0, \ell_j \to \ell_0 $ and $ n_j \to n_0 $. The latter
                convergence relations follow from $ NMAU $ being open in $ G_d $ and the
                fact that the multiplication map $ N \times MA \times U \to NMAU $ is
                a diffeomorphism (see \Cref{subsec:bruhat}).  Therefore, for all $ j $
		\[ g\gamma_j = a_{-t_j} n_j \ell_j u_j a_{t_j} g. \]
		This in turn implies
		\begin{align*}
			g\Gamma &= a_{-t_j}n_j \ell_j u_j a_{t_j}g \Gamma = \\
			&= (a_{-t_j} n_ja_{t_j}) \ell_j (a_{-t_j}u_j a_{t_j})g \Gamma.
		\end{align*}
                This is an instance of the phenomenon described in  \Cref{fig:bruhat-trick}. 
		Since $ t_j \to +\infty $ we are assured that $ a_{-t_j} u_j a_{t_j} \to e $. Denoting 
		\[ \tilde{n}_j= a_{-t_j}n_j^{-1} a_{t_j} \in N \]
		we conclude
		\[ \tilde{n}_j g \Gamma = \ell_j (a_{t_j} u_ja_{-t_j}) g \Gamma \to \ell_0 g\Gamma . \]
		In other words, the $ N $-orbit of $ x $ accumulates onto $ \ell_0 x $ and in particular $  \ell_0 x \in \overline{Nx} $ and $ \ell_0 \in \Delta_x $.
	\end{proof}

	From this simple lemma we draw several useful conclusions:
	\begin{prop}\label{prop:Z-dense geometric limit implies Delta non-compact}
		Let $ \Theta $ be a geometric limit as above and assume that $ \Theta $ is Zariski dense in $ G_d $. Then $ Nx $ has an accumulation point on $ MAx \smallsetminus Mx $ and in particular $ \Delta_x $ is non-compact.
	\end{prop}	
	
	\begin{proof}
	The set $ NMAU \smallsetminus NMU $ is Zariski open. Indeed (see discussion
        in \Cref{subsec:bruhat}), since the multiplication map $ N \times MA \times U \to NMAU $ is an isomorphism of varieties then $ NMU $ is Zariski closed in $ NMAU $.
	Therefore if $ \Theta $ is Zariski dense there must exist an element $ \theta \in \Theta \cap NMAU \smallsetminus NMU $. This directly gives $ \delta(\theta)\notin M $, implying the claim.
	\end{proof}

	More generally, denote by $ I(x) $ the limit inferior of the injectivity radius in $ \GdmodGamma $ along the geodesic ray $ (a_t x)_{t \geq 0} $, that is,
	\[ I(x)=\liminf_{t \to +\infty} \;\mathrm{Inj}_{\GdmodGamma}(a_t x). \] 
	
	\begin{cor}
		Assume $ \Gamma $ is torsion-free. Then if $ I(x)<\infty $ then $ Nx $ has an accumulation point on $ MAx $. Moreover, if $ I(x)=0 $ then $ x $ itself is an accumulation point of $ Nx $.
	\end{cor}

	\begin{remark}
	See \cite[Theorem 1.1]{Bellis} for a different proof of this statement in the case of $ d=2 $. See also \cite{CM}.
	\end{remark}
	
	\begin{proof}                
		The assumption that $ I(g\Gamma) < \infty $ implies there exist $ R>0 $ and $ t_j \to \infty $ for which
		\[ a_{t_j}g \Gamma g^{-1}a_{-t_j} \cap B_R^{G_d} \neq \{e\} \quad \text{for all }j, \]
		where $ B_R^{G_d} $ is a ball of radius $ R $ around the identity in $ G_d $. In particular, there exists an accumulation point $ g_j \to g_0 $ where  $ e \neq g_j \in a_{t_j}g \Gamma g^{-1}a_{-t_j} \cap B_R^{G_d} $.
		Therefore, we conclude $  g_0 \in \Xi_{g \Gamma} $.
		
		If $ e \neq g_0 \in NMAU $ then by \cref{lemma_delta of geometric limit in Delta} we conclude $ Nx $ accumulates onto $ \delta(g_0)x $. If $ g_j \to e $ then for any $ \varepsilon > 0 $ we may replace $ g_j $ by some power $ g^m_j \in a_{t_j}g \Gamma g^{-1}a_{-t_j} \cap B_R^{G_d} $ satisfying $ \frac{1}{2}\varepsilon \leq \|g_j^m\|_{G_d} < \varepsilon $ (recall that since $ \Gamma $ is torsion-free, all $ g_j $ generate a non-compact subgroup in $ G_d $). Hence $ \Xi_{g \Gamma} \cap B_\varepsilon^{G_d} \smallsetminus B_{\frac{1}{2}\varepsilon}^{G_d} \neq \emptyset $ for all small enough $ \varepsilon $, implying $ Nx $ accumulates onto $ x $.
		Notice that if $ I(g\Gamma) = 0 $ then there exist $ g_j \to e $ as above.
		
		In the remaining case where $ g_0 \in NMA\omega $ we notice that
		\[ g_0^2 \in NMA\omega NMA \omega = NMAU \]
		and so the above argument shows $ Nx $ accumulates onto $ \delta(g_0^2)x $. 
	\end{proof}

	\begin{cor}\label{cor_positivity of geometric limits}
		If $ x \in \GdmodGamma $ is quasi-minimizing and $ \Theta $ is a geometric limit of $ \{a_{t}g\Gamma g^{-1}a_{-t} \}_{t\geq 0} $, then
		\[ \Theta \cap NMA_{< 0} U = \emptyset \]
		where $ A_{< 0} = \{ a_t : t < 0 \} $.
	\end{cor}

	\begin{proof}
		By \cref{lemma_positive_recurrence_semigp} we know that $ \Delta_x \cap MA_{<0} = \emptyset $. \Cref{lemma_delta of geometric limit in Delta} concludes the proof.
	\end{proof}

	Recall the definition of $ \Xi_x $ in \eqref{eq_def_Xix}, for $ x=g\Gamma $. In
        some cases the $ \delta $ projection of $ \Xi_x $ completely determines $ \Delta_x
        $. One such immediate case is whenever $ \delta(\Xi_x) $ generates, as a closed
        semigroup, all of $ MA_{\geq 0} $. For instance for $ d=2 $, whenever $ e $ is not an isolated point in $ \delta(\Xi_x) $ one can deduce $ \Delta_x=A_{\geq 0} $ for $ x $ quasi-minimizing.
	
	Another such case is the following:
	\begin{lemma}\label{lemma_basepoint_conj_in_Xi_determines_Delta}
		Let $ x=g\Gamma $ and assume that $ g\Gamma g^{-1} \subset \Xi_x $. Then 
		\[ \overline{\delta(g\Gamma g^{-1})} = \Delta_x \]
		where the closure is taken in $ MA $.
	\end{lemma}
	
	\begin{remark}
		Note that the hypothesis of this lemma holds for instance whenever the geodesic ray $ (a_t g\Gamma)_{t \geq 0} $ in $ \GdmodGamma $ has $ x=g\Gamma $ as an accumulation point. In such a case $ x $ is not quasi-minimizing in $ \GdmodGamma $ and $ \overline{Nx}=\GdmodGamma $ (assuming $ \Gamma $ was Zariski-dense) implying in particular that $ \Delta_x=MA $. Hence we conclude that $ g\Gamma g^{-1} $ has a dense $ \delta $-projection into $ MA $.
	\end{remark}

	\begin{proof}
		The inclusion $ (\subseteq) $ is immediate from \cref{lemma_delta_Xi_in_Delta} and the fact that $ \Delta_x $ is closed in $ MA $. For the other direction, let $ \ell_0 \in \Delta_x $, then by definition there exist sequences $ \tilde{n}_j \in N $, $ \gamma_j \in \Gamma $ and $ \epsilon_j \in G $ with
		\[ \tilde{n}_j g = \epsilon_j \ell_0 g \gamma_j \]
		and $ \epsilon_j \to e $. Recall that the set $ NMAU $ is an open neighborhood of the identity in $ G_d $, therefore for all large $ j $ we may write $ \epsilon_j=n_j \ell_j u_j $. Thus
		\[ \tilde{n}_jg = n_j \ell_j u_j \ell_0 g \gamma_j \]
		and
		\[ g \gamma_j^{-1} g^{-1} = (\tilde{n}_j^{-1}n_j)(\ell_j \ell_0)u_j. \]
		The right hand side above is written in $ NMAU $ form with $ \delta $ projection equal $ \ell_0 \ell_j $. In other words we get $ \ell_0 \ell_j \in \delta(g\Gamma g^{-1}) $ and since $ \ell_j \to e $, the claim follows.
	\end{proof}
	
	\begin{remark}
		As mentioned the entire discussion above holds just as well for any discrete group $ \Gamma < G_d $, not necessarily a normal subgroup of a uniform lattice. We note that actually much of the discussion extends even further to other ambient groups. Let $ G $ be any connected semisimple real algebraic group $ G $ (of arbitrary rank) and let $ \Gamma < G $ be a discrete subgroup. Let $ P $ be a minimal parabolic subgroup with Langlands decomposition $ P=MAN $ where $ A $ is the Cartan subgroup of $ \R $-diagonalizable elements, $ N $ a contracting horospherical subgroup with respect to a choice of Weyl chamber $ \frak a^+ $, and $ M $ the compact centralizer of $ A $. Given $ x \in \GmodGamma $, one may define $ \Delta_x $ as above and the map $ \delta $ as a projection from the open Bruhat cell $ NMAU $ (where $ U $ is the corresponding expanding horospherical subgroup). The proof of \cref{lemma_delta of geometric limit in Delta} holds almost verbatim for any geometric limit of $ \exp(v_j)g\Gamma g^{-1}\exp(-v_j) $ (i.e.~an accumulation point in the Chabauty topology) for a regular sequence $ v_j \to \infty $ in $ \frak a^+ $.
		One can also consider a form of \cref{cor_positivity of geometric limits} with respect to a higher rank notion of horospherical limit point in the Furstenberg boundary $ P\backslash G $, see \cite{LO}.
	\end{remark}
	
	\subsection{Returning to $ \Z $-covers}
	
	Let us apply what we have learned to the setting of this paper.
	Assume from this point on that $ \Gamma $ is a normal subgroup of a uniform lattice $ \Gamma_0 < G_d $ with $ \Gamma_0/\Gamma \cong \Z $.

	\begin{lemma}\label{lemma:stabilizer of accum point downstairs is a geom limit}
		If $ a_{t_j}\pZ{x} \to h_0\Gamma_0 \in \GdmodGamma_0 $ for some $ x \in \GdmodGamma $ and $ t_j \to \infty $, then $ h_0\Gamma h_0^{-1} \subset \Xi_x $.
	\end{lemma}

	\begin{proof}
In the language of the Gromov-Hausdorff topology on manifolds with basepoints, this is what happens for any regular cover $\pi:M\to M_0$ of Riemannian manifolds: For any element $\gamma$ of the deck group and $x\in M$, the pair $(M,x)$ and $(M,\gamma x)$ are isometric. If $x_i\in M$ is a sequence such that $\pi(x_i)\to y\in M_0$, then up to the action of the deck group we can assume $x_i\to x$ with $\pi(x)=y$, and so $(M,x_i)$ converges to $(M,x)$. 

 In our more group-theoretic language, denote $ x=g\Gamma $ and $ \pZ{x}=g\Gamma_0 $. Since $ a_{t_j} g \Gammazero \to h_0 \Gammazero $ in $ \GdmodGammazero $, then there exists a sequence of elements $ \eta_j \in \Gammazero $ satisfying $ a_{t_j}g\eta_j \to h_0 $. This in turn implies the following geometric convergence of subgroups
		\[ a_{t_j} g \eta_j \Gamma \eta^{-1}_j g^{-1} a_{-t_j} \to h_0 \Gamma h_0^{-1}. \]
		But as $ \Gamma \lhd \Gammazero $ we have $ \eta_j \Gamma \eta^{-1}_j = \Gamma $, implying that $ h_0\Gamma h_0^{-1} $ is a geometric limit of $ \{a_{t_j} g \Gamma g^{-1} a_{-t_j}\}_{j \geq 1} $ and hence contained in $ \Xi_x $.
	\end{proof}

	In certain cases, one geometric limit completely determines the semigroup:
	
	\begin{prop}\label{prop:lift of recurrent point has Delta=delta of conjugate}
		If $ \pZ{x} $ is an accumulation point of the geodesic ray $ (a_t \pZ{x})_{t \geq 0} $ in $ \GdmodGammazero $, then
		$\overline{\delta(g\Gamma g^{-1})} = \Delta_x$.
	\end{prop}
	
	\begin{proof}
		\Cref{lemma:stabilizer of accum point downstairs is a geom limit} together with \Cref{lemma_basepoint_conj_in_Xi_determines_Delta} imply the claim.
	\end{proof}
	
	Note that the conditions of this proposition are satisfied for example whenever $ \pZ{x} $ lies on a closed geodesic in $ \GdmodGammazero $. Recall the definitions and notations from \cref{Subsec:Tight map summary of notations}.
	\medskip
	
	\begin{prop}\label{cor:minimal components have constant Deltas}
		If $ \mu $ is a minimal component of $ \lambda_0 $ and
		\[ \L_\mu=p_\Z^{-1}(p_K^{-1}(\mu)) \cap \L \]
		is the set of tangent frames to its lift in the $ \Z $-cover, then the function $ x \mapsto \Delta_x $ is constant on $ \L_\mu $.
	\end{prop}

	\begin{proof}
		Let $ g_1\Gamma, g_2\Gamma \in \L_\mu $ be any two points. By definition, since the geodesic flow along $ \mu $ is minimal, the ray $ (a_tg_1 \Gammazero)_{t \geq 0} $ accumulates onto $ g_2 \Gammazero $. Hence $ g_2 \Gamma g_2^{-1} \subset \Xi_{g_1 \Gamma} $ and by \cref{lemma_delta of geometric limit in Delta}
		\[ \delta(g_2 \Gamma g_2^{-1}) \subset \Delta_{g_1 \Gamma}. \]
		
		On the other hand, the ray $ (a_tg_2 \Gammazero)_{t \geq 0} $ also accumulates onto $ g_2 \Gammazero $ and hence by \Cref{prop:lift of recurrent point has Delta=delta of conjugate}
		we get
		\[ \Delta_{g_2 \Gamma} = \overline{\delta(g_2 \Gamma g_2^{-1})}. \]
		Since $ \Delta_{g_1 \Gamma} $ is closed we deduce $ \Delta_{g_2 \Gamma} \subseteq \Delta_{g_1 \Gamma} $. By symmetry, the claim follows.
	\end{proof}
	\medskip
	
	Let us now establish a few basic properties of these semigroups:
	\begin{prop}\label{prop:covers_have_Delta_non-trivial}
		The semigroup $ \Delta_x $ is non-compact for all $ x \in \GdmodGamma $.
	\end{prop}
	
	\begin{proof}
		Since $ \GdmodGamma_0 $ is compact, for any $ x \in \GdmodGamma $ there exists a point $ h_0\Gamma_0 $ which is an accumulation point of $ (a_t \pZ{x})_{t \geq 0} $. This implies that there exists a Zariski dense discrete subgroup $ h_0 \Gamma h_0^{-1} $ contained in $ \Xi_x $. By \cref{prop:Z-dense geometric limit implies Delta non-compact} the claim follows.
	\end{proof}
	
	\begin{rmk}
		We deduce in particular that the horospherical flow on a regular cover of a compact hyperbolic $ d $-manifold has no minimal components. Indeed, given $ x \in \GdmodGamma $ the set $ F=\overline{Nx} $ is not $ N $-minimal. If $ F=\GdmodGamma $ then the existence of a non-horospherical limit point in a geometrically infinite manifold implies that there exists some $ y \in F $ with $ \overline{Ny} \neq F $. Otherwise, by \cref{lemma_positive_recurrence_semigp} we know that $ \Delta_x \cap MA_{<0} = \emptyset $. On the other hand \cref{prop:covers_have_Delta_non-trivial} implies that $ ma_t \in \Delta_x $ for some $ m \in M $ and $ t > 0 $. Consider $ y=ma_tx \in F $. Using \cref{lemma_positive_recurrence_semigp} once more we conclude that $ x \notin \overline{Ny} $, implying that $ F $ is not minimal.
		
		One can formulate weaker conditions under which the conclusion of the corollary above holds. See \cite[Corollary 1.2]{Bellis} for a much general statement in $ d=2 $.
	\end{rmk}
	\medskip
	
	Given a set $ D \subset MA $ denote by $ D/M $ its projection onto $ A \cong MA/M $. Since $ M $ and $ A $ commute, this projection is a group homomorphism.
	\begin{prop}\label{prop:Delta_x/M non discrete}
		The $ A $-component $ \Delta_x/M $ of the recurrence semigroup is non-discrete for all $ x \in \GdmodGamma $.
	\end{prop}
	
	\begin{proof}
		Denote $ x=g\Gamma $ and consider the set of accumulation points in $ \GdmodGamma_0 $ of $ (a_t g\Gamma_0)_{t \geq 0} $. Since $ \GdmodGamma_0 $ is compact and the above set of accumulation points is closed and $ A $-invariant, it contains an $ A $-minimal subset. In particular, the ray $ (a_t g\Gamma_0)_{t \geq 0} $ accumulates onto a point $ h_0\Gamma_0 $ satisfying that
		\begin{equation}\label{eq:h_0 self recurrent in lemma that Delta is non-discrete}
			a_{t_j}h_0 \Gamma_0 \to h_0 \Gamma_0 \quad \text{for some } t_j \to \infty.
		\end{equation}
		\Cref{lemma:stabilizer of accum point downstairs is a geom limit} implies that $ h_0 \Gamma h_0^{-1} \subset \Xi_x $ and hence $ \delta(h_0 \Gamma h_0^{-1}) \subset \Delta_x $, by \Cref{lemma_delta of geometric limit in Delta}.
		We will show that the projected set $ \delta(h_0 \Gamma h_0^{-1})/M $ is non-discrete.
		\medskip
		
		The group $ h_0\Gamma h_0^{-1} $ is Zariski-dense and hence contains an element
		\[ \gamma_0 \in h_0\Gamma h_0^{-1} \cap NMAU \smallsetminus \left(NMA \cup MAU \cup NU \right). \]
		In other words, $ \gamma_0 = n_0\ell_0 u_0 $ for some $ n_0 \in N $, $ u_0 \in U $, and $ \ell_0 \in MA $, all not equal to the identity. 
		
		By \eqref{eq:h_0 self recurrent in lemma that Delta is non-discrete}, we have the following convergence
		\[ a_{t_j}h_0 \Gamma h_0^{-1}a_{-t_j} \to h_0 \Gamma h_0^{-1}, \]
		in the sense of Hausdorff convergence on compact subsets.
		In particular, there exists a sequence $ \gamma_j \in h_0 \Gamma h_0^{-1} $ satisfying
		\begin{equation}\label{eq:non-discreteness proof limit of gamma_js}
			a_{t_j}\gamma_j a_{-t_j} \to \gamma_0.
		\end{equation}
		
		The multiplication map $ N \times MA \times U \to NMAU $ is a diffeomorphism implying by \eqref{eq:non-discreteness proof limit of gamma_js} that there exist
		$ n_j \in N, \ell_j \in MA, $ and $ u_j \in U $ satisfying
		\[ a_{t_j}\gamma_j a_{-t_j} = n_j \ell_j u_j \quad \text{for all } j \]
		with
		\[ n_j \to n_0 \quad,\quad \ell_j \to \ell_0 \quad\text{and} \quad u_j \to u_0. \]
		In particular
		\[ \gamma_j = a_{-t_j}n_j\ell_j u_j a_{t_j} = (a_{-t_j}n_j a_{t_j})\ell_j (a_{-t_j}u_j a_{t_j}) \in h_0 \Gamma h_0^{-1}. \]
		
		Note that $ \delta(\gamma_j)=\ell_j $. If $ \ell_j/M \neq \ell_0/M $ for some infinite subsequence of $ j $'s, this would immediately imply the non-discreteness of $ \delta(h_0 \Gamma h_0^{-1})/M $, as claimed.
		\medskip
		
		Assume therefore the converse, that $ \ell_j/M=\ell_0/M $ for all but finitely many $ j $'s. Now consider the sequence of elements $ \gamma_j \gamma_0 \in h_0 \Gamma h_0^{-1} $. We claim that
		\begin{claim*} $ \delta(\gamma_j \gamma) \to \ell_0^2 \quad \text{but}\quad \delta(\gamma_j \gamma)/M\neq \ell_0^2/M \quad \text{for all large } j. $
		\end{claim*}
		\noindent This would conclude the proof of the proposition.
		\medskip
		
		Consider 
		\[ \gamma_j \gamma = (a_{-t_j}n_j a_{t_j})\ell_j (a_{-t_j}u_j a_{t_j}) n_0 \ell_0 u_0, \]
		and note that by the definition of the map $ \delta $ we have
		\[ \delta(\gamma_j \gamma) = \delta(\ell_j (a_{-t_j}u_j a_{t_j}) n_0 \ell_0).  \]
		As $ j \to \infty $ we have $ a_{-t_j}u_j a_{t_j} \to e $ and $ \ell_j \to \ell_0 $, hence 
		\[ \ell_j (a_{-t_j}u_j a_{t_j}) n_0 \ell_0 \to \ell_0 n_0 \ell_0 \in NMA. \]
		Since $ NMAU $ is an open set containing $ n_0 $, we are ensured that the element $ (a_{-t_j}u_j a_{t_j}) n_0 $ may be written as
		\[ (a_{-t_j}u_j a_{t_j}) n_0=n'_j \ell'_j u'_j \quad \text{for all large }j. \]
		In particular,
		\begin{align*}
		\ell_j ((a_{-t_j}u_j a_{t_j}) n_0) \ell_0 &= \ell_j  (n'_j \ell'_j u'_j) \ell_0 =\\
		&= (\ell_j n'_j \ell_j^{-1}) \ell_j \ell'_j \ell_0 (\ell_0^{-1} u'_j \ell_0).
		\end{align*}
		Since $ \ell_j n'_j \ell_j^{-1} \in N $ and $ \ell_0^{-1} u'_j \ell_0 \in U $ we have
		\[ \delta(\gamma_j \gamma_0)=\delta(\ell_j (a_{-t_j}u_j a_{t_j}) n_0 \ell_0)=\ell_j \ell'_j \ell_0. \]
		Since $ n'_j\ell'_ju'_j \to n_0 $, we know $ \ell'_j \to e $ implying that $ \delta(\gamma_j \gamma) \to \ell_0^2 $.
		\medskip
		
		Showing that $ \delta(\gamma_j \gamma)/M\neq \ell_0^2/M $ amounts to showing 
		\[ \ell'_j=\delta((a_{-t_j}u_j a_{t_j}) n_0) \notin M, \quad \text{for all large }j. \]
		Since $ (a_{-t_j}u_j a_{t_j}) n_0 \in NMAU $ for all large $ j $ it suffices to show that
		\[ (a_{-t_j}u_j a_{t_j}) n_0 \notin NMU, \quad \text{for all large }j. \]
		We will do so by using an explicit representation of $ G_d $.
		\medskip
		
		Let us consider the matrix representation of $ G_d=\SO^+(d,1) $ as the identity component of the group of $ (d+1)\times(d+1) $ matrices preserving the quadratic form $ -2x_0 x_d + \sum_{i=1}^{d-1} x_i^2 $. Under this representation we have the following subgroups
		\[ MA= \left\{ \begin{bmatrix}
		e^{t} & & \\ & m & \\ & & e^{-t}
		\end{bmatrix} \;:\; m \in \SO(d-1) \quad \text{and} \quad t \in \R  \right\}, \]
		\[ N=\left\{ n(s) = \begin{bmatrix}
		1 & & \\ s & I_{d-1} & \\ \frac{1}{2}\|s\|^2 & s^{t} & 1
		\end{bmatrix} \;:\; s \in \R^{d-1} \;\;\text{a column vector} \right\}, \]
		where $ \|\cdot\| $ denotes the Euclidean norm on $ \R^{d-1} $, and
		\[ U = \left\{ u(s) = n(s)^t \;:\; s \in \R^{d-1} \right\}. \]
		
		Returning to our previous notations of $ n_0 \in N $ and $ u_j \to u_0 $ in $ U $, set non-zero vectors $ s_0,r_0 $ and $ r_j $ in $ \R^{d-1} $ satisfying
		\[ n_0=n(s_0) \quad,\quad u_0=u(r_0) \quad \text{and}\quad u_j=u(r_j). \]
		Recall that
		\[ a_{-t_j}u(r_j)a_{t_j} = u(e^{-t_j}r_j) \quad \text{for all }j. \]

		Now assume in contradiction that for an unbounded sequence of $ j $'s
		\begin{equation}\label{eq: lemma non-discrete Delta - un in NMN+}
		u(e^{-t_j}r_j) n(s_0) = n(v_j) \begin{bmatrix}
		1 & & \\ & m_j & \\ & & 1
		\end{bmatrix} u(w_j) \in NMU,
		\end{equation}
		for some sequence of vectors $ v_j,w_j \in \R^{d-1} $ and matrices $ m_j \in \SO(d-1) $.
		
		The $ 11 $-coordinate of the left-hand side of equation \eqref{eq: lemma non-discrete Delta - un in NMN+} is
		\[ \left( u(e^{-t_j}r_j) n(s_0) \right)_{11} = 1 + e^{-t_j}r_j \cdot s_0 + \frac{1}{4}\|e^{-t_j}r_j\|^2 \|s_0\|^2, \]
		where $ \cdot $ denotes the regular dot product in $ \R^{d-1} $. On the other hand, the $ 11 $-coordinate of the right-hand side of the equation is $ 1 $. Hence we get
		\begin{equation}\label{eq: lemma non discrete Delta - dot product vs norms squared}
		e^{-t_j}r_j \cdot s_0 = -\frac{1}{4}\|e^{-t_j}r_j\|^2 \|s_0\|^2
		\end{equation}
		But since $ r_j \to r_0 \neq 0 $ we see that
		\[ e^{-t_j}r_j \cdot s_0 \sim e^{-t_j} \]
		while
		\[ \frac{1}{4}\|e^{-t_j}r_j\|^2 \|s_0\|^2 \sim e^{-2t_j},  \]
		making the equality in \eqref{eq: lemma non discrete Delta - dot product vs norms squared} impossible for large enough $ j $.
	\end{proof}

	\begin{remark}
		Note that one should not expect a general claim of the sort that the $ \delta $ projection of any Zariski-dense group in $ G_d $ is non-discrete. Indeed, a counter example in $ \PSL_2(\R) $ is
		\[ \delta(\PSL_2(\Z)) = \left\{\pm \begin{bmatrix}
		k & \\ & k^{-1}
		\end{bmatrix} : k \in \Z \right\} \]
		which is discrete.
	\end{remark}
	\medskip
	
	While $ \Delta_x $ in a $ \Z $-cover is never discrete, under certain geometric conditions the point $ e \in \Delta_x $ is isolated. Recall the definition of $ \L \subset \GdmodGamma $ from \Cref{Subsec:Tight map summary of notations}.
	\begin{lemma}
		For any $ \rho>0 $ there exists $ \varepsilon >0 $ such that if $ \tau(x)-\beta_+(x) < \varepsilon $ then $ (a_t x)_{t\geq 0} \subset \L^{(\rho)} $, the $ \rho $-neighborhood of $ \L $.
	\end{lemma}

	\begin{proof}
		Consider the function $ \varphi_+(y)=\tau(y)-\beta_+(y) \geq 0 $. Recall that $ \beta_+(y) \leq \tau(y) $ for all $ y \in \GmodGamma $ and that $ \beta_+(y)=\tau(y) $ if and only if $ y \in \L $. In particular, $ \varphi_+^{-1}(0)=\L $. The function $ \varphi_+ $ roughly measures how much ``spare time'' the trajectory $ (a_t y)_{t\geq 0} $ has to waste during its voyage into the positive end of $ \GmodGamma $.
		Since $ \beta_+ $ is upper semi-continuous, $ \varphi_+ $ is lower semi-continuous. Therefore the set $ \varphi_+^{-1}([0,\eta]) $ is closed for any $ \eta > 0 $, and
		\[ F_{\eta}=p_\Z \left( \varphi_+^{-1}([0,\eta]) \right) \]
		is compact (the function $ \varphi_+ $ is invariant under the deck transformation because $ \beta_+(k.y)=(\tau(k.y)-\tau(y))+\beta_+(y) $). 
		
		Since
		\[ \bigcap_{\eta > 0} F_{\eta} = p_\Z(\varphi_+^{-1}(0))=p_\Z(\L)=\L_0 \]
		we conclude that for any $\rho>0$ there exists  $ \varepsilon>0 $ for which 
		$ F_\varepsilon \subset \L_0^{(\rho)} $, implying
		\[ \varphi_+^{-1}([0,\varepsilon]) \subset \L^{(\rho)}. \]
		
		The function $ \varphi_+(a_t x) $ is monotonically decreasing in $ t $, for all $ x \in \GmodGamma $ by the $ 1 $-Lipschitz of $ \tau $ and item (2) of \cref{lemma_properties_beta}. Therefore if the point $ x $ satisfies $ \tau(x)-\beta_+(x)<\varepsilon $ then the entire ray admits $ (a_t x)_{t \geq 0} \subset F_\varepsilon \subset \L^{(\rho)} $, as claimed.
	\end{proof}
	
	\begin{cor}\label{cor:isolated leaf implies isolated e in Delta}
		If $ x $ lies on a uniformly isolated ray in $ \L $, that is
		\[ (Ax)^{(\rho)} \cap \L=Ax \quad \text{ for some } \rho > 0, \]
		then $ p_K(Nx) $ does not accumulate onto $ p_K(x) $ in $ \Sigma $. In particular, there exists an $ \varepsilon > 0 $ with $ \Delta_x \cap MA_{[0,\varepsilon)}=\{e\} $.
	\end{cor}

	\begin{proof}
		Assume without loss of generality that $ \rho $ is smaller than half the infimal injectivity radius in $ \GdmodGamma $ and
		\[ \rho < \frac{1}{2}d_{\GdmodGamma}(Ax,\L \smallsetminus Ax). \]
		Let $ \varepsilon > 0 $ be the constant corresponding to $ \rho $ as per the previous lemma.
		
		Now assume in contradiction that $ Nx $ accumulates onto $ x' \in \GdmodGamma $ where $ p_K(x')=p_K(x) $. Recall that $ \beta_+(x)=\tau(x) $, since $ x \in \L $, and that $ \beta_+(nx)=\beta_+(x) $ for all $ n \in N $. Furthermore, since $ \tau $ is continuous, if $ n_j x \to x' $ then $ \tau(n_j x) \to \tau(x') $. Since $ \tau $ depends only on basepoints we further have $ \tau(x')=\tau(x) $. This in turn implies that for all large $ j $ we have $ \tau(n_j x) - \beta_+(n_j x) < \varepsilon $. By the previous lemma, we conclude that $ (a_t n_j x)_{t \geq 0} $ is contained in $ (Ax)^{(\rho)} $. But $ (a_t n_j x)_{t \geq 0} $ and $ (a_t x)_{t \geq 0} $ are asymptotic and $ \rho $ is smaller than half the injectivity radius, implying that $ n_j x=n' x $ for some small $ n' \in N $, with $ \|n\|<\rho $. In other words, $ \Gamma $ contains a parabolic element, in contradiction.
	\end{proof}

	\subsection{The case of a single bi-minimizing line}
	We will end this section with a description of the horocycle orbit closures in surfaces such as the ones constructed in \Cref{Subsection:single bi-minimizing ray}.
	
	\begin{theorem}
		Let $ \Sigma = \GmodGamma $ be a $ \Z $-cover of a compact hyperbolic surface $ \Sigma_0 $, which contains a single bi-minimizing geodesic as its lamination $ \lambda $ and where $ \lambda_0 $ is a closed geodesic in $ \Sigma_0 $. Let $ y_0=h_0 \Gamma \in \L=T^1\lambda $ be any bi-minimizing point, then
		\[ \overline{N y_0} = N \Delta y_0 \]
		where $ \Delta = \overline{\delta(h_0\Gamma h_0^{-1})} $, is non-trivial and non-discrete and the horocycle $ p_K(Ny_0) \subset \Sigma $ does not accumulate onto $ p_K(y_0) $.
		
		Furthermore, all non-maximal horocycle orbit closures in $ \GmodGamma $ are $ A $-translates of $ \overline{N y_0} $ or $ \overline{N \omega y_0} $, where $ \omega y_0 $ is the opposite vector to $ y_0 $ in $ T^1\lambda $.
	\end{theorem}

	\begin{proof}
		Assume without loss of generality that $ y_0 $ is facing the positive end. Since $ \L $ contains exactly one geodesic, we conclude that all quasi-minimizing rays in $ \Qm_+ $ are asymptotic to this unique bi-minimizing trajectory. In other words, $ \Qm_+ = NAy_0 $ for the point $ y_0 $ as described above. Therefore
		\[ \overline{Ny_0} = N\Delta_{y_0}y_0. \]
		As $ y_0 $ is a lift of a point on a closed geodesic in $ p_\Z (G / \Gamma) $, we have by \cref{prop:lift of recurrent point has Delta=delta of conjugate} that
		\[ \Delta_{y_0} = \overline{\delta(h_0\Gamma h_0^{-1})}. \]
		In addition, since $ \L $ contains a single, clearly isolated, leaf we deduce from \cref{cor:isolated leaf implies isolated e in Delta} that $ p_K(Ny_0) $ does not accumulate onto $ p_K(y_0) $. 
		
		Now given any $ x \in \Qm_+ $, we have $ Nx \cap Ay_0 = \{ a_b y_0 \} $ with
		\[ b = \beta_+(x) -\beta_+(y_0),  \] 
		by \Cref{lemma_properties_beta}. Therefore $ \overline{Nx}=\overline{Na_b y_0}=a_b \overline{Ny_0} $. An analogous claim holds for $ \Qm_- $ with respect to $ \omega y_0 $.
	\end{proof}
	
	\section{The proximality relation}
	In the previous section we have studied the intersection of horospherical orbit closures with the geodesic from which they emanate. In this section we study how the horospherical orbit accumulates onto other geodesic rays, showing how this is affected by the way the associated quasi-minimizing rays ``fellow travel''. 
	\medskip
	
	Let $ \Gamma'<G_d $ be any Zariski-dense discrete group. The notion of proximality is standard in the theory of topological dynamical systems. We recall its definition in our context:
	\begin{definition}
		Two points $ x,y \in \GdmodGamma' $ are called \emph{proximal} if 
		\[ \liminf_{t \to \infty} d(a_t x, a_t y) = 0. \]
	\end{definition}
	
	For example, any two asymptotic points are proximal. On the other hand, for any $ x \in \Qm $, the points $ x $ and $ a_sx $ are \emph{not} proximal for any $ s \neq 0 $.	
	
	We begin with the following proposition:
	\begin{prop}\label{prop:fellow travelers in horo closure}
		Let $ \Gamma' < G_d $ be any discrete group and let $ x $ and $y $ be two quasi-minimizing points in $ \GdmodGamma' $ satisfying
		\[ D=\liminf_{t \to \infty} d_{\GdmodGamma'}(a_tx,a_ty) < \infty. \]
		Then there exists $ \ell_0 \in MA $ for which $ \ell_0 y \in \overline{Nx} $. Moreover, if $ D = 0 $, that is if $ x $ and $ y $ are proximal, then $ y \in \overline{Nx} $.
	\end{prop}

	\begin{proof}
		Let $ x,y \in \GdmodGamma' $ and $ D \geq 0 $ be as above. By definition, there exist $ t_j \to \infty$ and $ h_j \in G_d $ satisfying
		\[ a_{t_j} x = h_j a_{t_j}y, \]
		where $ \|h_j\|_{G_d} \to D $. As all $ h_j $ are bounded we may further assume $ h_j \to h_0 $ with $ \|h_0\|_{G_d}=D $.
		
		We will later show that $ h_0 = n_0 \ell_0 u_0 \in NMA\ U $. Under this assumption the proof is quite similar to the proof of \cref{lemma_delta of geometric limit in Delta}. There exist
		\[  n_j \to n_0, \quad \ell_j \to \ell_0, \quad \text{and}\quad  u_j \to u_0 \]
		in $ N,MA $ and $ U $ respectively for which $ h_j=n_j \ell_j u_j $.
		Hence we have
		\[ a_{t_j} x = n_j \ell_j u_j a_{t_j}y \]
		for all $ j $, and consequently
		\[ (a_{-t_j}n^{-1}_j a_{t_j})x=\ell_j (a_{-t_j}u_j a_{t_j}) y \to \ell_0 y. \]
		Since $ a_{-t_j}n^{-1}_j a_{t_j} \in N $ for all $ j $ we deduce $ \ell_0 y \in \overline{Nx} $.
		
		Whenever $ D=0 $ we have $ h_0 = \ell_0 = e $ and hence $ y \in \overline{Nx} $.
		\medskip

        It is left to show $ h_0 \in NMAU $. The idea is the following, assume in contradiction that  \[
		h_0 \in G_d \smallsetminus NMAU = NMA\omega,  \]
see \eqref{eq:Bruhat decomposition of G_d}. Recall $ a_{-s} \omega = \omega a_{s} $, which
means that $h_0 a_t y$ has the property that	its forward endpoint is	equal to the backward endpoint of $y$.
Roughly this	means that if $a_tx$ is	close to $h_0 a_t y$ then, after
flowing forward	a moderately long time $L$,  $a_{t+L} x$ will be close to $a_{t-L}y$, see \Cref{fig:Fellow travelers}. On the
other hand, the	line $\{a_t y\}$ is quasi-minimizing, which means that for all times
$s>t+L$, $a_s y$ will be ahead of $a_s x$ by a margin of roughly $2L$. Choosing	$L$ large
enough relative to $D$, this contradicts the hypothesis.

\begin{figure}[ht]
	\includegraphics[scale=1]{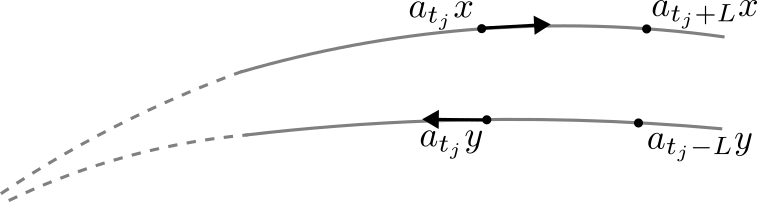}
	\caption{}
	\label{fig:Fellow travelers}
\end{figure}

  More precisely, denote $ h_0=n_0\ell_0 \omega $ with $ n_0 \in N, \ell_o \in MA $ and $ \epsilon_j = h_j h_0^{-1} \to e $. Let $ L > 0 $ be arbitrarily large, then for all sufficiently large $ j $ we have $ \|a_{L}\epsilon_j a_{-L}\| \leq 1 $. Compute
		\begin{align*}
			a_{t_j}x &= \epsilon_j n_0 \ell_0 \omega a_{t_j}y = \epsilon_j n_0 \ell_0 (\omega   a_L)  a_{t_j-L}y =\\
			&=\epsilon_j n_0 \ell_0  (a_{-L} \omega) a_{t_j-L}y =\\
			&= a_{-L} (a_{L}\epsilon_j a_{-L}) (a_{L}n_0 a_{-L}) \ell_0  \omega a_{t_j-L}y.
		\end{align*}
		Since $ n_0 \in N $ we have for $ L $ large enough
		\[ \|(a_L n_0 a_{-L})\ell_0 \omega \| \leq \|\ell_0 \omega\|+1 = E, \]
		implying for all large $ j $
		\begin{equation}\label{eq:lemma_inf distance orb closure}
			d\left(a_{t_j+L}x,a_{t_j-L}y\right) \leq \|a_{L}\epsilon_j a_{-L}\|+\|(a_L n_0 a_{-L})\ell_0 \omega \| \leq E+1.
		\end{equation}
		Fix $ j $ so large so as to satisfy both \eqref{eq:lemma_inf distance orb closure} and 
		\[ d(a_{t_{j\!+\!1}}x,a_{t_{j\!+\!1}}y) \leq D+1. \]
		Now we have
		\begin{align*}
			d(a_{t_j-L}y,a_{t_{j\!+\!1}}y) &\leq d(a_{t_j-L}y,a_{t_j+L}x) + d(a_{t_j+L}x,a_{t_{j\!+\!1}}x)\\
			&\quad + d(a_{t_{j\!+\!1}}x,a_{t_{j\!+\!1}}y) \\
			&\leq  (E+1)+(t_{j\!+\!1}-(t_j+L))+(D+1) =\\
			&= t_{j\!+\!1}-t_j +D+D'+2 - L.
		\end{align*}
		On the other hand since $ y $ is quasi-minimizing, there exists a constant $ C>0 $ for which 
		\[ t_{j\!+\!1}-(t_j-L) - C \leq d(a_{t_j-L}y,a_{t_{j\!+\!1}}y). \]
		This in turn implies
		\[ 2L \leq D+E+2+C, \]
		where all constants $ C,D,E $ are independent of $ L $, in contradiction to the fact that $ L $ was arbitrary.
	\end{proof}
	
	Symmetry of the proximality relation immediately implies the following:
	\begin{corollary}\label{cor:proximality implies identical closures}
		If $ x, y \in \GdmodGamma' $ are proximal then $ \overline{Nx}=\overline{Ny} $.
	\end{corollary}
	
	Let us return to the setting of $ \GdmodGamma $ a $ \Z $-cover of a compact hyperbolic $ d $-manifold. We draw the following conclusion:

	\begin{cor}\label{Cor:Horospherical orbit closure intersect all Q of the same end}
		If $ \GdmodGamma $ is a $ \Z $-cover of a compact hyperbolic $ d $-manifold, for any $ x,y \in \Qm_+ $ there exists $ \ell \in MA_{\geq 0} $ with $ \ell y \in \overline{Nx} $. In other words, $ \overline{Nx} $ intersects every geodesic line in $ \Qm_+ $. Similarly for $ \Qm_- $.
	\end{cor}
	
	\noindent This proves \Cref{Thm intro: structure of Nx bar}(ii) from the introduction for arbitrary dimension.
	
	\begin{proof}
		Fix $ x,y \in \Qm_+ $ and denote $ b_x=\beta_+(x) $ and $ b_y =\beta_+(y) $. For all $ t $ we have
		\[ t+b_x \leq \tau(a_t x) \leq \tau(x)+t \quad \text{and}\quad t+b_y \leq \tau(a_t y) \leq \tau(y)+t,   \]		
		and therefore
		\[ |\tau(a_t x)-\tau(a_t y)| \leq C=\max \{|\tau(y)-b_x|,|\tau(x)-b_y|\} \quad \text{for all }t. \]
		By \cref{lemma:u_is_a_1C-qi} we deduce that
		\[ d(a_t x, a_t y) \leq C+C_\tau \quad \text{for all }t. \]
		The claim now follows from \cref{prop:fellow travelers in horo closure} and the sub-invariance of $ \overline{Nx} $ by the non-compact semigroup $ \Delta_x \subseteq MA_{\geq 0} $.
	\end{proof}
	
	Recall the notations from \Cref{Subsec:Tight map summary of notations} and denote $ \Qm_\omega^\pm=\Qm_\omega \cap \Qm_\pm $, the set of points in $ \Qm_\omega $ facing the positive/negative end of $ \GdmodGamma $. The following theorem tells us that it suffices to study the horospherical orbit closures of points in $ \Qm_\omega^\pm $:
	
	\begin{prop}\label{Prop:existence of proximal point in L}
		For any $ x \in \Qm_+ $ there exists $ x_0 \in \Qm_\omega^+ $ such that $ x $ and $ x_0 $ are proximal in $ \GdmodGamma $.
	\end{prop}
	
	\begin{proof}
		Consider the compact topological dynamical system $ \left(\GmodGammazero,(a_t)_{t \in \R}\right) $ which has $ p_\Z(\Qm_\omega^+) $ as a closed $ A $-invariant subset. \Cref{Theorem:omega limit of Qm is a lamination in max stretch locus of u} tells us that $ (a_t \pZ{x})_{t \geq 0} $ accumulates onto $ \pZ{\Qm_\omega^+} $, that is, \[ \lim_{t \to \infty} d(a_t \pZ{x},\pZ{\Qm_\omega^+}) = 0. \]
		
		A classical result in topological dynamics \cite[Proposition 8.6]{Furstenberg_Book} implies that there exists $ \pZ{x_0} \in \pZ{\Qm_\omega^+} $ such that $ \pZ{x} $ and $ \pZ{x_0} $ are proximal. 
		
		We claim that there exists a lift $ x_0 \in \Qm_\omega^+ \subset \GmodGamma $ of $ \pZ{x_0} $ for which $ x $ and $ x_0 $ are proximal. 
		Indeed, assume $ t_j \to \infty $ with $ d(a_{t_j}\pZ{x},a_{t_j}\pZ{x_0}) \to 0 $. As $ x \in \Qm_+=\beta_+^{-1}(\R) $ there exists $ T \geq 0 $ so large as to ensure
		\[ |[\tau(a_t x)-\tau(a_s x)] - [t-s]| \leq \frac{c}{4} \quad \text{for all }s,t \geq T, \]
		where $ c $ is the equivariance constant for the $ \Z $-action, i.e.~$ \tau(k.z)=\tau(z)+kc $ for all $ z \in \GdmodGamma $ and $ k \in \Z $. Without loss of generality we may assume that all $ t_j $ satisfy $ t_j \geq T $ and $ d(a_{t_j}\pZ{x},a_{t_j}\pZ{x_0}) < \frac{c}{4} $. 
		
		The injectivity radius of the $ \Z $-action at every point in $ \GdmodGamma $ is at least $ c $, since if $ y_1=k.y_2 $ for some $ 0\neq k \in \Z $ then
		\[ c \leq |k|c = |\tau(y_2) - \tau(y_1)| \leq d(y_1,y_2).  \]
		Now for any $ j\leq 1 $ let $ x_j \in \GdmodGamma $ be the unique lift of $ \pZ{x_0} $ satisfying
		\begin{equation}\label{eq:x_j sequence proximal}
			d(a_{t_j}x,a_{t_j}x_j)=d(a_{t_j}\pZ{x},a_{t_j}\pZ{x_0}) < \frac{c}{4}.
		\end{equation}		
		
		Fix some $ j \geq 1 $, we will show that $ x_j=x_1 $. 
		By our assumption that $ t_j \geq t_1 \geq T $ we know
		\[ 0 \leq \tau(a_{t_j} x)-\tau(a_{t_1} x) \leq t_j-t_1+\frac{c}{4}. \]
		On the other hand, since $ x_j \in \Qm_\omega^+ $ we know
		\[ \tau(a_{t_j} x_j)-\tau(a_{t_1} x_j)=t_j-t_1. \]
		Therefore
		\begin{align}
			|\tau(a_{t_1} x_j)-\tau(a_{t_1}x_1)| &= |[\tau(a_{t_1} x_j)-\tau(a_{t_j}x_j)]+ [\tau(a_{t_j} x_j)-\tau(a_{t_j}x)]+\nonumber\\
			&\quad +[\tau(a_{t_j} x)-\tau(a_{t_1}x)]+[\tau(a_{t_1} x)-\tau(a_{t_1}x_1)]|=\nonumber\\
			&\leq |[t_1-t_j]+[\tau(a_{t_j} x)-\tau(a_{t_1}x)]| +2\frac{c}{4} \label{eq:u-values of a_t_1 x_1 and a_t_1 x_j} \leq \\
			&\leq \frac{3c}{4} \nonumber
		\end{align}		
		
		By definition $ a_{t_1}x_j=k.a_{t_1}x_1 $ for some $ k \in \Z $, the inequality in \eqref{eq:u-values of a_t_1 x_1 and a_t_1 x_j} then implies $ k=0 $ and $ x_j=x_1 $ as claimed. 		
		Hence, by \eqref{eq:x_j sequence proximal} we conclude $ x_0:=x_1 \in \Qm_\omega^+ $ is proximal to $ x $.
	\end{proof}
	
	We further draw the following relation between different recurrence semigroups in $ \GdmodGamma $:
	\begin{prop}
		For all $ x,y \in \Qm $ there exists an isometry $ f_{x,y} $ of $ MA $ satisfying
		\[ f_{x,y}(\Delta_y) \subseteq \Delta_x.  \]
		Furthermore, $ f_{x,y} $ acts as a translation on the $ A $-component, that is, there exists an $ s \in \R $ for which $ f_{x,y}(ma_t)=m'a_{t+s} $ for all $ m \in M $ and $ t $.
	\end{prop}

	\begin{proof}
		Let us first consider $ x,y \in \Qm_+ $. By \cref{prop:fellow travelers in horo closure} there exists $ \ell_1,\ell_2 \in MA $ satisfying $ \ell_1 y \in \overline{Nx} $ and $ \ell_2 x \in \overline{Ny} $. Therefore
		\[ \ell_1 \overline{N y}=\overline{N \ell_1 y} \subseteq \overline{Nx}, \]
		and hence for any $ \ell_y \in \Delta_y $ 
		\[ \ell_1 \ell_y \overline{N y} = \ell_1 \overline{N \ell_y y} \subseteq \ell_1 \overline{N y} \subseteq \overline{Nx}. \]
		Now since 
		\[ \ell_2 \overline{N x} \subseteq \overline{Ny}, \]
		we obtain
		\[ \ell_1 \ell_y \ell_2 \overline{N x} \subseteq \ell_1 \ell_y \overline{N y} \subseteq \overline{N x}, \]
		and conclude
		\[ \ell_1 \Delta_y \ell_2 \subseteq \Delta_x. \]
		Clearly the map $ f_{x,y}(\ell)= \ell_1 \ell \ell_2 $ satisfies the conditions claimed in the proposition. The same holds when $ x,y \in \Qm_- $.
		
		Note that the relation on $ \Qm $ defined by $ x \sim y $ if there exists $ f_{x,y} $ as in the statement above, is an equivalence relation. We have shown that $ \Qm_+ \subset [x]_\sim $ whenever $ x \in \Qm_+ $, and similarly for $ \Qm_- $. To show $ [x]_\sim=\Qm $ as claimed, it suffices to find two points $ x_\pm \in \Qm_\pm $ with $ x_+ \sim x_- $. 
		
		Indeed, let $ x_+ \in \Qm_\omega^+ $ be a point in $ \Qm_\omega $ with $ \pZ{x_+} $ contained in a minimal $ A $-component of $ \L_0 $. Therefore, $ (a_t \pZ{x_+})_{t \geq 0} $ accumulates onto $ \pZ{x_+} $ implying that
		\[ \Delta_{x_+} = \overline{\delta (h_+ \Gamma h_+^{-1})} \]
		by \Cref{lemma_basepoint_conj_in_Xi_determines_Delta}, where $ x_+=h_+ \Gamma $. 
		
		Let $ \omega \in K $ be an element satisfying $ \omega a_t \omega^{-1}=a_{-t} $ for all $ t $. Set $ x_-=\omega x_+ $ and note that $ x_- \in \Qm_- $ as $ x_+ \in \Qm_\omega^+ $. Since $ (a_t x_-)_{t \geq 0} $ also accumulates onto $ x_- $ we have
		\[ \Delta_{x_-} = \overline{\delta(\omega h_+ \Gamma h_+^{-1}\omega^{-1})}. \]
		Given an element $ nma_t u \in NMAU $ we have
		\[ [\omega (n\ell u) \omega^{-1}]^{-1} = [u' m'a_{-t} n']^{-1} = n' m'a_t u'  \]
		where $ n' \in N, m' \in M, u' \in U $, since conjugation by $ \omega $ maps $ N \leftrightarrow U $ and normalizes $ M $.
		In particular, for any $ \gamma \in h_+ \Gamma h_+^{-1} $ we have
		\[ \delta (\omega \gamma^{-1} \omega^{-1} ) = f_\omega (\delta(\gamma)) \]
		where $ f_\omega(ma_t)=(\omega m \omega^{-1})a_t $ is an isometry of $ MA $ as claimed.
		We have thus shown that 
		\[ \Delta_{x_-}=\overline{\delta(\omega h_+ \Gamma h_+^{-1}\omega^{-1})}=\overline{f_\omega(\delta(h_+ \Gamma h_+^{-1}))}= f_\omega(\Delta_{x_+}), \]
		concluding the proof.
	\end{proof}
	
	\section{Weakly mixing laminations}\label{Sec: Weakly mixing}
	
	In this section we study the non-maximal horocycle orbit closures in surfaces with a weakly mixing first return map to a transversal to the maximal stretch lamination.
	
	Let $ \Sigma \to \Sigma_0 $ be any $ \Z $-cover as constructed in \Cref{Thm:Construction of surface with weak mixing IET} together with a 1-Lipschitz tight map $ \tau_0: \Sigma_0 \to \R/c\Z $ and $ \tau: \Sigma \to \R $ its lift. 
	
	Recall our definition of a ``uniform horoball'' from \Cref{eq:definition of betaplus horoball}:
	\[ \Horobetaplus{x} = \beta_+^{-1}\left( [\beta_+(x),\infty) \right) \quad \text{and} \quad \Horobetaminus{x} = \beta_-^{-1}\left( (-\infty, \beta_-(x)] \right). \]
	
	In this section we will prove \Cref{Thm intro:Nx bar in horoball} from the introduction:
	\begin{theorem}\label{Theorem:weakly mixing orbit closure is horobal}
		If $ \Sigma $ is a $ \Z $-cover surface constructed in \Cref{Thm:Construction of surface with weak mixing IET} from a weakly-mixing and minimal IET, then all non-maximal horocycle orbit closures in $ \Sigma $ are uniform horoballs. That is, for all $ x \in \Qm_\pm $
		\[ \overline{Nx}=\mathcal{H}_\pm(x). \]
	\end{theorem}

    Recall the definitions of weakly mixing dynamical systems from \Cref{Def:weakly mixing}.
	We have from the construction that the first return map (\Cref{Sec_construction of minimizing laminations})
	\[ P_0: \tau_0^{-1}([0]) \cap \lambda_0 \to \tau_0^{-1}([0]) \cap \lambda_0 \]
	is minimal and measure theoretically weakly mixing with respect to a measure of full support. 
    This implies in particular that $P$ is topologically weakly mixing.

	Lifting this to $ \L_+ = \L \cap \Qm_+ $ in the unit tangent bundle $ T^1\Sigma $ gives:
	\[ P: \tau^{-1}(0)\cap \L_+ \to \tau^{-1}(0)\cap \L_+ \]
	defined for all $ x \in \L_+ $ by
	\[ P(x)=(-1).a_c x \]
	where $ k.\square $ denotes the isometric action of the element $ k \in \Z $ of the deck group, and $ c $ is the circumference of $ \R / c\Z $, the codomain of $ \tau_0 $.

	Denoting $W_m = \tau^{-1}([m])\cap \L_+$, the systems
	\[ (\tau_0^{-1}([0]) \cap \lambda_0,P_0) \quad \text{and}\quad (W_0,P) \]
	are isomorphic as topological dynamical systems, implying $ (W_0,P) $ is weakly mixing and minimal.
	
	Recall that in a topological dynamical system $ (X,T) $ two points $ x_1,x_2 \in X $ are called proximal if 
	\begin{equation*}
		\liminf_{j \to \infty} d_X(T^j x_1,T^j x_2) = 0.
	\end{equation*}
	The following can be found in \cite[Theorem 1.13]{Glasner_Book}:
	\begin{theorem}\label{thm:Glasner dense proximality cell}
		If a continuous dynamical system $ (X,T) $ is minimal and topologically weakly mixing then for all $ x \in X $ the set
		\[ \cal P_x = \{y \in X : y \text{ is proximal to } x \} \]
		is dense in $ X $.
	\end{theorem}
	
	We may now prove the main result of this section:

	\begin{proof}[Proof of \Cref{Theorem:weakly mixing orbit closure is horobal}]
		Let us first consider a point $ x_0 \in W_0 = \tau^{-1}(0)\cap \L_+ $.
		Notice that if $ y_0 \in W_0 $ is $ P $-proximal to $ x_0 $ then it is also $ A $-proximal. That is, if
		\[ \liminf_{j \to \infty} d(P^j x_0,P^j y_0) = 0 \]
		then
		\[ \liminf_{t \to \infty} d(a_t x_0,a_t y_0) = 0. \]
		Indeed, since the $ \Z $-action and the $ A $-action on $ \GmodGamma $ commute we have
		\[ P^j x = (-j).a_{jc}x. \]
		The metric on $ W_0 $ is the one induced from $ \GmodGamma $, which is $ \Z $-invariant hence
		\[ d(P^j x_0,P^j y_0) < \varepsilon \Longrightarrow d(a_{jc} x_0,a_{jc} y_0) < \varepsilon \]
		and the implication follows.
		
		By \cref{thm:Glasner dense proximality cell}, the set of points $ y_0 \in W_0 $ that are $ P $-proximal, and hence also $ A $-proximal, to $ x_0 $ is dense in $ W_0 $. Therefore by \cref{cor:proximality implies identical closures} we get that $ \overline{Nx_0} $ contains all of $ W_0 $ and so for any $ y_0 \in W_0 $ we have $ \overline{Ny_0} \subseteq \overline{Nx_0} $. By symmetry we conclude
		\begin{equation}\label{eq:orbit closure constant on W_0}
			\overline{Nx_0} = \overline{Ny_0} \quad \text{for all }x_0,y_0 \in W_0.
		\end{equation}
		\medskip
		
		Now consider $ x \in \Qm_+ $. Note that 
		\[ \overline{Nx} = \Horobetaplus{x} \iff \forall s \in \R \quad \overline{Na_s x} = \Horobetaplus{a_s x} \]
		by property (3) of \cref{lemma_properties_beta}, hence without loss of generality we may replace $ x $ with $ a_{-\beta_+(x)} x $ and assume $ \beta_+(x)=0 $. 		
		
		If $ y \in \beta_+^{-1}(0) $ is any other point, we claim $ \overline{Ny}=\overline{Nx} $. By \Cref{Prop:existence of proximal point in L}, there exist $ x_0, y_0 \in \L_+ $ such that $ x $ and $ x_0 $ are proximal, and $ y $ and $ y_0 $ are proximal. By \cref{cor:proximality implies identical closures} we have
		\[ \overline{Nx}=\overline{Nx_0} \quad \text{and} \quad \overline{Ny}=\overline{Ny_0}. \]
		In particular, we deduce $ \beta_+(x_0)=\beta_+(x)=0 $ and similarly for $ y_0 $. Recall that on $ \L_+ $ we have $ \beta_+(x_0)=\tau(x_0) $. Therefore $ x_0,y_0 \in \tau^{-1}(0) $ and $ x_0,y_0 \in W_0 $. From \eqref{eq:orbit closure constant on W_0} we conclude
		\[ \overline{Nx} = \overline{Ny} \quad \text{for all }x,y \in \beta_+^{-1}(0). \]
		\medskip
		
		Showing that $ \Delta_{x_0}=A_{\geq 0} $ for some $ x_0 \in W_0 $ (and hence for all $ W_0 $ by \cref{cor:minimal components have constant Deltas}), implies
		\[ \overline{Nx_0} = \beta_+^{-1}([0,\infty)), \]
		concluding the proof. This remaining claim is the content of \Cref{Theorem:semigroup is everything} below.
	\end{proof}

\subsection{Full recurrence semigroup}
Using the notations above, we will show:
\begin{theorem}\label{Theorem:semigroup is everything}
	There exists $ x_0 \in W_0 $ for which
	\[ \Delta_{x_0} = A_{\geq 0}. \]
\end{theorem}

\noindent As discussed above, this theorem would imply $ \Delta_x=A_{\geq 0} $ for all $ x \in \Qm_+ $.

Fix $\ep$ less than half the minimal injectivity radius of $\Sigma$. If $x,y \in W_m$
are in a ball of radius $\ep$, we can define $\Delta(x,y)\in\R$ as follows: Lift the
$\ep$-ball containing them to $T^1\bH^2 = G$, and let $h,g\in G$ denote the lifts of $x,y$
respectively. Recall the Bruhat decomposition $hg^{-1} = n\ell u$ as in \Cref{subsec:bruhat} and \Cref{figure_bruhat}. (Note that
this is defined because $x$ and $y$ are either equal or on distinct leaves and hence $h^+
\ne g^-$).

In particular $\ell = \delta(hg^{-1})$ has the form $a_r$ and we define $\Delta(x,y)$ to be $r$.

Note that this definition does not depend on the choices, and  $\Delta$ is continuous as
$x,y$ vary in an $\ep$-ball. Moreover $\Delta $ is $\Z$-invariant, i.e. $\Delta(m.x,m.y) =
\Delta(x,y)$. 
Let $\mu$ denote the geodesic joining $g^-$ to $h^+$.

\begin{lemma}\label{Delta positive}
  With definitions as above $\Delta(x,y) \ge 0$, with equality if and only if the geodesic
  $\mu$ is  (a lift of) a leaf of $\L$. 
\end{lemma}

\begin{proof}
  Fix $z$ tangent to $\mu$ so that $z^+ = h^+$. Let $\omega\in K$ be the involution
  conjugating $a$ to $a^{-1}$ as in \Cref{subsec:bruhat}.   We claim that
  \begin{equation}\label{Delta beta}
\Delta(x,y) = \beta_-(\omega z) - \beta_+(z).
  \end{equation}
This will imply the inequality because $\beta_+ \le \tau$ and $\beta_-\circ \omega \ge
\tau$. Moreover, $\Delta(x,y)=0$ only occurs if  $\beta_+(z)=\tau(z)=\beta_-(\omega z)$, and this
corresponds to $z$ (and hence the whole leaf) being in $\L_+$.

\begin{figure}[htbp]
	\includegraphics{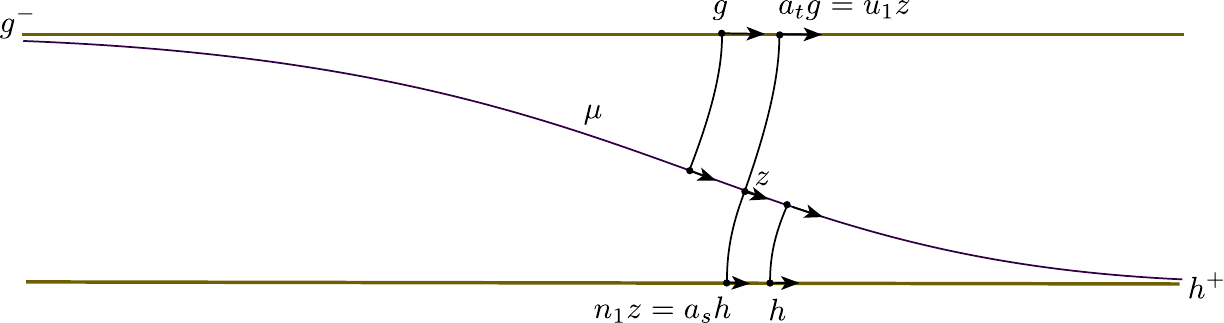}
	\caption{Where $z$ lands on the axes $Ag$ and $Ah$ determines the value of $\Delta$.}
	\label{fig:Delta-positive}
\end{figure}

It remains to prove \Cref{Delta beta}.
Move $z$ along a stable horocycle to the geodesic $Ax$, and along an unstable horocycle to
the geodesic $Ay$. In the lifted picture (see \Cref{fig:Delta-positive}) we have 
$$n_1 z = a_s h \quad \text{and} \quad u_1 z = a_t g$$
for some $n_1\in N$, $u_1\in U$. 
This implies
\begin{align*} h= a_{-s} n_1 z &= a_{-s} n_1 u_1^{-1} a_t g \\
  &= n' a_{t-s} u' g
\end{align*}
where $n'\in N, u'\in U$. In other words,
$$t-s = \Delta(x,y).
$$

Now, recall that $\beta_+$ is $N$-invariant and (since $\omega U \omega = N$) $\beta_-\circ \omega$ is $U$-invariant. Thus we
have
$$
\beta_+(z) = \beta_+(a_s h) \quad \text{and} \quad \beta_-(\omega z) = \beta_-(\omega a_t
g).
$$
Now on $\L_+$, $\beta_+ = \tau$ and $\beta_-\circ\omega = \tau$, so
$$\beta_+(z) = \tau(a_s h) = \tau(h)+s = s,$$
$$\beta_-(\omega z) = \tau(a_t g) = \tau(g)+t = t.$$
\Cref{Delta beta} follows. 
\end{proof}

\begin{proof}[Proof of \Cref{Theorem:semigroup is everything}]
Recall our notation $ W_0 = \tau^{-1}(0) \cap \L_+ $ and let $ d\lambda $ denote the
transverse measure on $ W_0 $ under which $ (W_0,P,d\lambda) $ is a weakly mixing system,
or equivalently, the system $ (W_0 \times W_0, P \times P, d\lambda \times d\lambda) $ is
ergodic. In particular, almost every $ (x,y) \in W_0 \times W_0 $ has a dense $P\times
P$-orbit. By Fubini there exists a point $ x_0 \in W_0 $ for which almost every $ y \in
W_0 
$ satisfies that $ (x_0,y) $ has a dense $ P \times P $-orbit in $ W_0 \times W_0 $. Fix
such a point $ x_0 $ and let $y_j$ be a sequence of points as above having $ y_j \to x_0
$. Note that every $ y_j $ is also $ A $-proximal to $ x_0 $ in $ \GmodGamma $ (because $
x_0 $ and $ y_j $ are $ P $-proximal). 

Now set $ C = B^{\GmodGamma}_r(x_0) \cap W_0 $, where $ r>0 $ is smaller than half the 
injectivity radius of $\Sigma$. 
Note that $\forall x,y\in C$, their corresponding geodesic $\mu$ is not contained in $\L$, unless $\mu$, $Ax$ and $Ay$ are boundary leaves. In our case $\L$ has no isolated leaves so this occurs only if $Ax$ and $Ay$ are asymptotic and $\mu$ is equal to one of them. Since non-boundary leaves are dense in $\L$, there exist $x,y\in C$ for which $\mu$ is not in $\L$. By Lemma \ref{Delta positive}
we are assured that  $\eta = \sup_{x,y \in C}(\Delta(x,y)) > 0$. 

By the continuity of $ \Delta $, choosing $r$ even smaller we can have $\eta$ as small as
we like. Then, again by continuity of $\Delta$
and by the density of the $P$-orbit of
$(x_0,y_j)$, we can find for any $ j $ a $q$ arbitrarily large so that $\Delta(P^q
x_0,P^q y_j)$ lies in $(\eta/2,\eta)$. Since  $ \Delta $ is $ \Z $-invariant, 
 we deduce that for every $ j $ there exist
arbitrarily large $ T = qc >0$ satisfying $ \Delta(a_T x_0,a_T y_j) \in (\eta/2,\eta)
$. Choose one such $T_j$ for 
each $j$, so that $T_j\to\infty$, and denote $s_j = \Delta(a_{T_j} x_0,a_{T_j} y_j)$ for
convenience.

\begin{figure}[htbp]
	\includegraphics{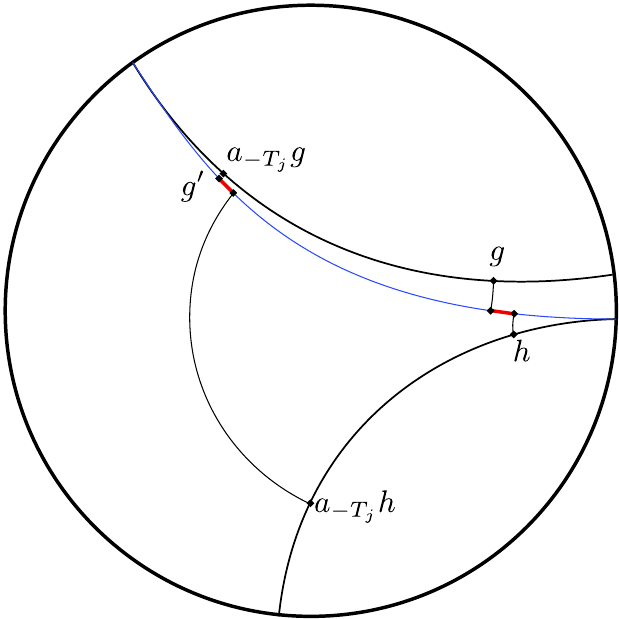}
	\caption{In this figure $h$ and $g$ project to $a_{T_j}x_0$ and $a_{T_j}y_j$ respectively. The point $g'$ projects to $y'_j$. The red segments, of equal length $s_j$, indicate the value of $\Delta$ and the intersection of $Nx_0$ on $Ay_j'$.}
	\label{fig:finish-semigroup}
\end{figure}

Lift $a_{T_j} x_0$ and $a_{T_j} y_j$ to $h$ and $g$ (as in the definition of $\Delta$ above),
and let $h = n \ell u g$ with $\ell = a_{s_j}$. Note that $a_{-T_j} h$ projects to $x_0$,
and $a_{-T_j}g$ projects to $y_j$ (See \Cref{fig:finish-semigroup}).
Let
$$g' = a_{-T_j} ug.$$
That is, starting at $g$ move along an unstable horocycle to the cross-over geodesic $\mu$
and then back up by $T_j$ along $\mu$. This is a point that projects to $y'_j$ which is
very close to $y_j$. Indeed,
$$g' = a_{-T_j} u a_{T_j} (a_{-T_j} g)$$
and $|| a_{-T_j} u a_{T_j}|| = O(e^{-T_j})$, since $u$ is uniformly bounded by the fact
that $a_{T_j}y_j$ and $a_{T_j}x_0$ are in a translate of the small ball $C$. In particular
we have
\begin{equation}\label{yjprime converge}
d(y_j,y'_j) = O(e^{-T_j}).
\end{equation}

On the other hand,
\begin{align*} g' &= a_{-T_j} \ell^{-1} n^{-1} h \\
  & = \hat n a_{-s_j} (a_{-T_j} h)
\end{align*}
for some $\hat n \in N$. 
Thus we have $y'_j\in N a_{-s_j} x_0,$ or equivalently
\begin{equation}\label{ayprime in N}
  a_{s_j} y'_j \in Nx_0.
\end{equation}
Assume (taking a subsequence) that $s_j \to s\in [\eta/2,\eta]$. Since $y_j\to x_0$,
\Cref{yjprime converge} implies
$a_{s_j} y'_j \to a_s x_0$, so
$$
a_s x_0 \in \overline{Nx_0}.
$$
Since $\eta$ was arbitrarily small, we conclude that $\Delta_{x_0}$ contains arbitrarily
small non-zero elements, so $\Delta_{x_0}$ is all of $A_{\ge 0}$. 
\end{proof}

\subsection{Non-rigidity of horocycle orbit closures}
We may now conclude the non-rigidity result stated in the introduction.
\begin{proof}[Proof of \Cref{Theorem:non-rigidity}]
	By \Cref{subsec:geometric convergence}, given a $ \Z $-cover $ \Sigma $ constructed from a weakly-mixing IET there exist arbitrarily small deformations $ \Sigma' $ of $ \Sigma $ having a bi-minimizing locus $ \lambda' $ consisting of finitely many uniformly isolated geodesics. Since quasi-minimizing points tend to the bi-minimizing locus, we conclude that all quasi-minimizing points in $ T^1\Sigma' $ are asymptotic to one of the geodesics in $ \lambda' $. In particular, such a quasi-minimizing point has the same orbit closure as some point on $ T^1\lambda' $. By \Cref{cor:isolated leaf implies isolated e in Delta}, all horocycles based at $ T^1\lambda' $ do not accumulate onto their base. 
	
	The above discussion implies in particular, that for any non-maximal $ N $-orbit closure $ \overline{Nx} $ in $ T^1\Sigma' $, there exists a ball $ B $ such that  
	\begin{equation}\label{eq:isolated horo segment}
	\overline{Nx} \cap B \quad \text{is equal to a single horocyclic segment},
	\end{equation}
	this also holds after projecting onto $ \Sigma' $.
	
	On the other hand, in $ \Sigma $ all $ N $-orbit closures are $ NA_{\geq 0} $-sub-invariant implying in particular that no ball $ B $ satisfies \eqref{eq:isolated horo segment}. We therefore conclude that no two non-maximal horocycle orbit closures in $ \Sigma $ and in $ \Sigma' $ are homeomorphic.
\end{proof}

\providecommand{\bysame}{\leavevmode\hbox to3em{\hrulefill}\thinspace}
\providecommand{\MR}{\relax\ifhmode\unskip\space\fi MR }
\providecommand{\MRhref}[2]{%
  \href{http://www.ams.org/mathscinet-getitem?mr=#1}{#2}
}
\providecommand{\href}[2]{#2}

\bibliographystyle{amsalpha.bst}

\end{document}